%% file: buhovski-opshtein-C0_2015-09-07.tex
\documentclass[11pt]{article}
\usepackage{latexsym,amsmath,amsthm,amssymb,amscd,amsfonts}
\usepackage{epsfig}
\usepackage{nicefrac}
\usepackage[all]{xy}
\usepackage{graphicx}
\usepackage{ulem}

\usepackage[usenames,dvipsnames]{xcolor}

\newtheorem{lemma}{Lemma}[section]
\newtheorem{proposition}[lemma]{Proposition}
\newtheorem{remark}[lemma]{Remark}
\newtheorem{theorem}[lemma]{Theorem}
\newtheorem{definition}[lemma]{Definition}
\newtheorem{claim}[lemma]{Claim}
\newtheorem{corollary}[lemma]{Corollary}

\newtheorem{question}{Question}
\newcounter{varquest}\newtheorem{question'}[varquest]{Question}
\newtheorem{thm}{Theorem}

\newtheorem{defn}[thm]{Definition}

\newtheorem{conj}{Conjecture}
\newtheorem{prop}[thm]{Proposition}

\newtheorem*{defn*}{Definition}

\newtheorem*{remark*}{Remark}

\newtheorem*{thm*}{Theorem}
\newtheorem*{thm'}{Theorem'}
\newtheorem*{exs*}{Examples}
\newtheorem*{ex*}{Example}

\newtheorem*{quesA'}{Question A'}
\newtheorem*{quesD'}{Question D'}

\newcommand{\R}{\mathbb{R}}
\newcommand{\C}{\mathbb{C}}
\newcommand{\D}{\mathbb{D}}

\newcommand{\Q}{\mathbb{Q}}
\newcommand{\T}{\mathbb{T}}

\newcommand{\cc}{\mathbb{\mathcal C}}

\newcommand{\ca}{\mathcal A}
\newcommand{\ce}{\mathcal E}
\newcommand{\cd}{\mathcal D}
\newcommand{\cv}{\mathcal V}
\newcommand{\cu}{\mathcal U}

\newcommand{\cf}{\mathcal{F}}

\newcommand{\id}{\textnormal{Id}\,}

\newcommand{\nbd}{neighbourhood }
\newcommand{\nbds}{neighbourhoods }
\newcommand{\fonction}[5]
{$$ 
\begin{array}{rcccl}
 #1 & : & #2 & \longrightarrow &#3 \\
    &   & #4 & \longmapsto &#5 
\end{array}
$$}

\newcommand{\Graph}{\textnormal{Graph}\,}

\newcommand{\red}{\textnormal{Red}\,}

\newcommand{\priv}{\backslash}
\newcommand{\lra}{\longrightarrow}
\newcommand{\hra}{\hookrightarrow}

\newcommand{\om}{\omega}
\newcommand{\eps}{\varepsilon}
\renewcommand{\phi}{\varphi}

\newcommand{\st}{\textnormal{st}}

\newcommand{\diam}{\text{diam}\,}
\newcommand{\supp}{\text{Supp}\,}

\newcommand{\cqfd}{\hfill $\square$ \vspace{0.1cm}\\ }
\newcommand{\sbull}{{\tiny $\bullet$ }}

\newcommand{\ds}{\displaystyle}
\newcommand{\im}{\textnormal{Im}\,}

\newcommand{\can}{\textnormal{can}}

\newcommand{\nf}[2]{{\nicefrac{#1}{#2}}}

\newcommand{\spec}{\textnormal{Spec}\,}

\newcommand{\osc}{\textnormal{osc}\,}
\newcommand{\lag}{\textnormal{Lag}}
\newcommand{\hz}{\textnormal{HZ}}

\newcommand{\sympeo}{\textnormal{Sympeo}\,}

\newcommand{\size}{\textnormal{size}\,}

\def\eps{\varepsilon}

\def\Ham{\text{Ham}}

\makeatletter \@addtoreset {equation}{section}

\renewcommand\theequation
  {\ifnum \c@subsection>\z@ \arabic{section}.\arabic{subsection}.\arabic{equation}
  \else \arabic{section}.\arabic{equation} \fi}
\makeatother

\newcounter{veq}
\renewcommand*{\theveq}{$\cv$\arabic{veq}}
\newcommand{\numv}{\refstepcounter{veq}\tag{\theveq}}

\newcounter{aeq}
\renewcommand*{\theaeq}{$\ca$\arabic{aeq}}
\newcommand{\numa}{\refstepcounter{aeq}\tag{\theaeq}}

\newcounter{eeq}
\renewcommand*{\theeeq}{$\ce$\arabic{eeq}}
\newcommand{\nume}{\refstepcounter{eeq}\tag{\theeeq}}

\newcounter{feq}
\renewcommand*{\thefeq}{$\cf$\arabic{feq}}
\newcommand{\numf}{\refstepcounter{feq}\tag{\thefeq}}

\begin{document}


\title{\vspace*{-0,5cm}Some quantitative results in $\cc^0$ symplectic geometry.}

\author{Lev Buhovsky$^{1}$, Emmanuel Opshtein}

\footnotetext[1]{This author also uses the spelling ``Buhovski"
for his family name.}

\date{\today}
\maketitle

\begin{abstract}
This paper proceeds with the study of the $\cc^0$-symplectic geometry of smooth submanifolds, as initiated in \cite{opshtein,hulese}, with the main focus on the behaviour of symplectic homeomorphisms with respect to numerical invariants like capacities. Our main result is that a symplectic homeomorphism may preserve and squeeze codimension $4$ 
symplectic submanifolds ($\cc^0$-flexibility),  while this is impossible for codimension $2$ symplectic submanifolds ($\cc^0$-rigidity).  We also discuss $\cc^0$-invariants of coistropic and Lagrangian submanifolds, proving some rigidity results and formulating some conjectures. We finally formulate an Eliashberg-Gromov $\cc^0$-rigidity type question for submanifolds, which we solve in many cases. Our main technical tool is a quantitative $h$-principle result in symplectic geometry. 
\end{abstract}


\section{Introduction}
The starting point of this paper is the celebrated $\cc^0$-rigidity theorem by Eliashberg-Gromov:
\begin{thm*}
Any diffeomorphism which is a $\cc^0$-limit of symplectic diffeomorphisms is symplectic.
\end{thm*}
A consequence of this result is that symplectic homeomorphisms, defined as homeomorphisms which are $\cc^0$-limits of symplectic diffeomorphisms, define a proper subset of the volume-preserving homeomorphisms, which deserves attention. A natural question in the field can be stated as follows: {\it can a symplectic homeomorphism do something that a symplectic diffeomorphism cannot?} There is however a rather trivial answer to this question: it can take some smooth object to a singular one. In order to focus on the symplectic side of the problem, we therefore restrict the question.

\begin{question}[$\cc^0$-rigidity]\label{q:c0rigid} Assume that the image of a smooth submanifold $X$ by  a symplectic homeomorphism $h$ is smooth. Is there a symplectic diffeomorphism $f$ such that $f(X)=h(X)$?
\end{question}
Although very natural, this question seems at the moment out of reach, except in some very special situations. This paper is mostly concerned with two weaker versions of question \ref{q:c0rigid} for which we can provide more satisfactory answers. The first one concerns the preservation of some quantitative symplectic invariants. 
\begin{question}\label{q:c0quantrig}
Assume that the image of a smooth symplectic submanifold $X$ by  a symplectic homeomorphism $h$ is 
smooth and symplectic. Does the restriction of $h$ to $X$ preserve some capacities?  
\end{question}
The second one is a relative version of Eliashberg-Gromov's theorem.
\begin{question} \label{q:Eli-Grom-submanifolds}
Assume that the restriction of a symplectic homeomorphism $h : M \rightarrow M' $ to a smooth submanifold $ X $ is a diffeomorphism onto its image $ X' = h(X) $ (which is itself a smooth submanifold). Is it then true that the restriction $ h_{|X} : X \rightarrow X' $ is symplectic (i.e. $h_{|X}^*\om_{|X'}'=\om_{|X}$, where $ \omega $ and $ \omega' $ are the symplectic structures on $ M $ and $ M' $ respectively)?
\end{question}

An important related result in the field is due to Humili\`ere-Leclercq-Seyfaddini:
\begin{thm*}[\cite{hulese}, see also \cite{lasi} and \cite{opshtein} for the Lagrangian and hypersurface case respectively]\label{thm:HLS} If $C$ is a coisotropic submanifold of $M$ and $h(C)$ is smooth (where $ h: M \rightarrow M' $ is a symplectic homeomorphism), then $h(C)$ is coisotropic, and $h$ takes the characteristic foliation of $C$ to that of $h(C)$. 
\end{thm*}
In fact, \cite{lasi} proves a stronger result for closed Lagrangian submanifolds, under some additional assumptions. It is however not hard to see that their proof implies the previous theorem, in the setting of closed Lagrangian submanifolds. In view of question \ref{q:c0rigid}, this theorem is an important step, since the local symplectic invariants of coisotropic submanifolds are only two: characteristic foliation and transverse symplectic structure.

\subsection{Main results}
 
Our main result is the following flexibility statement:

\begin{thm}\label{thm:flexcodim4}
For $ n \geqslant m+2 $, there exists a symplectic homeomorphism $h : \mathbb{C}^n \rightarrow \mathbb{C}^n $ with support in an arbitrary \nbd of $Q:=\D(1)^m\times\{0_{n-m}\}\subset \C^n$, such that $h_{|Q}=\frac 12 \id$ (here $0_{n-m}$ stands for $0\in \C^{n-m}$).
\end{thm}

By standard \nbd theorems, this statement holds true when $Q$ is the image of a symplectic embedding of a polydisc $\D(a)^m$ of codimension at least $4$ in any symplectic manifold. Theorem \ref{thm:flexcodim4} implies a negative answer to questions \ref{q:c0rigid}, \ref{q:c0quantrig}, \ref{q:Eli-Grom-submanifolds} for symplectic polydiscs of codimension greater than $2$. In particular, $\cc^0$-symplectic geometry is not a trivial generalization of classical symplectic geometry: not all classical invariants are $\cc^0$-rigid. Theorem \ref{thm:flexcodim4} follows from the following quantitative $h$-principle result:

\begin{thm}[Quantitative $h$-principle for symplectic discs] \label{T:qh-principle-main}
Let $ \epsilon > 0 $ be a positive real, $ n \geqslant 3 $ be an integer, $ W \subset \mathbb{C}^n $ be an open set, and let $ u_1,u_2 : \overline{D} \rightarrow W $ be symplectic discs, with $ u_1^* \omega_{\st} = u_2^* \omega_{\st} = \omega_{\st} $. Assume that $ u_1 $ can be homotoped to $ u_2 $ in $ W $, such that in this homotopy, the trajectory of any point on the disc $ u_1 $ has diameter less than $ \epsilon $. Then there exists a smooth Hamiltonian flow $ \phi $ supported in $ W $, whose trajectories have diameter less than $2\eps$, and whose time-$1$ map sends $ u_1 $ to $ u_2 $. 
\end{thm} 

The novelty of this quantitative $h$-principle result with respect to the classical ones \cite{gromov2,elmi} is that it ensures the $\cc^0$-smallness of the Hamiltonian flow which brings one disc to the other, and not only the $\cc^0$-smallness of the isotopy between the two discs. This point is crucial for applications to $\cc^0$-symplectic geometry.

Next, we show that the flexibility statement provided by theorem \ref{thm:flexcodim4} is optimal in the dimension: rigidity holds for codimension $2$ symplectic submanifolds. 
\begin{thm}\label{thm:nssymp}
Let $h:M \to M'$ be a symplectic homeomorphism between symplectic manifolds of dimension $2n$, that takes a codimension $2$ symplectic submanifold $N$ to a symplectic submanifold $N'$. Then the restriction of $h$ to $N$ verifies the non-squeezing property. Namely, if $B(a)\subset N$ is a symplectic ball of size $a$, and $h(B(a))\subset N'$ can be symplectically embedded into the cylinder $Z(A):=\D(A)\times \C^{n-2}$, then $A\geqslant a$. 
\end{thm}
This non-squeezing type theorem is a form of answer to question \ref{q:c0quantrig}. In terms of capacities, it can be expressed as an inequality between the 
Gromov's width of a subset 
and the cylindrical capacity of its image: 
$$
c^Z(h(U))\geqslant c_B(U) \hspace*{1cm} \forall \text{ open subset } U\subset N.
$$ 
Recall that given a $ 2n $-dimensional symplectic manifold $ (M,\omega) $, its Gromov width and cylindrical capacity are
\begin{align*} 
& c_B(M) = \sup \{ a \, | \, B^{2n}(a) \text{ symplectically embeds in } (M,\omega) \},\\
&c^Z(M) = \inf \{ a \, | \, (M,\omega) \text{ symplectically embeds into } \D(a) \times \mathbb{C}^{n-1} \}.
\end{align*} 

\begin{remark}
In view of \cite{hulese}, it is legitimate to ask whether the assumption that $N'$ is also symplectic is automatically satisfied as soon as $N'$ is smooth. This is however not the case, see corollary \ref{cor:symptononsymp}. 
\end{remark}

The rigidity expressed by theorem \ref{thm:nssymp} implies a corresponding rigidity for hypersurfaces. Recall that by \cite{hulese} (or \cite{opshtein} in this particular setting), when a symplectic homeomorphism $h:M\to M'$ takes a  smooth hypersurface $\Sigma^{2n-1}$ to a smooth hypersurface $\Sigma'$, it also sends the characteristic foliation of $\Sigma$ to that of $\Sigma'$. It therefore induces a map at the level of the reduction $\hat h:\red (\Sigma)\to \red (\Sigma')$, called the reduction of $h$. The next theorem asserts a $\cc^0$-rigidity for the elements of the reduction $ \red(\Sigma)$ (we refer the reader to paragraph \ref{sec:defred} for precise definitions). 

\begin{thm}\label{thm:nshyp}
Let $h:M \to M'$ be a symplectic homeomorphism between symplectic manifolds of dimension $2n$, that takes a smooth hypersurface $\Sigma$ to a smooth hypersurface $\Sigma'$. Then the reduction of $h$ verifies the non-squeezing theorem. Namely, if $[B^{2n-2}(a)]\in \red (\Sigma)$, and $\hat h([B^{2n-2}(a)])\in\red(\Sigma')$ can be symplectically embedded into $Z(A)=\D(A)\times \C^{n-2}$, then $a\leqslant A$. 
\end{thm}

\begin{remark}
In dimension $ 2n = 4 $, the rigidity of the reduction for hypersurfaces implies that $\hat h$ is an area-preserving homeomorphism: for any $ [U] \in \red(\Sigma) $ we have $\ca_{\om'}(\hat h(U))=\ca_\om(U)$. Arguing as in \cite{opshtein}, we can deduce that the boundaries of rational ellipsoids in dimension $4$ are rigid in the strongest sense (question \ref{q:c0rigid}): if a smooth hypersurface $\Sigma$ in a symplectic $4$-manifold is symplectic homeomorphic to  $\partial E(a,b)$ with $\nf ab\in \Q$, then it is also symplectic diffeomorphic to $\partial E(a,b)$. The argument makes however a crucial use of the rationality of $\nf ab$ to avoid any dynamical complication. It would be interesting to know whether this rigidity holds in richer dynamical contexts, like for instance irrational ellipsoids. 
\end{remark}

In fact we conjecture that theorem \ref{thm:nshyp} holds for general coisotropic submanifolds. 
\begin{conj}\label{conj:rigred} Let $h$ be a symplectic homeomorphism which sends a coisotropic submanifold to a smooth (hence coisotropic) submanifold, and $\hat h$ its reduction. If $[B^{2n-2k}(a)] \in \red (C)$, and $\hat h([B^{2n-2k}(a)])\overset{\om}{\hra} Z(A)$, then $A\geqslant a$. (Here $ k $ is the codimension of the coisotropic submanifold, and $ 2n $ is the dimension of the ambient symplectic manifold).
\end{conj}

Our final ``quantitative" result concerns Lagrangian submanifolds. Although their area vanish, these objects have symplectic sizes, measured by their 
area homeomorphism: 
\fonction{\ca_\om^L}{H_2(M,L)}{\R}{A}{\ca_\om(\Sigma)}
where $\Sigma$ is any smooth surface representing $ A $. \label{def:areahom}
Recall that by the theorem already quoted by Laudenbach-Sikorav, a smooth image of a closed Lagrangian by a symplectic homeomorphism is Lagrangian \cite{lasi}. We conjecture that the area homomorphism of a closed Lagrangian is $\cc^0$-rigid (see conjecture \ref{conj:lagr-spec-strong}, p. \pageref{conj:lagr-spec-strong} for a precise statement). We prove it here for tori. 
 \begin{thm}\label{thm:rigspec}
 If  $L\subset M^{2n}$ is a Lagrangian torus, and if 
  $L'=h(L)$ is smooth for some symplectic homeomorphism $h$, then 
$h^*\ca_{\om'}^{L'}=\ca_\om^L$.  In other words, for any class $A\in H_2(M,L)$, we have
$\ca_\om^L(A)=\ca^{L'}_{\om'}(h_*A)$.
 \end{thm}

\subsection{Variations on Eliashberg-Gromov's theorem}\label{sec:preseligr}
We now turn to question \ref{q:Eli-Grom-submanifolds}. For symplectic submanifolds, theorem \ref{thm:flexcodim4} explicitely gives a negative answer to question \ref{q:Eli-Grom-submanifolds} in codimension at least $4$. On the other hand, it is well known that Eliashberg-Gromov's theorem follows from the non-squeezing theorem \cite{hoze}. Analogously, theorem \ref{thm:nssymp} answers question \ref{q:Eli-Grom-submanifolds} by the affirmative for symplectic submanifolds of codimension $2$.

\begin{prop} \label{prop:eli-grom-symp-codim2}
Let $ N^{2n-2} \subset M $ be a symplectic submanifold of codimension $ 2 $, and let $ h : M \rightarrow M' $ be a symplectic homeomorphism whose restriction to $ N $ is a diffeomorphism onto a smooth submanifold $ N' = h(N) $. Then $ N' $ is a symplectic submanifold of $ M' $, and the restriction $ h_{|N} : N \rightarrow N' $ is a symplectomorphism.
\end{prop} 

Similarly, theorem \ref{thm:nshyp} provides a positive answer to question \ref{q:Eli-Grom-submanifolds} for hypersurfaces. Conjecture \ref{conj:rigred} would also imply the same for coisotropic submanifold. We provide however another argument for general coisotropic submanifolds that only relies on \cite{hulese}.

\begin{prop} \label{prop:eli-grom-coisotrop}
Let $ C $ be a coisotropic submanifold of $ M $, and let $ h:M\to M' $ be a symplectic homeomorphism  whose restriction to $ C $ is a diffeomorphism onto a smooth submanifold $ C' = h(C) $. Then the restriction $ h_{|C} : C \rightarrow C' $ is symplectic.
\end{prop}

Let us remark, that propositions \ref{prop:eli-grom-symp-codim2}, \ref{prop:eli-grom-coisotrop} imply that for {\it any} codimension $ 2 $ smooth submanifold $ N^{2n-2} \subset M $, and a symplectic homeomorphism $ h : M \rightarrow M' $ whose restriction to $ N $ is a diffeomorphism onto a smooth submanifold $ N' = h(N) $, we have that the restriction $ h_{|N} : N \rightarrow N' $ is a symplectomorphism. Indeed, in such a situation, for any point $ x \in N $, the tangent space $ T_x N $ is either symplectic or coisotropic. If we denote by $ V \subset N $ the set of points $ x \in N $ such that $ T_x N $ is symplectic, then $ V $ is a relatively open subset of $ N $, and hence by proposition \ref{prop:eli-grom-symp-codim2}, the restriction $ h_{|V} : V \rightarrow h(V) $ is symplectic. Now, if we denote by $ W \subset N $ the relative interior of $ N \setminus V $ in $ N $, then $ W $ is relatively open in $ N $ and moreover $ W $ is a coisotropic submanifold of $ M $. Hence by proposition \ref{prop:eli-grom-coisotrop}, the restriction $ h_{|W} : V \rightarrow h(W) $ is symplectic. But since the union of $ W $ and the relative closure of $ V $ in $ N $, equals the whole $ N $, by continuity we conclude that the restriction $ h_{|N} : N \rightarrow N' = h(N) $ is symplectic.

In fact, propositions \ref{prop:eli-grom-symp-codim2}, \ref{prop:eli-grom-coisotrop} and theorem \ref{thm:flexcodim4} put together  answer question \ref{q:Eli-Grom-submanifolds} for many  presymplectic submanifolds.

\begin{definition}
Let $ (M^{2n},\omega) $ be a symplectic manifold. A smooth submanifold $ X \subset M $ is presymplectic if all the spaces $ T_x X $, $ x \in X $, are symplectically isomorphic (i.e. that for any $ x,y \in X $, the restrictions of $ \omega $ to $ T_x X $ and to $ T_y X $ are isomorphic via a linear isomorphism $ T_x X \rightarrow T_y X $). This is equivalent to saying that the dimension of the kernel of $ \omega $ in $ T_x X $ does not depend on $ x \in X $. In that case we call this dimension the symplectic co-rank of $ X $.
\end{definition} 
 
As in the case of a coisotropic submanifold, presymplectic submanifolds carry  characteristic foliations (tangent to the distribution $ T_x X \cap (T_x X)^{\perp_{\omega}} $), have local symplectic models of the form $\D^p\times[0,1]^r$, where $r$ is the co-rank of $X$ and $k=2p+r$ is its dimension.  
Let us remark that considering only presymplectic submanifolds is still general enough, since 
every smooth submanifold of a symplectic manifold has a dense relatively open subset which is a union of presymplectic relatively open subsets (see lemma~\ref{L:sympl-homogen-generic}). Moreover, for a generic smooth submanifold, the complement of this dense relatively open subset is rather ``small". 

Here are all possible occurences for the pairs $(k,r)$  and the answers to question \ref{q:Eli-Grom-submanifolds}. Notice that the   case $k=1$ is trivial and not considered below. Moreover, $k+r$ must be even, and not more than $2n$.
\begin{description}
\item{$k+r=2n$:} This is the coisotropic case, and the answer to question \ref{q:Eli-Grom-submanifolds}  is affirmative by proposition \ref{prop:eli-grom-coisotrop}.  
\item{$k+r\leqslant 2n-4$, $k\neq r$:} One checks that  $X$  lies in a codimension $4$ symplectic submanifold, is not isotropic, so theorem \ref{thm:flexcodim4} provides a negative answer to question \ref{q:Eli-Grom-submanifolds}.
\item{$k+r=2n-2 $, $ r\geqslant 2$ or $k+r\leqslant 2n-4$, $k=r$ (isotropic case):}  Our work does not show anything in this situation but the following conjecture  would give a negative answer to question \ref{q:Eli-Grom-submanifolds} in all these cases, which comprise in particular the isotropic non-Lagrangian $X$. 
\begin{conj}\label{conj:flex-symp-iso} There exists a symplectic homeomorphism with compact support in $\C^3$ that diffeomorphically maps a symplectic disc to a smooth isotropic disc.  
\end{conj}

\item{$k+r=2n-2$ and $r=0$:} This is the case of symplectic submanifolds of codimension $2$, and the answer to question  \ref{q:Eli-Grom-submanifolds}  is affirmative by proposition \ref{prop:eli-grom-symp-codim2}.

\item{$k+r=2n-2$, $r=1$:} This is the last remaining case, and it is quite mysterious for us since we do not even  know what to conjecture here - rigidity or flexibility.

\end{description}

\subsection{Further questions}

Let us mention a stronger, more fundamental version of questions \ref{q:c0rigid} and \ref{q:c0quantrig}: 

\begin{question}\label{q:subsympeo}
Assume that the image of a smooth submanifold $X$ by a symplectic homeomorphism $h$ is smooth. Does there exist a sequence of symplectic diffeomorphisms $f_k$, $ k =1,2, \ldots $, which $ \mathcal{C}^0 $-converges to $ h $, such that $f_k(X)=h(X)$ for each $ k $?
\end{question}
This question seems for the moment to be completely out of reach, and it might require novel ideas. 
Questions \ref{q:c0rigid}, \ref{q:subsympeo} 
admit the following weaker versions:

\begin{question'}\label{q:c0rigid'}
Assume that the image of a smooth submanifold $X$ by  a symplectic homeomorphism $h$ is smooth. Is there a diffeomorphism $f : X \rightarrow h(X) $ which is symplectic (i.e. which maps the symplectic form on the source to the symplectic form on the target)?
\end{question'}
\setcounter{varquest}{3}
\begin{question'}\label{q:subsympeo'}
Assume that the image of a smooth submanifold $X$ by  a symplectic homeomorphism $h$ is smooth. 
Does there exists a sequence of diffeomorphisms $f_k : X \rightarrow h(X) $, $ k =1,2,\ldots $, which $ \mathcal{C}^0 $-converges to $ h_{|X} : X \rightarrow h(X) $, such that $f_k$ is symplectic for each $ k $?
\end{question'} 
It is not hard to see that for a general submanifold $ X \subset M $ we get an affirmative answer to question \ref{q:Eli-Grom-submanifolds}, provided that we have a positive answer to question \ref{q:c0rigid'} for stabilisations of relatively open subsets of $ X $ (see lemma~\ref{L:quesA-implies-quesB}), or a positive answer to question \ref{q:subsympeo'} for $ X $ (see lemma~\ref{L:quesD'-implies-quesB}).

\subsection{Organization of the paper}  
Section \ref{sec:c0flexcodim4} is devoted to the flexibility of codimension $4$ symplectic submanifolds. We prove theorem \ref{thm:flexcodim4}, using the quantitative $h$-principle for discs which is formulated in theorem~\ref{T:qh-principle-main}, and several classical $h$-principle statements in symplectic geometry \cite{gromov2,elmi}. Theorem~\ref{T:qh-principle-main} is proved right after, and the proofs of the classical $h$-principle statements that we use are given in appendix \ref{sec:append-flex}. 
We establish theorem \ref{thm:nssymp} in section \ref{sec:rigcodim2}, and apply it to the $\cc^0$-rigidity of the reduction of symplectic hypersurfaces in section \ref{sec:rigred}. 
In section \ref{sec:rigspec} we prove theorem \ref{thm:rigspec} on the rigidity of the spectrum of tori, and state a precise conjecture for general Lagrangians.
Finally, we prove propositions \ref{prop:eli-grom-symp-codim2}, \ref{prop:eli-grom-coisotrop} (\`a la Eliashberg-Gromov) in section \ref{sec:elgr}.
In appendix \ref{sec:relquest}, we explain some relations between the different questions stated in the introduction.

\subsection{Notations and conventions}\label{notation}

 Let us first give a precise definition of symplectic homeomorphisms.
 \begin{definition} \label{def:sympeo} Let $M,M'$ be two symplectic manifolds. We say that a homeomorphism $h:M\to M'$ is symplectic if there exists a sequence of open subsets $U_1\subset U_2\subset \dots\subset M$ which exhaust $M$, and a sequence of symplectic embeddings $h_i:U_i\to M'$ which converge uniformly to $h$ on compact subsets of $M$. We denote by $\sympeo(M,M')$ the set of symplectic homeomorphisms between $M$ and $M'$.
 \end{definition}
 Although this definition does not concern {\it a priori} the inverse of the homeomorphism, we have the following classical proposition:
 \begin{proposition}\label{prop:defsympeo}
Symplectic homeomorphisms can be restricted to open subsets and inverted. Precisely, if $h\in \sympeo(M,M')$, then $h^{-1}\in \sympeo(M',M)$ and for an open subset $U\subset M$, $h_{|U}\in \sympeo(U,h(U))$. In addition, for any $ h \in \sympeo(M,M') $ and $ h' \in \sympeo(M',M'') $ we have $ h' \circ h \in \sympeo(M,M'') $.
\end{proposition}
\noindent {\it Proof:} Let $h\in \sympeo(M,M')$, $U_1 \subset U_2 \subset \ldots $ an exhaustion of $M$, and $h_i:U_i\to M'$ be symplectic embeddings which $\cc^0$-converge on compact sets to $h$. Let $ U \subset M $ be an open subset. To show that the restriction $h_{|U} : U \rightarrow h(U) $ is a symplectic homeomorphism, look at the sequence of open sets $ \displaystyle V_i = \text{int} ( \bigcap\limits_{k=i}^{\infty} h_k^{-1}(h(U)) ) $. It is easy to see that $ V_1 \subset V_2 \subset \ldots $ is an exhaustion of $ U $, and the sequence of restrictions $ h_{i|V_i} : V_i \rightarrow h_i(V_i) \subset h(U) $ converges uniformly to $ h_{|U} : U \rightarrow h(U) $ on compact subsets of $ U $. \\
\indent Let us now show that the inverse $h^{-1} : M' \rightarrow M $ is also a symplectic homeomorphism. We first claim that every compact subset of $ M' $ is contained in $h_i(U_i)$ for large $ i $, and that $h_i^{-1}:h_i(U_i)\to M$ $\cc^0$-converge on compact sets to $h^{-1}$. To see this, endow $M,M'$ with Riemannian metrics, and fix a compact set $K'\subset M'$. Since $h$ is a homeomorphism, $h^{-1}(K')$ is compact, so there exist an open set $V$ and a number $i_0$ such that $h^{-1}(K')\Subset V\Subset U_{i_0}$. 
By standard degree theory, since $d(K',\partial h(V))>0$, there exists $\eps>0$ such that  any continuous map $f$ whose $\cc^0$-distance from $ h $ is less than $\eps$ verifies $\deg(f,y,V)=\deg (h,y,V)$ for all $y\in K'$. Hence, there exists $i_1\geqslant i_0$ such that
$$
\forall i\geqslant i_1, \forall y\in K', \; \deg(h_i,y,V)=\deg(h,y,V)=1.
$$
There is therefore a solution to $h_i(x)=y$ in $V\subset U_{i_0}\subset U_i$, so $h_i(U_i)\supset K'$ for all $i\geqslant i_1$.  Now since $h^{-1}$ is uniformly continuous on $h(V)$, it has a modulus of continuity 
$$
\delta_{h(V)}(h^{-1},\eps):=\sup\{d(h(x), h(y)), \; x,y\in h(V), \; d(x,y)\leqslant \eps\},
$$
which tends to $0$ when $\eps$ goes to $0$. 
If $y\in K'$ and $i\geqslant i_1$,
$$
\begin{array}{lll}
d(h^{-1}(y),h_i^{-1}(y))&\leqslant d(h^{-1}(y), h^{-1}(h\circ h_i^{-1}(y))) &\\
& \leqslant \delta_{h(V)}\big(h^{-1},d(y,h\circ h_i^{-1}(y))\big), & \text{ since } y,h\circ h_i^{-1}(y)\in h(V)\\
& \leqslant \delta_{h(V)}\big(h^{-1},d(h_i(h_i^{-1}(y)),h(h_i^{-1}(y)))\big) & \\
& \leqslant \delta_{h(V)}\big(h^{-1},d_{\cc^0,V}(h_i,h)\big),& \text{ since } h_i^{-1}(y)\in V. 
\end{array}
$$ 

Therefore, the sequence of open sets $ \displaystyle V_i = \text{int} (\bigcap_{j=i}^{\infty} h_j (U_j)) $ exhausts $ M' $, and since $ V_i \subset h_i(U_i) $ for each $ i $, $h_i^{-1}$ is defined on $V_i$ and converges to $h^{-1}$ on compact subsets of $M'$. Thus $h^{-1}$ is indeed a symplectic homeomorphism.

Finally, let $ h \in \sympeo(M,M') $ and $ h' \in \sympeo(M',M'') $, such that $ h $ is the $ \cc^0 $-limit of a sequence of symplectic embeddings $ h_i : U_i \rightarrow M' $, and $ h' $ is the $ \cc^0 $-limit of a sequence of symplectic embeddings $ h_i' : U_i' \rightarrow M'' $, where $ U_1 \subset U_2 \subset \ldots $ is an exhaustion of $ M $ and $ U_1' \subset U_2' \subset \ldots $ is an exhaustion of $ M' $. Denote $ \displaystyle V_i =  \text{int} ( \bigcap_{j=i}^{\infty} h_j^{-1}(U_j') ) $ and $ W_i = U_i \cap V_i $. Then $ W_1 \subset W_2 \subset \ldots $ is an exhaustion of $ M $, and the sequence of symplectic embeddings $ h_i' \circ h_{i|W_i} : W_i \rightarrow M'' $, $ i=1,2,\ldots $ converges uniformly to the composition $ h' \circ h $ on compact subsets of $ M $, and hence $ h' \circ h \in \sympeo(M,M'') $. \cqfd

We will also make use of the following notations in the course of the paper: 
 
\noindent $\bullet$ $D = \{ z \in \mathbb{C} \; | \; |z| < 1 \} $; $ D(r) = \{ z \in \mathbb{C} \; | \; |z| < r \} $; $ \overline{D} =  \{ z \in \mathbb{C} \; | \; |z| \leqslant 1 \} $. 

\noindent $ \bullet $ $ S(r) = \partial D(r) = \{ z \in \mathbb{C} \; | \; |z| = r \} $ stands for the circle of radius $ r $.

\noindent $ \bullet $ $ A(r_1,r_2) = \{ z \in \mathbb{C} \; | \; r_1 < |z| < r_2 \} $ is an annulus.

\noindent $ \bullet $ $\D(a)=D(\sqrt{\frac{a}{\pi}})$ stands for the disc of area $a$.

\noindent $ \bullet $ $B(a)=B^{2n}(a) \subset \mathbb{R}^{2n}$ stands for the euclidean ball of capacity $a$. 

\noindent $ \bullet $ For a pair of points $ x,y \in \mathbb{R}^m $, we denote by $ d(x,y) $ the standard Euclidean distance between $ x $ and $ y $. 

\noindent $ \bullet $ For any set $ X $ and a subset $ Y \subset \mathbb{R}^m $, the $\cc^0$-topology on the space of functions $f : X \rightarrow Y $ is the topology of the $\cc^0$-distance, defined by 
$$
d_{\cc^0}(f,g)=\sup_{x\in X} d(f(x),g(x)).
$$

\noindent $ \bullet $ Let $ M $ be a symplectic manifold, and let $ \cv \Subset \cv' \subset M $ be open sets. A $ (\cv,\cv') $-cutoff of a Hamiltonian $ H : M \times [0,1] \rightarrow \mathbb{R} $ is a Hamiltonian $ \widetilde{H} : M \rightarrow \mathbb{R} $ of the form $ \widetilde{H}(x,t) = \phi(x) \cdot H(x,t) $, 
$ (x,t) \in M \times [0,1] $, where $ \phi : M \rightarrow [0,1] $ is a smooth function such that $ \supp \phi \Subset \cv' $ and $ \phi = 1 $ on a neighbourhood of $ \overline{\cv} $. The function $ \phi $ is called a $ (\cv,\cv') $-cutoff function.

\subsection*{Acknowledgements}
LB was partially supported by the Israel Science Foundation grant 1380/13, and by the Raymond and Beverly Sackler Career Development Chair. EO was supported by the grant ANR-116JS01-010-01. We thank Yakov Eliashberg, Vincent Humili\`ere, Leonid Polterovich, Sobhan Seyfaddini and Sasha Sodin for carefully listening to the proofs, and for their remarks. We thank Yaron Ostrover for providing us a reference which improved our paper. We finally thank the referees for valuable suggestions, which in particular improved the readability of the paper.

\section{Flexibility of codimension $4$ symplectic submanifolds}\label{sec:c0flexcodim4}
Let us start with a decompressed statement of theorem \ref{thm:flexcodim4}, in the particular case of discs. 
\begin{theorem} \label{T:C0-flexibility}
Let $ (M,\omega) $ be a connected symplectic manifold of dimension $ 2n \geqslant 6 $. Let $ a > 0 $, and let $ u_1,u_2 : \overline{D} \rightarrow M $ be smooth embeddings such that $ u_1^* \omega = \omega_{\st} $, $ u_2^* \omega = a^2 \omega_{\st} $. 
Then there exists a sequence $ \phi_1,\phi_2,\ldots $ of uniformly compactly supported symplectic diffeomorphisms of $ M $ (by uniformly compactly supported we mean that there exists a compact set $ K \subseteq M $ such that $ supp(\phi_i) \subseteq K $ for $ i = 1,2,\ldots $), and a homeomorphism $ \phi : M \rightarrow M $, such that $ \phi_i $, $ i=1,2,\ldots $ $\cc^0 $-converges to $ \phi $, and such that $ u_2 = \phi \circ u_1 $.
\end{theorem}
In fact, as we now explain, this particular case implies theorem \ref{thm:flexcodim4}. 

\noindent {\it Proof of theorem \ref{thm:flexcodim4}:} Consider coordinates $(z_1,z_2,z')$ on $\C^n$ with $z':=(z_3,\dots,z_n)$, $|z'|:=\max |z_i|$, and define $Q:=\{(0,0,z'),\; |z'|\leqslant 1\}$. We need to prove that there exists a symplectic homeomorphism $h$ with support in $\D(\eps)^2\times \D(1+\eps)^{n-2}$ such that $h_{|Q}=\frac 12 \id$: 
$$
h(0,0,z')=(0,0,\frac 12 z'), \hspace{,5cm} \forall |z'|\leqslant 1
$$ 
Let us fix $\eps'<\eps$. By theorem \ref{T:C0-flexibility} (applied to $ M = \D(\eps')^2 \times \D(1+\eps') $), there is a symplectic homeomorphism $\phi : \mathbb{C}^3 \rightarrow \mathbb{C}^3 $ with support in $\D(\eps')^2\times \D(1+\eps')$ such that $\phi(0,0,w)=(0,0,\frac 12 w)$ for $|w|\leqslant 1$. Define the maps $f_i:=\phi_{|\C^3(z_1,z_2,z_i)}\times \id_{\C^{n-3}(z_3,\dots,\hat z_i,\dots,z_n)}$. In other terms, $f_i$ is the symplectic homeomorphism of $\C^n$ defined by 
$$
f_i(z_1,\dots,z_n)=(\phi_{1}(z_1,z_2,z_i),\phi_{2}(z_1,z_2,z_i),z_3,\dots,z_{i-1},\phi_{3}(z_1,z_2,z_i),z_{i+1},\dots,z_n), 
$$
where $\phi_{1},\phi_{2},\phi_{3}$ are the complex coordinates of the map $\phi$ in $\C^3$.
These maps verify $f_{i|Q}=(\frac 12 \id_{z_i})\times \id_{z_3,\dots,\hat z_i,\dots,z_n}$, so their composition $\tilde  h:=f_3\circ \dots \circ f_n$ is $\frac 12 \id$ on $Q$, but they only have support in $\D(\eps')^2\times \C^{n-2}$, so the support of $\tilde h$ is too big for us. Observe however that since $f_i$ does not act on the $(z_3,\dots,\hat z_i,\dots,z_n)$ components, each $f_i$ preserves $\D(\eps)^2\times \D(1+\eps')^{n-2}$. Since $Q$ can be displaced out of $\D(\eps')^2\times \C^{n-2}\supset \supp \tilde h$ by a symplectic diffeomorphism $ g $ with support in $\D(\eps)^2\times \D(1+\eps')^{n-2}$, we conclude by the next lemma (with $S=\supp \tilde h$, $K=\supp g$).\cqfd

\begin{lemma}\label{le:restrsympeo}
Let $\tilde  h:M\to M$ be a symplectic homeomorphism with possibly non-compact support $S$. Let $Q\subset M$ be a subset which can be displaced from $S$ by a symplectomorphism of $M$ with compact support $K$. Then, there is a symplectic homeomorphism $ h$ with support in $K\cup \tilde h(K)$ and such that $h_{|Q}=\tilde  h_{|Q}$.
\end{lemma}
\noindent{\it Proof:} Let $g:M\to M$ be a symplectomorphism with $g(Q)\cap S=\emptyset$ and compact support $ K \subset M $. The map $h:=\tilde  h\circ g^{-1}\circ \tilde  h^{-1}\circ g$ has the required property. It is indeed a symplectic homeomorphism, and if $x\in Q$, $g(x)\notin S$ so $\tilde  h^{-1}(g(x))=g(x)$, so $h(x)=\tilde  h\circ g^{-1}(g(x))=\tilde  h(x)$. Moreover, $\supp h\subset K\cup \tilde  h(K)\Subset M$ because if $x\notin K\cup \tilde  h(K)$, $g(x)=x$ and $g(\tilde  h^{-1}(x))=\tilde  h^{-1}(x)$, so $h(x)=x$. \cqfd 

\begin{remark}\label{rk:accurate-flexibility}
It is easy to see from proofs of theorem~\ref{thm:flexcodim4} and lemma~\ref{le:restrsympeo}, that the symplectic homeomorphism $ h : \C^n \rightarrow \C^n$ which was constructed in the proof of theorem~\ref{thm:flexcodim4} is not only supported in an arbitrary neighbourhood of $ Q = \D(1)^m\times\{0_{n-m}\}\subset \C^n$, but moreover is the $ \cc^0 $-limit of a sequence of symplectomorphisms of $ \C^n $ which are all supported in an arbitrary neighbourhood of $ Q $. Hence standard \nbd theorems imply that the statement of theorem~\ref{thm:flexcodim4} holds true when $ Q $ is the image of a symplectic embedding of a polydisc $ \D(a)^m $ of codimension at least $ 4 $ in any symplectic manifold.
\end{remark}

The rest of this section is aimed at proving theorem \ref{T:C0-flexibility}.

\subsection{Proof of theorem~\ref{T:C0-flexibility}} \label{SubS:proofs-main}
We now prove theorem \ref{T:C0-flexibility} modulo four claims (\ref{Cl:1}, \ref{Cl:2}, \ref{Cl:1-step-m}, \ref{Cl:2-step-m}), which will be proven in the next section. 
The case of $ a = 1 $ is clear, since in this case  $ u_1 $ is Hamiltonian isotopic to $ u_2 $ (see lemma~\ref{L:isotopy-closed-discs}). Therefore in the sequel we may assume that $ a \neq 1 $. For each $ k \in \mathbb{N} $ choose an immersion $ f_k : D(2) \rightarrow \mathbb{C} $, such that $ f_k^{*} \omega_{\st} = (1-a^2) \omega_{\st} $ on $ D $, and such that $ f_k(\overline{D}) \subset  D(1/2^k) $. Consider symplectic embeddings $ i_k : \overline{D} \rightarrow \mathbb{C}^n $ given by $ i_k(z) = (az, f_k(z),0,\ldots,0) $. Also consider the embedding $ i^a : \overline{D} \rightarrow \mathbb{C}^n $ given by $ i^a (z) = (az,0,\ldots,0)  $. We define a family of neighbourhoods $ W_k(\delta) $ of $ i_k(\overline{D}) $ and $ W^a(\delta) $ of $ i^a(\overline{D}) $, $ k=1,2,\ldots $, for $ 0 < \delta < 1 $, by $$ W_k(\delta) = \{ (z_1,\ldots,z_n) \in \mathbb{C}^n \; | \; |z_1| < a + a \delta, |z_2 - f_k(z_1 / a)| < \delta, |z_j| < \delta \,\, \text{for} \,\, 3 \leqslant j \leqslant n \} ,$$ $$ W^a(\delta) = \{ (z_1,\ldots,z_n) \in \mathbb{C}^n \; | \; |z_1| < a + \delta,  |z_j| < \delta \,\, \text{for} \,\, 2 \leqslant j \leqslant n \}.$$ Also introduce a family of discs $ i_{k,l}^a : \overline{D} \rightarrow \mathbb{C}^n $ by $ i_{k,l}^a(z) = (az, f_k(z), \overline{f_l(z)},0,\ldots,0) $. We have $ i_k(\overline{D}) \subset W^a(1/2^k) $, $ i_{k,l}^a(\overline{D}) \subset W_k(1/2^l) $, $ (i^a)^* \omega = (i_{k,l}^a)^* \omega = a^2 \omega_{\st} $, and $ i_k^* \omega = \omega_{\st} $. 

Since $i^a(\overline{D}), i_k(\overline{D})\subset W^a(1/2^k)$, it is enough to show that for any $k$, there exists a sequence $ \phi_1,\phi_2,\ldots $ of symplectic diffeomorphisms of $ \mathbb{C}^n $ supported inside $ W^a(1/2^k) $, and a homeomorphism $ \phi : \mathbb{C}^n \rightarrow \mathbb{C}^n $, such that the sequence $(\phi_i)_{i\geqslant 1}$ $C^0 $-converges to $ \phi $, and such that $ i^a = \phi \circ i_k $ (the sufficiency to prove the latter follows from the symplectic neighbourhood theorem, and from the fact that in any connected symplectic manifold, any two embedded closed symplectic discs of the same symplectic area are Hamiltonianly isotopic, see lemma~\ref{L:isotopy-closed-discs}). The rest of the proof is aimed at proving the latter statement. 

So, let $ k $ be any natural number. Denote $ k_1 = k $, $ \epsilon_1 = 1/ 2^{k_1} $, $ U_1 = W^a(\epsilon_1) \supset i_{k_1}(\overline{D}) $. Choose a decreasing sequence $ U_1 \supset U_2 \supset U_3 \supset \ldots $ of open sets, such that $ \cap_j U_j  = i_{k_1}(\overline{D}) $. Now choose some $ l_1 \geqslant k_1 $, such that for $ 0 <  \delta_1:= 1/2^{l_1} \leqslant \epsilon_1 < 1 $ we have $ U_2 \cap W^a(\epsilon_1) \supset W_{k_1}(\delta_1) $. In addition we have $ W^a(\epsilon_1) \supset W_{k_1}(\delta_1) \supset i_{k_1,l_1}^a(\overline{D}) $.

\begin{claim} \label{Cl:1}
There exists a Hamiltonian isotopy supported inside $ W^a(\epsilon_1) $, whose time-1 map $ \psi_1' $ satisfies $ i^a = \psi_1' \circ i_{k_1,l_1}^a $, and $ d_{\cc^0}(\id, \psi_1') < 3 \epsilon_1 $.
\end{claim}

Now look at $ \psi_1'(W_{k_1}(\delta_1)) \supset i^a(\overline{D}), \psi_1' \circ i_{k_1} (\overline{D}) $. Choose some $ k_2 \in \mathbb{N} $ such that $ k_2 > k_1 $ and such that for $ 0 < \epsilon_2 := 1/2^{k_2} < 1 $ we have $ W^a(\epsilon_2) \subset \psi_1'(W_{k_1}(\delta_1)) $. We have $ \psi_1' \circ i_{k_1} (\overline{D}) \subset \psi_1'(W_{k_1}(\delta_1)) $ and $ i_{k_2} (\overline{D}) \subset W^a(\epsilon_2) \subset \psi_1'(W_{k_1}(\delta_1)) $.

\begin{claim} \label{Cl:2}
There exists a Hamiltonian isotopy supported inside $ \psi_1'(W_{k_1}(\delta_1)) $, whose time-1 map $ \psi_1'' $ satisfies $ i_{k_2} = \psi_1'' \circ \psi_1' \circ i_{k_1} $, and $ d_{\cc^0}(\id, \psi_1'') < 30 \epsilon_1 $.
\end{claim}

Letting  $ \psi_1 = \psi_1'' \circ \psi_1' $, we get a symplectomorphism with support in $ W^a(\epsilon_1) \subset U_1 $, such that $ \psi_1(U_2) \supset W^a(\epsilon_2) \supset i_{k_2}(\overline{D}) $,
 $ i_{k_2} = \psi_1 \circ i_{k_1} $, and $ d_{\cc^0}(\id,\psi_1) < 33 \epsilon_1 $. Figure \ref{fig:balagan} below summarizes this first step of the induction.

\begin{figure}[h!]
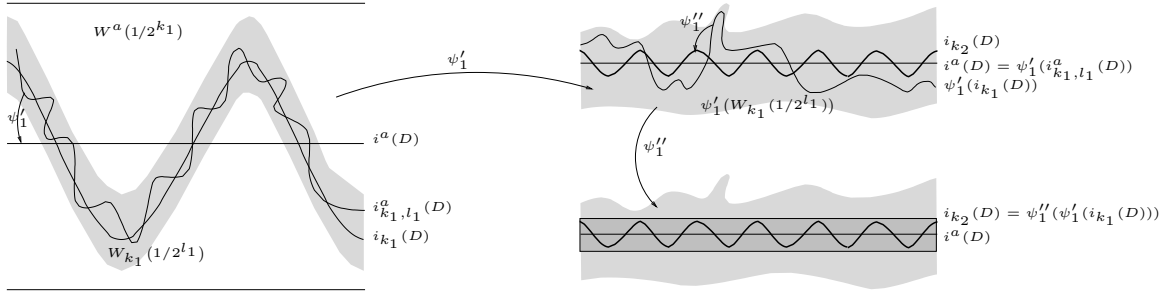

\input bazar.pstex_t
\begin{center}
\caption{A typical step of the induction. }\vspace*{,2cm}
{\footnotesize 
\begin{itemize}
\item[\sbull] $\psi_1$ has support in $W^a(\nf 1{2^{k_1}})$ and takes $U_2\subset W_{k_1}(\nf 1{2^{l_1}})$ to a \nbd of $i^a(\overline{D})$, so that  $\psi_2$ is ready to have support in $\psi_1(U_2)$.
\item[\sbull] The role of $\psi_1'$ is to achieve $ \psi_1' (U_2) \supset i^a(\overline{D})$.
\item[\sbull] The role of $\psi_1''$ is to achieve $\psi_1 \circ i_{k_1} = \psi_1'' \circ \psi_1' \circ i_{k_1} = i_{k_2} $.
\item[\sbull] Although the first isotopy ($\psi_1'$) lies in a standard \nbd of $i^a(\overline{D})$ and can be made relatively explicit, the second one ($\psi_1''$) takes an unknown disc to $i_{k_2}(D)$, in some open set on which we have few control.  
\end{itemize}}
\label{fig:balagan}
\end{center}
\end{figure}

Let us describe the step $ m $, when $ m > 1 $. At steps $ 1, \ldots , m-1 $ we have constructed Hamiltonian diffeomorphisms $ \psi_1,\ldots,\psi_{m-1} $. 
Denote $ \phi_{m-1} = \psi_{m-1} \circ \psi_{m-2} \circ \ldots \circ \psi_1 $. According to the previous step, we have $ \phi_{m-1}(U_m) \supset  
W^a(\epsilon_m) \supset i_{k_m}(\overline{D}) $ and $ i_{k_m} = \phi_{m-1} \circ i_{k_1} $. 

Choose some $ l_m \geqslant k_m $, such that for $ 0 < \delta_m:= 1/2^{l_m} \leqslant \epsilon_m < 1 $ we have $ \phi_{m-1}(U_{m+1}) \cap W^a(\epsilon_m) \supset W_{k_m}(\delta_m) $. In addition we have $ W^a(\epsilon_m) \supset W_{k_m}(\delta_m) \supset i_{k_m,l_m}^a(\overline{D}) $.

\begin{claim} \label{Cl:1-step-m}
There exists a Hamiltonian isotopy supported inside $ W^a(\epsilon_m) $, whose time-1 map $ \psi_m' $ satisfies $ i^a = \psi_m' \circ i_{k_m,l_m}^a $, and $ d_{\cc^0}(\id, \psi_m') < 3 \epsilon_m $.
\end{claim}

Now look at $ \psi_m'(W_{k_m}(\delta_m)) \supset i^a(\overline{D}), \psi_m' \circ i_{k_m} (\overline{D}) $. Choose some $ k_{m+1} \in \mathbb{N} $ such that $ k_{m+1} > k_m $, and such that for $ 0 < \epsilon_{m+1} := 1/2^{k_{m+1}} < 1 $ we have $ W^a(\epsilon_{m+1}) \subset \psi_m'(W_{k_m}(\delta_m)) $. We have $ \psi_m' \circ i_{k_m} (\overline{D}) \subset \psi_m'(W_{k_m}(\delta_m)) $ and $ i_{k_{m+1}} (\overline{D}) \subset W^a(\epsilon_{m+1}) \subset \psi_m'(W_{k_m}(\delta_m)) $.

\begin{claim} \label{Cl:2-step-m}
There exists a Hamiltonian isotopy supported inside $ \psi_m'(W_{k_m}(\delta_m)) $, whose time-$1$ map $ \psi_m'' $ satisfies $ i_{k_{m+1}} = \psi_m'' \circ \psi_m' \circ i_{k_m} $, and $ d_{\cc^0}(\id, \psi_m'') < 30 \epsilon_m $.
\end{claim}

Letting  $ \psi_m := \psi_m'' \circ \psi_m' $, we get a symplectomorphism with support in $ W^a(\epsilon_m) \subset \phi_{m-1}(U_m) \subset U_1 $, such that $ \psi_m \circ \phi_{m-1}(U_{m+1}) \supset W^a(\epsilon_{m+1}) \supset i_{k_{m+1}}(\overline{D}) $,  $ i_{k_{m+1}} = \psi_m \circ i_{k_m} $, and $ d_{\cc^0}(\id,\psi_m) < 33 \epsilon_m $.

As a result of this induction, we get a sequence $ \psi_1,\psi_2,\ldots $ of symplectomorphisms supported inside $ U_1 $ such that:
\begin{itemize}
\item[(i)] $\psi_m$ has support in $W_a(\eps_m)\subset \phi_{m-1}(U_m)$, where $\phi_{m-1}:=\psi_{m-1}\circ \dots\circ \psi_1$,
\item[(ii)]  $ d_{\cc^0}(\id,\psi_m) < 33 \epsilon_m \leqslant 33 / 2^m $,
\item[(iii)] $ i_{k_{m+1}} = \phi_{m} \circ i_{k} $.
\end{itemize}
It follows by (ii) that $\phi_m$ is a Cauchy sequence in the $ \cc^0 $ topology hence uniformly converges to some continuous map $ \phi : \mathbb{C}^n \rightarrow \mathbb{C}^n $. Next, since $ i_{k_{m+1}} = \phi_{m} \circ i_{k} $ for every $ m \geqslant 1 $, we have $ i^a = \phi \circ i_k $. Finally, we claim that $\phi$ is an injective map, hence a homeomorphism. To see this, consider two points $x\neq y\in \mathbb{C}^n$. If $x,y\in i_k(\overline{D})$, then by (iii), $\phi(x)=i^a\circ i_k^{-1}(x)\neq  i^a\circ i_k^{-1}(y)=\phi(y)$. If $x,y\notin i_k(\overline{D})$, then $x,y\in {}^cU_m$ for $m$ large enough, so by (i), $ \phi_m(x) = \phi_{m+1}(x) = \phi_{m+2}(x) = \ldots = \phi(x) $, and similarly $\phi_m(y)=\phi(y)$ (because the supports of $\psi_m,\psi_{m+1},\dots$ are contained in $\phi_{m-1}(U_m)$), so $\phi(x)=\phi_m(x)\neq \phi_m(y)=\phi(y)$. Finally, if $x\in i_k(\overline{D})$ and $y\notin  i_k(\overline{D})$, $y\in {}^cU_m$ for $m$ large enough, so $\phi(y)=\phi_m(y)\in \phi_m({}^c U_m)\subset \phi_m({}^c U_{m+1})\subset {}^c W_a(\eps_m)$ by (i). Since $\phi(x)\in \im i^a\subset W_a(\eps_m)$, we conclude that in this case also $\phi(x)\neq \phi(y)$.\cqfd

\subsection{Proofs of theorem~\ref{T:qh-principle-main} and of claims \ref{Cl:1}, \ref{Cl:2}, \ref{Cl:1-step-m}, \ref{Cl:2-step-m}} \label{SubS:proofs-auxiliary}

Claims~\ref{Cl:1},~\ref{Cl:2},~\ref{Cl:1-step-m},~\ref{Cl:2-step-m} are consequences of the quantitative $h$-principle for discs, which is summarised in theorem~\ref{T:qh-principle-main}, and which might be of  independent interest. Before we pass to the proof of the theorem, let us introduce some notation. For a given homotopy of discs $ F : \overline{D} \times [0,1] \rightarrow \mathbb{R}^m $, we define its size by $$ \size F:= \max_{z \in \overline{D}} \diam F(\{ z \} \times [0,1]) ,$$ and for a given compactly supported flow $ \phi : \mathbb{R}^m \times [0,1] \rightarrow \mathbb{R}^m $, $ \phi(x,t) = \phi^t(x) $, we define its size by $$ \size \phi := \max_{x \in \mathbb{R}^m} \diam \phi(\{ x \} \times [0,1]) .$$ 
In this terminology theorem \ref{T:qh-principle-main} can be stated as follows:

{\bf \noindent Theorem \ref{T:qh-principle-main}} \text{[Quantitative $h$-principle for symplectic discs]}{\bf .} {\it Let $ \epsilon > 0 $ be a positive real, $ n \geqslant 3 $ be an integer, $ W \subset \mathbb{C}^n $ be an open set, and let $ u_1,u_2 : \overline{D} \rightarrow W $ be symplectic discs, with $ u_1^* \omega_{\st} = u_2^* \omega_{\st} = \omega_{\st} $. Assume that there exists a (continuous) homotopy between $ u_1 $ and $ u_2 $ in $ W $, of size less than $ \epsilon $ (i.e. a continuous map $ F : \overline{D} \times [0,1] \rightarrow W $, such that $ F(z,0) = u_1(z) $, $ F(z,1) = u_2(z) $, for all $ z \in \overline{D} $, and that $ \size F < \epsilon $). Then there exists a smooth Hamiltonian flow $ \phi $ supported in $ W $ and of size less than $ 2\epsilon $, whose time-$1$ map sends $ u_1 $ to $ u_2 $ (i.e. a Hamiltonian isotopy $ \phi^t : W \rightarrow W $, $ t \in [0,1] $, which is generated by a smooth time-dependent Hamiltonian supported in $ W $, such that $ u_2 = \phi^1 \circ u_1 $, and such that $ \size \phi < 2 \epsilon $). Moreover, when $ n > 3 $ there exists such a Hamiltonian flow $ \phi $ of $ \size \phi < \epsilon $.} \\ 

The first step of the proof, implemented in  lemma \ref{L:homotopy-to-isotopy}, consists in constructing, from our homotopy, an {\it isotopy} between $u_1$ and $u_2$ ({\it i.e.} a smooth embedding $F:\overline D\times[0,1]\to W$ such that $F(z,0)=u_1(z)$ and $F(z,1)=u_2(z)$ $\forall z\in \overline{D}$), of size less than $2\eps$ ($\eps$ for $n\geqslant 4$). Theorem \ref{T:qh-principle-main} then follows from the following central statement:

\begin{proposition} \label{P:qh-principle}
Let $ \epsilon > 0 $ be a positive real, $ n \geqslant 3 $ be an integer, $ W \subset \mathbb{C}^n $ be an open set, and let $ u_1,u_2 : \overline{D} \rightarrow W $ be symplectic discs, with $ u_1^* \omega_{\st} = u_2^* \omega_{\st} = \omega_{\st} $. Assume that there exists a smooth embedded isotopy between $ u_1 $ and $ u_2 $ in $ W $, of size less than $ \epsilon $ (i.e. a smooth embedding $ F : \overline{D} \times [0,1] \rightarrow W $, such that $ F(z,0) = u_1(z) $, $ F(z,1) = u_2(z) $, for all $ z \in \overline{D} $, and $ \size F < \epsilon $). Then there exists a smooth Hamiltonian flow supported in $ W $ and of size less than $ \epsilon $, whose time-$1$ map sends $ u_1 $ to $ u_2 $ (i.e. a Hamiltonian flow $ \phi^t : W \rightarrow W $, $ t \in [0,1] $ which is generated by a smooth time-dependent Hamiltonian supported in $ W $, such that $ u_2 = \phi^1 \circ u_1 $, and such that $ \size \phi < \epsilon $). 
\end{proposition} 

\begin{proof}[Proof of proposition~\ref{P:qh-principle}]

Choose $ \delta > 0 $ small enough. Consider the grid on $ \mathbb{R}^2 $ of lines parallel to the axes, with step $ \delta $, or in other words the set $ \{ (x,y) \in \mathbb{R}^2 \; | \; x \in \delta \mathbb{Z} \,\, \text{or} \,\, y \in \delta \mathbb{Z} \} $. Look at the intersection of this grid with the open unit disc $ D $, and add to it the unit circle $ S^1 = \partial D $. The set that we get represents a smooth graph which belongs to $ \overline{D} $, and which divides $ \overline{D} $ into small regions (which mostly are squares with side $ \epsilon $). The vertices of this graph are the points $ (x,y) \in D $ such that $ x,y \in \delta \mathbb{Z} $, and also points $ (x,y) \in S^1 = \partial D $ such that either $ x \in \delta \mathbb{Z} $ or $ y \in \delta \mathbb{Z} $. The edges of this graph are either sides, or parts of sides, of squares that the grid divides, or arcs on $ S^1 $. We denote this graph by $ \Gamma $. The vertices of $ \Gamma $ will be represented by points $ z \in \overline{D} $. For each edge of $ \Gamma $, pick a path $ \gamma : [0,1] \rightarrow \overline{D} $ which parametrises it, and which will represent it during the proof. The faces will be represented by open subsets $ G \subset D $.

The images of $ \Gamma $ by $ u_1 $ and $ u_2 $ provide us with a graph on $ u_1(\overline{D}) $ and a graph on $ u_2(\overline{D}) $. We construct the isotopy between $ u_1 $ and $ u_2 $ in three steps. At the first step, we isotope the vertices of the corresponding graph on $ u_1(\overline{D}) $ to the vertices of the corresponding graph on $ u_2(\overline{D}) $. At the second step, we isotope edges to edges, and at the third step we isotope faces to faces. After performing these three steps, the discs coincide.  All the isotopies of the vertices, edges and faces are combined in such a way that the resulting isotopy is of size less than $ \epsilon $.

Now let us turn to the actual proof. Before we start performing the three steps described above, let us extend the maps $ u_1,u_2 : \overline{D} \rightarrow W $ to a slightly larger open disc, so that we get smooth embeddings $ u_1,u_2 : \overline{D}(1+\mu) \rightarrow W $ with $ u_1^* \omega_{\st} = u_2^*\omega_{\st} = \omega_{\st} $ on $ D(1+\mu) $. Moreover, after decreasing $ \mu $ if necessary, we can extend the embedding $ F : \overline{D} \times [0,1] \rightarrow W $ to a smooth embedding $ F : \overline{D}(1+\mu) \times [-\mu,1+\mu] \times [-\mu,\mu]^{2n-3} \rightarrow W $ (where we identify $ \overline{D} \times [0,1] $ with $ \overline{D} \times [0,1] \times 0_{2n-3} \subset \overline{D}(1+\mu) \times [-\mu,1+\mu] \times [-\mu,\mu]^{2n-3} $), such that $ F(z,0,0_{2n-3}) = u_1(z) $ and $ F(z,1,0_{2n-3}) = u_2(z) $ for every $ z \in D(1+\mu) $. 

Choose a small enough $ \eta > 0 $. For any vertex $ z $, edge $ \gamma $, and face $ G $ of $ \Gamma $, denote by $ U_z $, $ U_\gamma $, and $ U_G $, the $ \eta $-neighbourhoods of $ z $, $ \gamma([0,1]) $, and $ G $ in $ D(1+\mu) $, respectively. $ U_z $, $ U_\gamma $ and $ U_G $ are (topological) discs. Denote $ W_z = F(U_z \times (-\eta,1+\eta) \times (-\eta,\eta)^{2n-3}) $, $ W_\gamma = F(U_\gamma \times (-\eta,1+\eta) \times (-\eta,\eta)^{2n-3}) $ and  $ W_G = F(U_G \times (-\eta,1+\eta) \times (-\eta,\eta)^{2n-3}) $. Then $ W_z $, $ W_\gamma $ and $ W_G $ are (topological) balls in $ \mathbb{C}^n $, and moreover if $ z_1 $ and $ z_2 $ are two different vertices of $ \Gamma $ then $ W_{z_1} $ and $ W_{z_2} $ are disjoint, and if $ G_1 $ and $ G_2 $ are two non-neighbouring faces of $ \Gamma $ (that is, $ \overline{G_1} \cap \overline{G_2} = \emptyset $), then $ W_{G_1} $ and $ W_{G_2} $ are disjoint (provided that $ \eta $ is small enough). For any pair of faces $ G_1 $, $ G_2 $ of $ \Gamma $, we define $ l(G_1,G_2) $ as the minimal integer $ m \geqslant 0 $, such that there exists a sequence of $ m+1 $ faces of $ \Gamma $, which starts with $ G_1 $ and ends with $ G_2 $, and such that any two consequent elements of the sequence are neighbouring faces (and so in particular, according to this definition, we have $ l(G,G) = 0 $ for every face $ G $ of $ \Gamma $). For any face $ G $ of $ \Gamma $ and for every integer $ l \geqslant 1 $, we denote by $ W_G^{(l)} $ the union of all $ W_{G_1} $ where $ G_1 $ is a face of $ \Gamma $ and $ l(G,G_1) \leqslant l $. We also denote $ W_G' := W_G^{(1)} $.
 \\ \\
{\bf Step I} \\ \\
Consider any vertex $ z $ of $ \Gamma $. We have $ u_1(z),u_2(z) \in W_z $. Therefore we can find a Hamiltonian isotopy $ \psi_z^t $, $ t \in [0,1] $, supported in $ W_z $, such that $ \psi_z^1 \circ u_1 = u_2 $ on a neighbourhood of $ z $ (this follows from lemma~\ref{L:isotopy-closed-discs}). Denote by $ \Psi_{\cv}^t $, $ t \in [0,1] $ a (reparametrised) concatenation of all the isotopies $ \psi_z $ when $ z $ runs over the vertices of $ \Gamma $. Denote  $ u_1' = \Psi_{\cv}^1 \circ u_1 $.  We have $ u_1'=u_2 $ on a \nbd of each vertex of $ \Gamma $. Note that the flow $ \Psi_{\cv} $ is supported inside the disjoint union $ \cup_z W_z $. Moreover, for each vertex $ z $, edge $ \gamma $, and face $ G $ of $ \Gamma $, if $ W_z \cap W_\gamma \neq \emptyset $ then $ W_z \subset W_\gamma $, and if $ W_z \cap W_G \neq \emptyset $ then $ W_z \subset W_G $. Hence $ \Psi_{\cv}^t(W_\gamma) = W_\gamma $ and $ \Psi_{\cv}^t(W_G) = W_G $ for any edge $ \gamma $ and face $ G $ of $ \Gamma $, and any $ t \in [0,1] $. In particular, $ u_1' \circ \gamma([0,1]) \subset W_\gamma $ for any edge $ \gamma $ of $ \Gamma $, and $ u_1'(\overline{G}) \subset W_G $ for any face $ G $ of $ \Gamma $. Let us summarise the properties which will be important in the sequel: \\
\begin{align}
& u_1'=u_2 \text{ on a \nbd of each vertex of }\Gamma, \numv\label{A:Vachieve}\\
& \Psi_{\cv}^t(W_G) = W_G, \text{ for any face } G \text{ of }  \Gamma \text{ and each } t \in [0,1], \numv\label{A:Vcontrolfnbd}\\
& u_1' \circ \gamma([0,1]) \subset W_\gamma \text{ for any edge } \gamma  \text{ of }  \Gamma, \numv\label{A:Vcontrole}\\
& u_1'(\overline{G}) \subset W_G \text{ for any face } G \text{ of }  \Gamma. \numv\label{A:Vcontrolf}
\end{align}
\\ \\
{\bf Step II} \\ \\
Choose some 1-form $ \lambda $ on $ \mathbb{C}^n $ such that $ d \lambda = \omega_{\st} $. We divide the second step into two sub-steps: \\ \\
{\it \underline{Adjusting actions of the edges:}} \\ 
Look at the discs $ u_1' $ and $ u_2 $. For any edge $ \gamma $ of $ \Gamma $, the integrals $\ca(u_1' \circ \gamma):= \int_{u_1' \circ \gamma} \lambda $ and $ \ca(u_2 \circ \gamma):=\int_{u_2 \circ \gamma} \lambda $  are well-defined, and do not necessarily coincide (we call these integrals ``actions of $ u_1' \circ \gamma $ and $ u_2 \circ \gamma $"). The purpose of this sub-step is to slightly perturb $ u_1' $ so that the actions of $ u_1' \circ \gamma $ and $ u_2 \circ \gamma $ coincide.
To do this, fix a vertex $z_0$ of $ \Gamma $, and for any other vertex $z$ of $ \Gamma $, choose a path $\gamma_z$ made of successive edges of $ \Gamma $ which joins $z_0$ to $z$. Define 
$$
a_z:=\int_{u_2 \circ \gamma_z}\lambda-\int_{u_1' \circ \gamma_z}\lambda.
$$
Notice that these numbers depend on the choice of $z_0$ but not of $\gamma_z$ because $u_2^*\om_\st={u_1'}^*\om_\st$. Then, for each edge $\gamma$ of $ \Gamma $, 
$$
\ca(u_1' \circ \gamma)+a_{\gamma(1)}-a_{\gamma(0)}=\ca(u_2 \circ \gamma)
$$
(because $a_{\gamma(1)}$ can be obtained by integrating $\lambda$ along a path that joins $z_0$ to $\gamma(0)$, concatenated with $\gamma$).
Choose now disjoint annuli $ A_z = \{ w \in \mathbb{C} \; | \; \rho_z  < |w-z| < \rho_z' \} \subset U_z \subset D(1+\mu) $, for all vertices $ z $ of $ \Gamma$. Consider a Hamiltonian function $ H_\mathcal{D} $ on $D(1+\mu)$ with support in $\cup D(z,\rho_z')$, and which is equal to $ -a_z $ on $ D(z,\rho_z) $. The induced Hamiltonian isotopy is supported inside $\cup_z A_z$, and its time-$1$ map $ f : D(1+\mu) \rightarrow D(1+\mu) $ is such that for each edge $ \gamma $ of $ \Gamma $, the area between $ \gamma$ and $ f \circ \gamma$ equals $ a_{\gamma(1)}-a_{\gamma(0)} $, so the actions of $u_2$ and $u_1' \circ f$ coincide on each edge. Extending the push-forward $ (u_1')_* H_\cd $ to a smooth function $ H_\ca$ with support in $\cup_z W_z$, and whose derivatives vanish in the directions symplectic orthogonal to $Tu_1'(D)$, we get a Hamiltonian function whose Hamiltonian flow $ \Psi_{\ca}^t $, $ t \in [0,1] $ is supported inside $\cup_z W_z$, and the time-$1$ map of its flow verifies $ \Psi_{\ca}^1 \circ u_1'=u_1' \circ f  $. Denote $ v:=\Psi_{\ca}^1 \circ u_1' $. 
We have:
\begin{align}
& \ca(v \circ \gamma) = \ca(u_2 \circ \gamma) \text{ for each edge } \gamma \text{ of } \Gamma, \numa\label{A:Aactions-coincide} \\
& v=u_2 \text{ on a \nbd of each vertex of }\Gamma. \numa\label{A:Aachieve}
\end{align}
The flow $ \Psi_{\ca} $ is supported inside $ \cup_z W_z $. Moreover, recall that for each vertex $ z $, edge $ \gamma $, and face $ G $ of $ \Gamma $, if $ W_z \cap W_\gamma \neq \emptyset $ then $ W_z \subset W_\gamma $, and if $ W_z \cap W_G \neq \emptyset $ then $ W_z \subset W_G $. Hence $ \Psi_{\ca}^t(W_\gamma) = W_\gamma $ and $ \Psi_{\ca}^t(W_G) = W_G $ for every edge $ \gamma $ and face $ G $ of $ \Gamma $, and any $ t \in [0,1] $, so in particular, by \eqref{A:Vcontrole} and \eqref{A:Vcontrolf}, we have
\begin{align}
& v \circ \gamma([0,1]) \subset W_\gamma, \text{ for any edge } \gamma \text{ of } \Gamma, \numa\label{A:Acontrole} \\
& v(\overline{G}) \subset W_G, \text{ for any face } G \text{ of } \Gamma. \numa\label{A:Acontrolf}
\end{align}

Given two different edges $ \gamma_1 $, $ \gamma_2 $ of $ \Gamma $, the curves $ v \circ \gamma_1 $ and $ u_2 \circ \gamma_2 $ might intersect away from their endpoints. However, since we are in dimension $ 2n > 2 $, we claim that a small perturbation allows to get rid of these intersections.
\begin{claim}
We can slightly perturb the flow $ \Psi_\ca $, such that the correspondingly perturbed disc $ v $ satisfies 
\begin{align}
& v \circ \gamma_1((0,1)) \cap u_2 \circ \gamma_2((0,1)) = \emptyset, \text{ for any pair of different edges } \gamma_1 , \, \gamma_2 \text{ of } \Gamma, \numa\label{A:Aedges-perturb}
\end{align}
and such that we still have \eqref{A:Aactions-coincide}, \eqref{A:Aachieve}, \eqref{A:Acontrole}, \eqref{A:Acontrolf}. 
\end{claim}

\noindent {\it Proof:}
 By \eqref{A:Aachieve}, for each edge $ \gamma $ of $ \Gamma $, we can find a closed interval $ I_\gamma \subset (0,1) $ such that $ v \circ \gamma = u_2 \circ \gamma $ on $ [0,1] \setminus I_\gamma $. Then, pick open sets 
 $$ 
 v \circ \gamma (I_\gamma) \subset \cv_\gamma \Subset \cv_\gamma' \Subset W \setminus \{ v \circ \gamma (0), v \circ \gamma (1) \} ,
 $$ 
 such that the $ \cv_\gamma' $ are pairwise disjoint. For each edge $ \gamma $ of $ \Gamma $ choose a $ (\cv_\gamma,\cv_\gamma') $-cutoff function.

Since we are in dimension $ 2n > 2 $, for each edge $ \gamma $ of $ \Gamma $ we can find an arbitrarily small vector $ r_\gamma \in \mathbb{C}^n $, such that $ (v \circ \gamma_1 ((0,1)) + r_{\gamma_1}) \cap u_2 \circ \gamma_2((0,1)) = \emptyset $ for any pair $ \gamma_1 $, $ \gamma_2 $ of different edges of $ \Gamma $. Then consider the small linear (autonomous) Hamiltonian function whose time-$1$ map is the affine shift by $ r_\gamma $, and consider its $ (\cv_\gamma,\cv_\gamma') $-cutoff (via the initially chosen cutoff function). This gives autonomous Hamiltonian functions, indexed by the edges of $ \Gamma $, with pairwise disjoint supports. Their sum is a Hamiltonian function $H$ which generates a $ \cc^0 $ small Hamiltonian flow $ \theta^t $, $ t \in [0,1] $, such that we have $$ \theta^1 \circ v \circ \gamma_1((0,1)) \cap u_2 \circ \gamma_2((0,1)) = \emptyset $$ for any two different edges $ \gamma_1 $, $ \gamma_2 $ of $ \Gamma $. Now replace the flow $ \Psi_\ca^t $, $ t \in [0,1] $ by the composition of flows $ \theta^t \circ \Psi_\ca^t $, $ t \in [0,1] $, and correspondingly replace $ v $ by $ \theta^1 \circ v $. \eqref{A:Aactions-coincide}, \eqref{A:Aachieve} still hold because $H$ vanishes near $v(z)=u_2(z)$ for each vertex $z$ of $\Gamma$, and \eqref{A:Acontrole}, \eqref{A:Acontrolf} are not destroyed if the flow $(\theta^t)$ is $\cc^0$-small enough.
\cqfd

If an edge $ \gamma $ of $ \Gamma $ is not incidental to a face $ G $ of $ \Gamma $ (i.e. $ \gamma([0,1]) \not\subset \overline{G} $), then by \eqref{A:Aedges-perturb}, we have 
$ v(\partial G) \cap u_2 \circ \gamma((0,1)) = \emptyset $. However, we might still have $ v(G) \cap u_2 \circ \gamma((0,1)) \neq \emptyset $. Nevertheless, since we are in dimension $ 2n \geqslant 4 $, we claim, similarly as before, that a small perturbation allows to get rid of these intersections

\begin{claim}
We can slightly perturb the flow $ \Psi_\ca $, such that correspondingly perturbed disc $ v $ satisfies 
\begin{align}
& v (G) \cap u_2 \circ \gamma((0,1)) = \emptyset, \text{ for any edge } \gamma \text{ and face } G \text{ of } \Gamma \text{ with } \gamma([0,1]) \not\subset \overline{G}, \numa\label{A:Afaces-perturb}
\end{align}
and such that  \eqref{A:Aactions-coincide}, \eqref{A:Aachieve}, \eqref{A:Acontrole}, \eqref{A:Acontrolf}, and \eqref{A:Aedges-perturb} still hold. 
\end{claim}
\noindent{\it Proof:} By \eqref{A:Aachieve} and \eqref{A:Aedges-perturb}, for each face $ G $ of $ \Gamma $, we can find a compact subset $ K_G \subset G $, such that $$ v(G \setminus K_G) \cap u_2 \circ\gamma((0,1)) = \emptyset ,$$ for any 
face $G$ and non-incidental edge $\gamma$ of $\Gamma$.
Thus, we can pick open sets $\cv_G, \cv_G'$ such that $$ v ( K_G ) \subset \cv_G \Subset \cv_G' \Subset W \setminus (v ( \partial G) \cup \{ v(z) \, | \, z \text{ is a vertex of } \Gamma \}),$$ and such that the $\cv_G'$ are pairwise disjoint. For each face $ G $ of $ \Gamma $, choose a $ (\cv_G,\cv_G') $-cutoff function.

Since we are in dimension $ 2n \geqslant 4 $, for each face $ G $ of $ \Gamma $ we can find an arbitrarily small vector $ r_G \in \mathbb{C}^n $, such that $ (v (G) + r_{G}) \cap u_2 \circ \gamma((0,1)) = \emptyset $ for any face $G$ and edge $\gamma$ non-incidental to $ G $.
Then for each face $ G $ of $ \Gamma $, consider the small linear (autonomous) Hamiltonian function whose time-$1$ map is the affine shift by $ r_G $, and consider its $ (\cv_G,\cv_G') $-cutoff (via the initially chosen cutoff function). This gives a collection of autonomous Hamiltonian functions (one for each face $ G $ of $ \Gamma $), whose supports are pairwise disjoint and avoid a \nbd of the image by $v$ and $u_2$ of each vertex  $\Gamma$. Their sum 
is a Hamiltonian function generating a $ \cc^0 $ small Hamiltonian flow $ \theta^t $, $ t \in [0,1] $, which verifies $ \theta^1 \circ v (G) \cap u_2 \circ \gamma((0,1)) = \emptyset $ for any non-incidental face $ G $ and edge $ \gamma $ of $ \Gamma $. Now replace the flow $ \Psi_\ca^t $, $ t \in [0,1] $ by the composition of flows $ \theta^t \circ \Psi_\ca^t $, $ t \in [0,1] $, and correspondingly replace $ v $ by $ \theta^1 \circ v $. It is again immediate to verify that \eqref{A:Aactions-coincide}-\eqref{A:Aedges-perturb} still hold if $(\theta^t)$ is $\cc^0$-small enough. \cqfd

Notice that the proof above uses only \eqref{A:Aachieve} and \eqref{A:Aedges-perturb}, which are symetric in $u_2$ and $v$. Thus a further perturbation, obtained by the same procedure, allows to ensure that 
\begin{align}
& u_2 (G) \cap v\circ \gamma((0,1)) = \emptyset, \text{ for any edge } \gamma \text{ and face } G \text{ of } \Gamma \text{ with } \gamma([0,1]) \not\subset \overline{G}. \numa\label{A:Afaces-perturb'}
\end{align}   
Precisely, the above procedure provides us with a $ \cc^0 $-small Hamiltonian flow $ \theta^t $, $ t \in [0,1] $, supported away from $ \{ u_2(z) \, | \, z 
\text{ is a vertex of } \Gamma \} $, such that $ \theta^1 \circ u_2 (G) \cap v \circ \gamma((0,1)) = \emptyset $
for any face $G$ of $ \Gamma $, and edge $\gamma$ of $ \Gamma $ non-incidental to $G$. Then we replace the flow $ \Psi_\ca^t $, $ t \in [0,1] $ by the composition of flows $ (\theta^t)^{-1} \circ \Psi_\ca^t $, $ t \in [0,1] $, and correspondingly replace $ v $ by $ (\theta^1)^{-1} \circ v $. Notice that the properties \eqref{A:Aactions-coincide}-\eqref{A:Afaces-perturb} still hold after the perturbation, provided that the flow $ \theta $ is $ \cc^0 $-small enough.

After the  above three perturbations we may no longer have $ \Psi_{\ca}^t(W_G) = W_G $ for any face $ G $ of $ \Gamma $ and any $ t \in [0,1] $. Nevertheless, we can assume that these perturbations of $ \Psi_\ca $ are sufficiently $ \cc^0 $-small, so that we still have:
\begin{align}
& \Psi_{\ca}^t(W_G) \subset W_G', \text{ for any face } G \text{ of } \Gamma, \text{ and any } t \in [0,1]. \numa\label{A:Acontrolfnbd} 
\end{align}

\medskip
\noindent {\it \underline{Moving edges to edges:}} \\
Look at the disc $ v = \Psi_{\ca}^1 \circ u_1' $. Since  by  \eqref{A:Aachieve}, $v$ and $u_2$ coincide on a \nbd of  each vertex  of $\Gamma$, there exists $\nu>0$ such that 
$v\circ \gamma(s)=u_2\circ \gamma(s)$ $\forall s\in [0,2\nu]\cup [ 1-2\nu,1]$ and for all edge $\gamma$ of $\Gamma$. By \eqref{A:Acontrole}, \eqref{A:Aedges-perturb}, \eqref{A:Afaces-perturb}, \eqref{A:Afaces-perturb'} if we denote by $G_\gamma$ the union of the two faces that contain $\gamma([0,1])$ in their boundaries (or one face when $\gamma([0,1])$ lies in $\partial D(1+\mu)$), 
\begin{align}
& u_2\circ\gamma([\nu,1-\nu]),v\circ \gamma([\nu,1-\nu])\Subset W_\gamma\priv\big(u_2(D\priv G_\gamma)\cup v(D\priv G_\gamma)\big). \nume\label{A:Eu2vlocalize}
\end{align}
Choose now relative homotopies $F_\gamma$ between 
 $v\circ \gamma_{|[\nu,1-\nu]}$ and $u_2\circ \gamma_{|[\nu,1-\nu]}$ in $W_\gamma\priv \big(u_2(D\priv G_\gamma)\cup v(D\priv G_\gamma) \big)$, {\it i.e.} smooth maps
$F_\gamma: [\nu,1-\nu]\times[0,1]\to W_\gamma\priv \big(u_2(D\priv G_\gamma)\cup v(D\priv G_\gamma) \big)$  such that $F_\gamma(s,0)=v\circ \gamma(s) $, $F_\gamma(s,1)=u_2\circ \gamma(s)$ and $F(s,t)=u_2\circ \gamma(s)=v\circ \gamma(s)$ for all $t\in [0,1] $ and $s\in [\nu ,2\nu]\cup [1-2\nu,1-\nu]$ (this can be done by considering  a relative homotopy in $W_\gamma$, and  perturb it by general position arguments). Notice now that \eqref{A:Aedges-perturb} implies in particular that $u_2\circ \gamma_1([\nu,1-\nu])\cap v\circ \gamma_2([\nu,1-\nu])=\emptyset$ for all pair of distinct edges $\gamma_1,\gamma_2$, so a similar general position argument allows to choose the isotopies $F_\gamma$ with disjoint images. Finally, taking small enough \nbds with smooth boundary $\cv_\gamma$ of each image of $F_\gamma$, we get open sets with the following properties:
 \begin{align}
 & \cv_\gamma\Subset W_\gamma\subset \C^n \nume\label{A:EVgamma-in-Wgamma}\\
 & \cv_{\gamma_1}\cap \cv_{\gamma_2}=\emptyset \text{ for all pair of distinct edges of } \Gamma \nume\label{A:EVgamma-pair-disjoint}\\
  & \cv_\gamma \cap \big(u_2(D\priv G_\gamma)\cup v(D\priv G_\gamma) \big)=\emptyset\nume \label{A:EVgamma-disjoint-f}\\
  & 
  \left\{\begin{array}{l}
  \cv_\gamma\cap u_2\circ \gamma((0,1))=u_2\circ \gamma((\nu',1-\nu')) \\
    \cv_\gamma\cap v\circ \gamma((0,1))=v\circ \gamma((\nu',1-\nu')), 
    \end{array}\right. \text{ for some }\nu'\in ( \nu ,2\nu). \nume\label{A:EVgamma-proper}\\
    & v\circ \gamma_{|(\nu',1-\nu')} \text{ and } u_2\circ \gamma_{|(\nu',1-\nu')} \text{ are relative homotopic in } \cv_\gamma \nume\label{A:Evu2homotopic}
 \end{align}
By  \eqref{A:Aactions-coincide}, we have the equality of actions $ \int_{v \circ \gamma} \lambda = \int_{u_2 \circ \gamma} \lambda $. Since moreover $v\circ \gamma(s)=u_2\circ \gamma(s)$ for $s\in[0,\nu']\cup [1-\nu',1]$, the actions of $v\circ \gamma_{|[\nu',1-\nu']}$ and $u_2\circ \gamma_{|[\nu',1-\nu']}$ also coincide. By \eqref{A:Evu2homotopic}, lemma \ref{L:isotopy-nbd-edge} (b) (completed by remark \ref{rk:balltononball}) therefore gives Hamiltonian functions $H_\gamma:W\times [0,1]\to \R$ with compact supports in $\cv_\gamma\times [0,1]$, whose Hamiltonian flows verify $\psi^1_{H_\gamma}\circ v=u_2$ on a \nbd of $\gamma((\nu',1-\nu')) $ in $ D(1+\mu) $. Since $v$ already coincides with $u_2$ near the part of $v(\gamma)$ outside of $\cv_\gamma$, we conclude that $\psi^1_{H_\gamma}\circ v=u_2$ on a \nbd of $\gamma([0,1])$. The functions $H_\gamma$ have disjoint supports by \eqref{A:EVgamma-pair-disjoint}, so their sum $H:=\sum_\gamma H_\gamma$ is a Hamiltonian function whose flow $(\Psi_\ce)$ verifies
\begin{align}
& u_1'' = \Psi_{\ce}^1 \circ v = u_2 \text{ on a neighbourhood of each edge of } \Gamma. \nume\label{A:Eachieve-2} 
\end{align}
Moreover, by \eqref{A:EVgamma-pair-disjoint} and \eqref{A:EVgamma-in-Wgamma}, a point of $W_G$ is moved along the flow of $H$ by at most one flow $(\psi_{H_\gamma})$, which have support in $W_\gamma\subset W_G'$. Hence, 
\begin{align}
& \Psi_{\ce}^t (W_G) \setminus W_G \Subset W_G' \text{ for any face } G \text{ of } \Gamma \text{ and any } t \in [0,1]\;\footnotemark. \nume\label{A:Econtrolfnbd-prelim} 
\end{align}
\footnotetext{Of course, $W_G\Subset W_G'$ as soon as $G\Subset D$, but not when $G$ is adjacent to $\partial D$.}
In addition, we claim that we have:
\begin{align}
& u_1'' (\overline{G}) = \Psi_{\ce}^1 \circ v  (\overline{G})  \subset W_G \text{ for any face } G \text{ of } \Gamma. \nume\label{A:Econtrolf} 
\end{align}
Indeed, let $ G $ be a face of $ \Gamma $ and let $ z \in \overline{G} $. If $ z \in \partial G $, then by \eqref{A:Eachieve-2}, we have $ u_1''(z) = u_2(z) \in W_G $. If $ z \in G $ and $ v(z) \in \cv_\gamma$ for some edge $ \gamma $ of $ \Gamma $, then by \eqref{A:EVgamma-disjoint-f} we conclude that $ \overline{G} \supset \gamma([0,1]) $, hence $ W_\gamma \subset W_G $, and therefore by \eqref{A:EVgamma-pair-disjoint} and \eqref{A:EVgamma-in-Wgamma}, we get $$ u_1''(z) = \Psi_{\ce}^1 \circ v(z) = \Psi_{H_\gamma}^1 \circ v(z) \in \cv_\gamma\subset W_\gamma \subset W_G.$$ Finally, if $ v(z) \notin\cv_\gamma $ for any edge $ \gamma $ of $ \Gamma $, then $ v(z) \notin \supp H $ and therefore $ u_1''(z) =  \Psi_{\ce}^1 \circ v(z) = v(z) \in W_G $.

\begin{claim}
Since we are in dimension $ 2n > 4 $, we can slightly perturb the flow $ \Psi_{\ce} $, such that \eqref{A:Eachieve-2}, \eqref{A:Econtrolfnbd-prelim} and \eqref{A:Econtrolf} still hold, and such that 
\begin{align}
& u_1''(G_1) \cap u_2(G_2) = \emptyset, \text{ for any two different faces } G_1, G_2 \text{ of } \Gamma. \nume\label{A:Efaces-perturb} 
\end{align}
\end{claim}
\noindent{\it Proof:} 
By \eqref{A:Eachieve-2}, for each face $ G $ of $ \Gamma $, we can find a compact subset $ K_G \subset G $ such that $ u_1'' = u_2  $ on $ G \setminus K_G $. Then, by \eqref{A:Econtrolf} we can pick open sets $\cv_G,\cv_G'$ such that
$$ 
u_1'' ( K_G ) \subset \cv_G \Subset \cv_G' \Subset W_G \setminus ( \cup_\gamma u_1'' \circ \gamma([0,1]) ) ,
$$ 
and such that the $\cv_G'$ are pairwise disjoint. Note that $ \cv_G' \Subset W_G \setminus u_1''(\partial G) $ since $ \partial G \subset \cup_\gamma \gamma([0,1]) $. For each face $ G $ of $ \Gamma $, choose a $ (\cv_G,\cv_G') $-cutoff function.

Since we are in dimension $ 2n > 4 $, for each face $ G $ of $ \Gamma $ we can find an arbitrarily small vector $ r_G \in \mathbb{C}^n $, such that $ (u_1'' (G_1) + r_{G_1}) \cap u_2 (G_2) = \emptyset $ for any pair of different faces $ G_1 $, $ G_2 $ of $ \Gamma $. Then for each face $ G $ of $ \Gamma $, take the small linear (autonomous) Hamiltonian function whose time-$1$ map is the affine shift by $ r_G $, and consider  its $ (\cv_G,\cv_G') $-cutoff (via the initially chosen cutoff function). This gives us a collection of autonomous Hamiltonian functions (one for each face $ G $ of $ \Gamma $), with pairwise disjoint supports. Their sum (over all faces $ G $ of $ \Gamma $) is a Hamiltonian function with $ \cc^0 $ small Hamiltonian flow $ \theta^t $, $ t \in [0,1] $, which verifies $$ \theta^1 \circ u_1'' (G_1) \cap u_2 (G_2) = \emptyset $$ for any pair of different faces $ G_1 $, $ G_2 $ of $ \Gamma $. Now replace the flow $ \Psi_{\ce}^t $, $ t \in [0,1] $ by the composition of flows $ \theta^t \circ \Psi_{\ce}^t $, $ t \in [0,1] $, and correspondingly replace $ u_1'' $ by $ \theta^1 \circ u_1'' $. \eqref{A:Eachieve-2} still holds after the perturbation because the flow $ \theta $ has support in $W_G \setminus \cup_\gamma u_1'' \circ \gamma([0,1])$. \eqref{A:Econtrolf} still holds provided the perturbation is chosen $\cc^0$-small enough. The same is true for \eqref{A:Econtrolfnbd-prelim}, for those faces $G$ of $ \Gamma $ which are compactly contained in $D$. For the other faces (whose closure intersects $ \partial D $), since we had \eqref{A:Econtrolfnbd-prelim} before the perturbation, since the flow $ \theta $ is $ \cc^0 $-small, and since the support of the flow $ \theta $ lies inside $ \cup_G \cv_G' $, where all $ \cv_G' $ are pairwise disjoint and $ \cv_G' \Subset W_G $ for every face $ G $ of $ \Gamma $, we conclude that \eqref{A:Econtrolfnbd-prelim} still holds after the perturbation.\cqfd

Note that \eqref{A:Econtrolfnbd-prelim} implies 
\begin{align}
& \Psi_{\ce}^t(W_G) \subset W_G' \text{ for any face } G \text{ of } \Gamma \text{ and each } t \in [0,1].  \nume\label{A:Econtrolfnbd}
\end{align}
The properties of the Hamiltonian flow $(\Psi^t_\ce)$ that will be important in the sequel are \eqref{A:Eachieve-2}, \eqref{A:Econtrolf}, \eqref{A:Efaces-perturb} and \eqref{A:Econtrolfnbd}.

\medskip
\noindent {\bf Step III} \\ \\
Let $ G $ be a face of $ \Gamma $. We have $ u_1''(\overline{G}), u_2(\overline{G}) \subset W_G $ by \eqref{A:Econtrolf}, and the restrictions of $ u_1'' $ and $ u_2 $ to $ \overline{G} $ coincide near the boundary of $ \overline{G} $ by \eqref{A:Eachieve-2}. The first point is to find local Hamiltonian isotopies in $W_G$ that bring $u_1''(\overline{G})$ to $u_2(\overline{G})$, relative to their boundaries, with the additional requirement that their supports do not intersect $u_2(D\priv \overline{G})\cup u_1''(D\priv \overline{G})$ (pairwise disjoint support would be even better, but now the dimension obstructs). This is doable directly, but not completely transparent from our proof of lemma \ref{L:isotopy-discs-fixed-bdry}, so we proceed slightly differently. By lemma~\ref{L:isotopy-discs-fixed-bdry}, for every face $ G $ of $ \Gamma $ there exists a Hamiltonian isotopy $ \psi_G^t $ of $ W $, generated by a Hamiltonian $ H_G : W \times [0,1] \rightarrow \mathbb{R} $, such that 
\begin{align}
& \supp H_G \Subset W_G \setminus u_2(\partial G), 
 \numf\label{A:Fsupport-H} \\
&  \psi_G^1 \circ u_1'' = u_2 \text{ on } \overline{G}. 
\numf\label{A:Fachieve} 
\end{align}
\begin{claim}
Since we are in dimension $ 2n \geqslant 6 $, after slightly perturbing the flow $ \psi_G $ for each face $ G $ of $ \Gamma $, if necessary, we may assume that \eqref{A:Fsupport-H}, \eqref{A:Fachieve} still hold, and that we moreover have 
\begin{align}
& \psi_G^t \circ u_1''(G) \cap \big(u_1''(\overline{D} \setminus G) \cup u_2(\overline{D} \setminus G)\big) = \emptyset \text{ for any face } G \text{ of } \Gamma \text{ and any } t \in [0,1].  \numf\label{A:Ffaces-perturb} 
\end{align}
\end{claim}
\noindent{\it Proof:} 
 By \eqref{A:Fsupport-H} and \eqref{A:Fachieve}, for each face $ G $ of $ \Gamma $, we can find a compact subset $ K_G \subset G $ such that $ u_1''(z) = \psi_G^t \circ u_1''(z) = u_2(z)  $ for any $ z \in G \setminus K_G $ and $ t \in [0,1] $. Then, by \eqref{A:Eachieve-2}, \eqref{A:Econtrolf} and \eqref{A:Ffaces-perturb}, we can pick open sets $$  \cup_{t \in [0,1]} \psi_G^t \circ u_1''(K_G) \subset \cv_G \Subset \cv_G' \Subset W_G \setminus u_1''(\partial G) = W_G \setminus u_2(\partial G) .$$ For each face $ G $ of $ \Gamma $, choose a $ (\cv_G,\cv_G') $-cutoff function.

Since we are in dimension $ 2n \geqslant 6 $, for each face $ G $ of $ \Gamma $, we can find an arbitrarily small vector $ r_G \in \mathbb{C}^n $, such that $$ ( (\cup_{t \in [0,1]} \psi_G^t \circ u_1''(G)) + r_G) \cap (u_1''(\overline{D} \setminus G) \cup u_2(\overline{D} \setminus G)) = \emptyset .$$ Next, take the small linear (autonomous) Hamiltonian function whose time-$1$ map is the affine shift by $ r_G $, and consider its $ (\cv_G,\cv_G') $-cutoff (via the initially chosen cutoff function). It is a Hamiltonian function which generates a $ \mathcal{C}^0 $-small autonomous Hamiltonian flow $ \theta_G^t $, $ t \in [0,1] $ on $ W $. Finally, by \eqref{A:Efaces-perturb} and \eqref{A:Fachieve}, we have $$ ( u_1''(G) \cup \psi_G^1 \circ u_1''(G) ) \cap (u_1''(\overline{D} \setminus G) \cup u_2(\overline{D} \setminus G)) = \emptyset .$$
Hence if we choose a small enough $ \nu > 0 $, and then pick a smooth function $ c : [0,1] \rightarrow [0,1] $, which equals $ 1 $ on $ [\nu,1-\nu] $ and equals $ 0 $ on a neighbourhood of $ \{ 0,1 \} $, then after replacing the Hamiltonian flows $ \psi_G^t $, $ t \in [0,1] $ by $ \theta_{G}^{c(t)} \circ \psi_G^t $, $ t \in [0,1] $, we obtain the desired situation. \cqfd 

Now, by \eqref{A:Fsupport-H}, for each face $ G $ of $ \Gamma $ we can find a compact subset $ \widetilde{K}_G \subset G $ such that $ H_G = 0 $ on a neighbourhood of $ u_1''(\overline{G} \setminus \widetilde{K}_G) $, and in particular, by  \eqref{A:Fachieve} we have $ u_1''(z) = \psi_G^t \circ u_1''(z) = u_2(z)  $ for any $ z \in G \setminus \widetilde{K}_G $ and $ t \in [0,1] $. Then, by \eqref{A:Ffaces-perturb}, we can pick open sets 
$$ \cup_{t \in [0,1]} \psi_G^t \circ u_1''(\widetilde{K}_G) \subset \widetilde{\cv}_G \Subset \widetilde{\cv}_G' \Subset W_G \setminus ( u_1''(\overline{D} \setminus G) \cup u_2 (\overline{D} \setminus G) ) .$$ 
Denote by $ \widetilde{H}_G $ a $ (\widetilde{\cv}_G, \widetilde{\cv}_G') $-cutoff of $ H_G $, and by $ \tilde{\psi}_G^t $, $ t \in [0,1] $ the Hamiltonian flow of $ \widetilde{H}_G $. Then we get that $ \tilde{\psi}_G^1 \circ u_1'' = u_2 $ on $ \overline{G} $ for every face $ G $ of $ \Gamma $, and that for any two different faces $ G_1, G_2 $ of $ \Gamma $ we have $ \tilde{\psi}_{G_1}^1 \circ u_1'' = u_1'' $ and $ \tilde{\psi}_{G_1}^1 \circ u_2 = u_2 $ on $ \overline{G_2} $. Moreover, for any face $ G $ of $ \Gamma $, the support of $ \tilde{\psi}_{G}^t $, $ t \in [0,1] $ lies inside $ W_G $. 

Now, partition the set of all faces of $ \Gamma $ into four sets $ \mathcal{F}_{00}, \mathcal{F}_{01}, \mathcal{F}_{10}, \mathcal{F}_{11} $, according to the parity. More precisely, for each $ i,j \in \{ 0 ,1 \} $, we denote by $ \mathcal{F}_{ij} $ the set of faces $ G $ of $ \Gamma $ such that for some $ k,l \in \mathbb{Z} $ with $ k \equiv i , l \equiv j (\text{mod} \,\, 2) $, we have $ G \subset [k\delta,(k+1)\delta] \times [l\delta,(l+1)\delta] $. According to our choice of $ \eta $ (which enters the definition of the neighbourhoods $ W_G $) at the beginning of the proof, for any $ i,j \in \{ 0 , 1 \} $ and any two faces in $ G_1,G_2 \in \mathcal{F}_{ij} $, we have $ W_{G_1} \cap W_{G_2} = \emptyset $. For each $ i,j \in \{ 0,1 \} $ denote by $ \Psi_{ij}^t $, $ t \in [0,1] $ a (reparametrised) concatenation of all the flows $ \tilde{\psi}_{G} $, where $ G \in \mathcal{F}_{ij} $. Further, denote by $ \Psi_{\cf}^t $, $ t \in [0,1] $ a (reparametrised) concatenation of the flows $ \Psi_{00} , \Psi_{01}, \Psi_{10}, \Psi_{11} $. We have $ \Psi_{\cf}^1 \circ u_1'' = u_2 $ on $ \overline{D} $. Moreover, for any $ t \in [0,1] $, any face $ G $ of $ \Gamma $, and any $ i,j \in \{ 0,1 \} $ we have $ \Psi_{ij}^t(W_G) \subset W_G' $, and hence for any $ t \in [0,1] $ and any face $ G $ of $ \Gamma $ we have $ \Psi_{\cf}^t(W_G) \subset W_{G}^{(4)} $. \\

Summarising the three steps above, we have constructed Hamiltonian flows $ \Psi_{\cv} $, $ \Psi_{\ca} $, $ \Psi_{\ce} $ and $ \Psi_{\cf} $, compactly supported in $ W $, and smooth symplectic discs $ u_1',u_1'' : D(1+\mu) \rightarrow W $, such that $ u_1' = \Psi_{\cv}^1 \circ u_1 $ on $ D(1+\mu) $, $ u_1'' = \Psi_{\ce}^1 \circ \Psi_{\ca}^1 \circ u_1' $ on $ D(1+\mu) $, and $ u_2 = \Psi_{\cf}^1 \circ u_1'' $ on $ \overline{D} $. Moreover, for any $ t \in [0,1] $ and any face $ G $ of $ \Gamma $, we have $ \Psi_{\cv}^t(W_G) = W_G $, $ \Psi_{\ca}^t(W_G) \subset W_G' $, $ \Psi_{\ce}^t(W_G) \subset W_G' $ and $ \Psi_{\cf}^t (W_G) \subset W_G^{(4)} $. Hence, if we denote by $ \phi^t $, $ t \in [0,1] $, a (reparametrised) concatenation of the flows $ \Psi_{\cf} $, $ \Psi_{\ce} $, $ \Psi_{\ca} $ and $ \Psi_{\cv} $, then the flow $ \phi $ is supported in $ \cup_G W_G \subset W $ and we have $ u_2 = \phi^1 \circ u_1 $. Moreover, for any $ t \in [0,1] $ and any face $ G $ of $ \Gamma $ we have $ \phi^t(W_G) \subset W_G^{(6)} $. Since the parameters $ \delta $ and $ \eta $, chosen at the beginning of the proof, are small enough, we conclude that $ \size \phi < \epsilon $.
\end{proof}

\begin{proof}[Proof of theorem~\ref{T:qh-principle-main}]
The theorem immediately follows from lemma~\ref{L:homotopy-to-isotopy} and proposition~\ref{P:qh-principle}.
\end{proof}

\begin{proof}[Proofs of claims~\ref{Cl:1} and~\ref{Cl:1-step-m}]
Let us present the proof of claim~\ref{Cl:1-step-m} (the proof of claim~\ref{Cl:1} is similar). Since $ W^a(\epsilon_m) $ is convex, and $ d(i^a,i^a_{k_m,l_m}) < \sqrt{2}\epsilon_{m} $, the linear homotopy between $ i^a $ and $ i^a_{k_m,l_m} $ lies in $ W^a(\epsilon_m) $, and has size less than $ \sqrt{2}\epsilon_m $. Hence by theorem~\ref{T:qh-principle-main}, there exists a smooth Hamiltonian isotopy supported in $ W^a(\epsilon_m) $ of size less than $ 2\sqrt{2} \epsilon_m < 3 \epsilon_m $, that sends $ i^a_{k_m,l_m} $ to $ i^a $. Denote by $ \psi_m' $ its time-$1$ map.
\end{proof}

\begin{proof}[Proofs of claims~\ref{Cl:2} and~\ref{Cl:2-step-m}]
Let us present the proof of claim~\ref{Cl:2-step-m} (the proof of claim~\ref{Cl:2} is similar). Consider a smooth embedding $ \tau : W_{k_m}(\delta_m) \rightarrow \mathbb{C}^n $ given by  
$ \tau (z_1,z_2,z') = (z_1, z_2 - f_{k_m}(z_1/a),z') $ (here $ z' = (z_3, \ldots, z_n) $). Then $ d(\id,\tau) < \epsilon_m $. In addition, by claim~\ref{Cl:1} we have $ d(\id,\psi_m') < 3 \epsilon_m $. The domain $ \tau(W_{k_m}(\delta_m)) $ is an euclidean polydisc, and hence is convex. Look at the discs $ \tau \circ i_{k_m} $, $ \tau \circ (\psi_m')^{-1} \circ i_{k_{m+1}} $. We have $ \tau \circ i_{k_m} (\overline{D}), \tau \circ (\psi_m')^{-1} \circ i_{k_{m+1}} (\overline{D}) \subset \tau(W_{k_m}(\delta_m)) $. Moreover, $$ d( \tau \circ i_{k_m} , \tau \circ (\psi_m')^{-1} \circ i_{k_{m+1}}) < 2\epsilon_m +  
d( i_{k_m} , (\psi_m')^{-1} \circ i_{k_{m+1}}) < 5\epsilon_m + d(i_{k_m}, i_{k_{m+1}}) < 7\epsilon_m .$$ Having in mind that $ \tau(W_{k_m}(\delta_m)) $ is convex, this implies that the linear homotopy between $ \tau \circ i_{k_m} $, $ \tau \circ (\psi_m')^{-1} \circ i_{k_{m+1}} $ lies inside $ \tau(W_{k_m}(\delta)) $, and has size less than $ 7 \epsilon_m $. Therefore, in view of $ d(\id, \psi_m' \circ \tau^{-1}) < d(\id,\psi_m') + d(\id, \tau) < 3\epsilon_m + \epsilon_m = 4 \epsilon_m $, applying the map $ \psi_m' \circ \tau^{-1} $, we conclude that there exists a homotopy between $ \psi_m' \circ i_{k_m} $ and $ i_{k_{m+1}} $ inside $ \psi_m'(W_{k_m}(\delta_m)) $, of size less than $ 7\epsilon_m + 2\cdot 4\epsilon_m = 15\epsilon_m $. And finally, by theorem~\ref{T:qh-principle-main}, there exists a smooth Hamiltonian isotopy supported inside $ \tau(W_{k_m}(\delta)) $, of size less than $ 2 \cdot 15\epsilon_m = 30\epsilon_m $, whose time-$1$ map sends $ \psi_m' \circ i_{k_m} $ to $ i_{k_{m+1}} $. Denote by $ \psi_m'' $ the time-$1$ map of this isotopy. 
\end{proof}

\section{$\cc^0$-rigidity of codimension $2$ symplectic submanifolds}\label{sec:rigcodim2}
Theorem \ref{thm:nssymp} will follow from a more precise statement involving Hofer-Zehnder capacities.  
Before stating it, let us recall the definition of $ \pi_1 $-sensitive Hofer-Zehnder capacity~\cite{lu,schwarz,gingur,macarini,schlenk}. 
\begin{defn*}\label{def:pi1HZ} Given an open set $U$ in a symplectic manifold $M$, an autonomous Hamiltonian $H:U\to \R^+$ is admissible if it has compact support and attains its maximum on an open subset of $U$. It is said slow if it has no non-constant periodic orbit of period $T\leqslant 1$, and $c$-slow if it has no non-constant periodic orbit with period $T\leqslant 1$ which is contractible in $M$. The Hofer-Zehnder and $\pi_1$-sensitive Hofer-Zehnder capacities of $U$ are
\begin{align*}
& c_\hz (U):=\sup\{\max H \; | \; H \text{ is admissible and slow}\},\\
& c_{\hz}^\circ(U,M):=\sup\{\max H \; | \; H \text{ is admissible and $c$-slow}\}.
\end{align*}
\end{defn*}
Also, for a subset $ A \subset M $, we use the notation of $ e_d(A,M) $ for the symplectic displacement energy of $ A $ in $ M $.
\begin{theorem}\label{thm:rigcodim2}
Let $h:M\to M'$ be a symplectic homeomorphism that takes a smooth codimension $2$ contractible symplectic submanifold $N^{2n-2}\subset M^{2n}$ to a smooth symplectic submanifold $N'$. Then 
\begin{itemize}
\item[(i)] Given $ K \subset V \Subset N $, where $ K $ is closed and $ V $ is relatively open in $ N $, for $ K' = h(K) $, $ V' = h(V) $, we have $$ e_d(K \times S^1,V \times T^* S^1) \leqslant e_d(K',V') .$$
\item[(ii)] Given $ U \subset V \Subset N $, where $ U $, $ V $ are relatively open in $ N $, for $ U' = h(U) $, $ h(V) \Subset V' \Subset N' $ (where $ V' $ is a relatively open subset of $ N' $), and for any $ r > 0 $, we have $$ c_\hz^\circ(U',V') \leqslant c_\hz^\circ(U \times T_r^*S^1 ,V \times T_r^* S^1) .$$ 
\end{itemize}
Here $ S^1 \subset T^* S^1 $ is the zero section of the cotangent bundle of the circle, and by $ T_r^* S^1 = \{ (q,p) \in T^* S^1 \; | \; |p| < r \} $ we denote the $r$-tubular neighbourhood of $ S^1 \subset T^* S^1$.
\end{theorem} 
The $\pi_1$-sensitive Hofer-Zehnder capacity and the stable displacement energy can be estimated by several means. Theorem \ref{thm:rigcodim2} allows therefore various more intrinsic rigidity statements (among which theorem \ref{thm:nssymp}), and we provide some examples in section \ref{sec:corrigcodim2}.

\subsection{Proof of theorem \ref{thm:rigcodim2}}
 Let $h:M\to M'$ be a symplectic homeomorphism that takes $N$ to $N'$, where $N,N'$ are codimension $2$  contractible symplectic submanifolds.

\noindent {\bf Proof of (i):} 
We consider $K\Subset V\Subset N$ and denote by $K':=h(K)$ and $V':=h(V)$ their images. 
By the symplectic \nbd theorem, since $N$ and $N'$ are contractible, and $V\Subset N$, $V'\Subset N'$, we can slightly reduce $N$ and $N'$ if necessary, so that $V$ and $V'$ are still compactly contained in $N,N'$, and some \nbds $\cv,\cv'$ of $N,N'$ can be presented as  symplectic products  $\cv\simeq N\times D(r)$, $\cv'\simeq N'\times D(r')$. Considering smaller $r$ if necessary, and restricting $h$ to $\cv$, we can therefore assume that 
$$
h:N\times D(r)\to  N'\times D(r'), \hspace{1cm} h(N\times \{0\})\subset N'\times \{0\}.
$$
We need to show that 
$$
e_d(K',V')\geqslant e_d(K\times S^1,V\times T^*S^1).
$$
This inequality is not trivial only when $K'$ is dispaceable in $V'$, which we assume henceforth.  
Fix $\eps>0$ and let $H_1':V'\times [0,1]\to \R$ be a smooth, compactly  supported function, whose Hamiltonian flow displaces $K'$ ({\it i.e.} $\Phi^1_{H_1'}(K')\cap K'=\emptyset$),  with energy 
$$
e:=\int_0^1 \osc H_1'(\cdot,t)dt<e_d(K',V')+\eps.
$$
Let $\chi:[0,r']\to [0,1]$ be a smooth function which equals $1$ near $0$, and vanishes identically for $t>r''$, where $r''<r'$. If $r''$ is small enough, the function \fonction{H_2'}{V'\times D(r')\times [0,1]}{\R}{(x,z,t)}{\chi(|z|)H_1'(x,t)} 
has compact support in $h(V\times D(r))$, same energy as $H_1'$, and its flow preserves $V'\times \{0\}$ and induces the same flow as $H_1'$ on it. The flow of $H_2'$ therefore displaces $K'=h(K)$, and by continuity, it also displaces $h(K\times S_\delta)$ for $ 0 < \delta\ll 1$ (here $ S(\delta) = \{ z \in \mathbb{C} \; | \; |z| = \delta \} $). Moreover, since the flow of $H_2'$ preserves $V'\times \{0\}$, the set $\ds \bigcup_{t\in [0,1]} \Phi^t_{H_2'}(K\times S_\delta)$ lies in a compact set of $V'\times (D(r')\priv \{0\})$, hence in some subset of $V'\times D(r')$ of the form $\{|z|>\delta'\}$.  Considering a further cut-off $H_3'(x,z,t):=H_2'(x,z,t)\rho(|z|)$ where $\rho(t)$ vanishes near $0$ and equals $1$ for $t>\delta'$, we therefore get a Hamiltonian, with still the same energy $e$, with compact support in $h(V\times D(r))\priv (V'\times \{0\})=h(V\times (D(r)\priv \{0\}))$, whose flow displaces $h(K\times S_\delta)$ (because it coincides with the flow of $H_2'$ on this set). \\
\indent Let  $h_k:V\times D(r)\to N'\times D(r')$ be symplectic embeddings which $\cc^0$-converge to $h$ (recall that we have already shrinked $N$ and $r$, so we can assume without loss of generality that $h_k$ is defined on the whole of $V\times D(r)$). For $k$ large enough,  $\supp (H_3')\subset \im h_k$ (see the proof of proposition \ref{prop:defsympeo}), so the function defined by $H_k:=H_3'\circ h_k$ whenever defined, and $H_k\equiv 0$ elsewhere is smooth. Moreover, since $H_3'$ has compact support in $h\big(V\times (D(r)\priv \{0\})\big)$ and $(h_k^{-1})$ $\cc^0$-converges to $h^{-1}$, $H_k$ has compact support in $V\times \big(D(r)\priv \{0\}\big)$ for $k$ large enough. It obviously still has energy $e$. Finally, since the flow of $H_3'$ displaces $h(K\times S_\delta)$, it also displaces $h_k(K\times S_\delta)$ - and so the flow of $H_k$ displaces $K\times S_\delta$ - for $k$ large enough. We conclude that  $e_d\big(K\times S_\delta,V\times (D(r)\priv \{0\})\big)\leqslant e<e_d(K',V')+\eps$. Since  $D(r)\priv \{0\}$ can be embedded into $ T^* S^1 $ in such a way that $ S(\delta) $ is mapped to the zero-section, this finishes the proof.\cqfd

\noindent  
{\bf Proof of (ii):} As in the proof of (i), since $N,N'$ are contractible, we can reduce slightly $N$ and apply the symplectic \nbd theorem, so that we can assume without loss of generality that $h:N\times D(r)\to \cv'\subset N'\times D(r')$, where $\cv'\supset N'\times \{0\}$ and $h(N\times\{0\})\subset N'\times \{0\}$.
We consider $U\subset V\Subset N$,  $U'=h(U)$, $V'$ an open set in $N'$ with $h(V)\Subset V'$. 
We need to show that for all positive $r$, 
$$
c_\hz^\circ(U',V')\leqslant c_\hz^\circ(U\times T_r^*S^1,V\times T_r^*S^1). 
$$
Since the function $r\mapsto c_\hz^\circ(U\times T_r^*S^1,V\times T_r^*S^1)$ is easily seen to be non-decreasing, we may assume, without loss of generality, that $r>0$ is very small. 

 Choose any $ \epsilon > 0 $ and pick a time-independent Hamiltonian $ H' : U' \rightarrow \mathbb{R} $, which is admissible, whose flow has no non-constant periodic orbit of period $ T \leqslant 1 $ which are contractible in $ V' $, and such that $ \max H' > c_\hz^\circ(U',V') - \epsilon $. Denote  $ K' := \supp (H') \Subset U' $ the support of $H'$. Since $ h(V) \Subset V' $, we can find small $ \delta > 0 $, such that $ h(V \times D(\delta)) \Subset V' \times D(r') $. Then, we can find some small $ \delta' > 0 $ such that $ h^{-1}(K' \times D(\delta')) \Subset U \times D(\delta) $. Since $h(N\times \{0\})\subset N'\times \{0\}$, for any given  $ 0 < \delta_1' < \delta_2' \leqslant \delta' $, we can find $ 0 < \delta_1 < \delta_2 \leqslant \delta $ such that $ h^{-1}(K' \times A(\delta_1',\delta_2')) \Subset U \times A(\delta_1,\delta_2) $. Choose a smooth compactly supported radial function $ \chi : A(\delta_1',\delta_2') \rightarrow [0,1] $, which equals $ 1 $ on some open subset of $ A(\delta_1',\delta_2') $. The function \fonction{\tilde H'}{ U' \times A(\delta_1',\delta_2')}{\R}{(y,z)}{H'(y)\chi(z).} is then an admissible function on $ U' \times A(\delta_1',\delta_2') $, and it has no non-constant periodic orbits of period $ T \leqslant 1 $ which are contractible in $ V' \times A(\delta_1',\delta_2') $. Let $h_k:N\times D(r)\hra N'\times D(r')$ be symplectic embeddings which $\cc^0$-converge to $h$. For large $ k $, we have $ h_k^{-1}(K' \times A(\delta_1',\delta_2')) \Subset U \times A(\delta_1,\delta_2) $ and $ h_k(V \times A(\delta_1',\delta_2')) \Subset V' \times A(0,r') $. Now choose sufficiently large $ k $, and define a Hamiltonian function $ \tilde{H}_k : U \times A(\delta_1,\delta_2) \rightarrow \mathbb{R} $ by $ \tilde{H}_k (x) = \tilde{H}' ( h_k (x)) $ for $ x \in  h_k^{-1} (K' \times A(\delta_1',\delta_2')) $, and $ \tilde{H}_k(x) = 0 $ otherwise.  Then $ \tilde{H}_k $ is an admissible function on $ U \times A(\delta_1,\delta_2) $. We claim that $ \tilde{H}_k $ has no non-constant periodic orbits of period $ T \leqslant 1 $ which are contractible in $ V \times A(\delta_1,\delta_2) $. Indeed, since $ h_k(V \times A(\delta_1',\delta_2')) \Subset V' \times A(0,r') $, a non-constant periodic orbit for $\tilde H_k$ contractible in $V\times A(\delta_1,\delta_2)$ would give a non-constant periodic orbit of $\tilde H'$ contractible in $V'\times A(0,r')$, hence in $ V' \times A(\delta_1',\delta_2') $. 
 We conclude that $ c_\hz^\circ(U',V') - \epsilon \leqslant \max H' = \max \tilde{H}_k \leqslant c_\hz^\circ(U \times A(\delta_1,\delta_2) ,V \times A(\delta_1,\delta_2)) $. Now, $ A(\delta_1,\delta_2) $ is symplectomorphic to $ T_{\rho}^* S^1 $, where $ \rho = \pi(\delta_2^2 - \delta_1^2)/2 < \pi r^2 /2 < r $. Thus, $$ c_\hz^\circ(U',V') - \epsilon \leqslant c_\hz^\circ(U \times T_{\rho}^*S^1 ,V \times T_{\rho}^* S^1) \leqslant c_\hz^\circ(U \times T_r^*S^1 ,V \times T_r^* S^1). $$ Since this holds for every $ \epsilon > 0 $, this finishes the proof. \cqfd

  \begin{remark}
    The proof of the point (i) above shows that if $h$ sends a smooth symplectic submanifold $ N $ of dimension $ 2(n-r) $ to a smooth submanifold $N'$, and if $K\subset V\Subset N$, then for any $ \epsilon > 0 $ we have $$ e_d\big(K\times S_\delta^{2r-1},V\times (\C^r\priv\{0\})\big) < e_d(h(K),h(V)) + \epsilon, $$ when $ \delta > 0 $ is small enough (here $ S_\delta^{2r-1} = \{ x \in \mathbb{R}^{2r} \, | \,  |x| = \delta \} $ is the sphere of radius $ \delta $). 
On the other hand, theorem \ref{thm:flexcodim4} implies that $ e_d(h(K),h(V)) $ can be made arbitrarily small (for a suitable choice of $ h $), when the codimension $2r\geqslant 4$. We therefore see that 
  \begin{equation}\label{eq:vanished}
  \lim_{\delta \rightarrow 0} e_d\big(K\times S_\delta^{2r-1},V\times (\C^r\priv\{0\})\big)=0,\hspace{1cm} \forall r\geqslant 2
  \end{equation}
  which contrasts with the situation when $r=1$. The equality \textnormal{(\ref{eq:vanished}\hspace*{-,15cm})} can be also proved by a more direct alternative argument, which we do not provide here.
  \end{remark}

\subsection{Rigidity with respect to more capacities, proof of theorem \ref{thm:nssymp}.}\label{sec:corrigcodim2}
The $\cc^0$-rigidity described by theorem \ref{thm:rigcodim2} might look unclear at first sight. For instance, statement (i) tells that provided $K\subset N$ is not too small (for instance a symplectic ball of size $1$), the displacement energy of $h(K)$ is bounded from below. But it might either be because $h(K)$ is large, or because it lies in $N'$ in such an intricate position that it is hard to displace. However, the displacement energy, as well as the $\pi_1$-sensitive Hofer-Zehnder capacity, can be estimated by several means, and we get from theorem \ref{thm:rigcodim2} different possible versions of rigidity statements. 
We now illustrate some of them, with particular attention to intrinsic ones. We prove in particular theorem \ref{thm:nssymp} (see corollary \ref{cor:rigcodim2} (iii) below). We first need to introduce some definitions or notations. 

  A symplectic embedding $U\subset \C^n\overset{f}{\hra} M^{2n}$ will be said {\it undistorted}  if  $e_d(f(U),M)\leqslant e_d(U,\C^n)$. We will say that $U\subset N$ is {\it well-embedded} if it is the image of an undistorted embedding of an open set in $\C^n$.
  \begin{ex*} If $f:B(a)\hra M$ extends to a symplectic embedding of $B(2a)$ then $f$ is an undistorted embedding.
On the other hand, it is not clear, except in dimension 4 by \cite{mcduff}, that a ball $B(a)$, even contained in a larger ball $B(Ka)\subset M$, $K\gg 1$ 
is always well-embedded. In particular, for $n\geqslant 3$ it is not known whether all symplectic embeddings of $B^{2n}(\eps)$ into $B^{2n}(1)$ can be extended to symplectic embeddings of $B^{2n}(2\eps)$, even if $\eps\ll 1$. 
\end{ex*}

Recall also that a symplectic manifold with boundary $M$ is said to be {\it $\om$-convex} (in the sense of \cite{elgr}) if there is a Liouville vector field defined near $\partial M$ which points outwards at the boundary. The boundary then satisfies the so-called contact-type condition, and, what is important for us, $M$ can be endowed with compatible almost-complex structures which make $M$ pseudo-convex: no holomorphic curves with boundary in some compact subset of $M$ can approach $\partial M$. 

We finally introduce the following Lagrangian capacity. Recall that for a closed Lagrangian $L$ in a symplectic manifold $M$, the Liouville class of $L$ is $\lambda(L,M):=\inf \spec(L)\cap \R_+^*$. We define
$$
c_{\lag}(U,M):=\sup\{\lambda(L,M), L\subset U\},
$$
and denote $c_\lag(U):=c_\lag(U,U)$.

\begin{corollary}\label{cor:rigcodim2} Let $h:M\to M'$ be a symplectic homeomorphism that takes a smooth contractible symplectic hypersurface $N^{2n-2}\subset M^{2n}$ to a smooth symplectic hypersurface $N'$, and let $U \Subset N $ be an 
open set of $N$. Then,
\begin{itemize}
\item[(i)] If $N'$ is $\om$-convex, 
$
c_\hz^\circ(h(U),N')\leqslant 4e_d(U,N) \; \cite{schlenk}.
$ 
(recall that $c_\hz^\circ$ denotes the $\pi_1$-sensitive Hofer-Zehnder capacity, defined in the introduction). If $U$ is moreover well-embedded, $c_\hz(h(U))\leqslant 4e_d(U)$.
\item[(ii)] $c_\lag(h(U),N')\leqslant e_d(U,N)\; \cite{chekanov}.$
When $U$ is well-embedded and $V'\subset h(U)$ is $\om$-convex, $c_\lag(V')\leqslant e_d(U)$. 
\item[(iii)] If  $W \subset N $ is symplectomorphic to the standard ball $ B(a)  \subset \mathbb{C}^{n-1} $ of capacity $ a $, and if $h(W) \subset N' $ can be symplectically embedded into the cylinder $  \D(A) \times \mathbb{C}^{n-2} $, then $ a \leqslant A $ (theorem \ref{thm:nssymp}).   
\end{itemize}
\end{corollary}

\noindent{\it Proof:} {\bf  (i)} We invoke here a result by Schlenk \cite{schlenk}. Provided that $N$ is tame (this is the case when $N$ is $\om$-convex), we have $c^\circ_\hz(U,N)\leqslant 4 e_d(U \times S^1,N \times T^* S^1)$. In view of theorem \ref{thm:rigcodim2} (applied to the inverse $ h^{-1} : N' \rightarrow N $), we get the desired inequality
$$
c^\circ_\hz(h(U),N')\leqslant 4e_d(U,N).
$$
Now $c_\hz(h(U))\leqslant c^\circ_\hz(h(U),N')$ by definition, while for well-embedded $U$, the right hand side is bounded by $4e_d(U)$. We therefore get $c_\hz(h(U))\leqslant 4e_d(U)$.

\noindent{\bf (ii)} Denote $e:=e_d(h(U)\times S^1,N'\times T^*S^1)$, and consider a Hamiltonian $\{H_t\}$  with compact support in $N'\times T^*S^1$ with energy $e+\eps$ and which displaces $h(U)\times S^1$. Then $\{H_t\}$ displaces $L\times S^1$ for any Lagrangian $L\subset h(U)$. Chekanov's theorem, applied to the Lagrangian $L\times S^1$, implies that $\lambda(L\times S^1,N'\times T^*S^1)\leqslant e+\eps$ \cite{chekanov}. Now notice that a disc $u:D\to N'\times T^*S^1$ with boundary on $L\times S^1$ splits as $u=(u_1,u_2)$ with $u_1:(\D,\partial \D)\to (N',L)$ and $u_2:\D\to T^*S^1$. Since $T^*S^1$ is an exact symplectic manifold, the disc $\ca_{\om_0}(u_2)$ has zero area, so $\ca_{\om\oplus\om_0}(u)=\ca_{\om}(u_1)$ and we conclude that $\lambda(L\times S^1,N'\times T^*S^1)=\lambda(L,N')$. We therefore get $\lambda(L,N')\leqslant e+\eps$, for all Lagrangian $L\subset U$ and all $\eps$. This exactly means that $c_\lag(U,N')\leqslant e_d(h(U) \times S^1,N' \times T^* S^1)$. By theorem \ref{thm:rigcodim2} (applied to the inverse $ h^{-1} : N' \rightarrow N $), $ e_d (h(U) \times S^1,N' \times T^* S^1)\leqslant e_d(U,N)$, so we get the announced inequality 
$$
c_\lag(h(U),N')\leqslant e_d(U,N).
$$
When $U$ is well-embedded, the right-hand side of the inequality is bounded by $e_d(U)$. Moreover, Chekanov's result provides a {\it holomorphic} disc of area $e+\eps$, with boundary in $L$, for any tame almost complex structure on $N'$. When  $V'\subset h(U)$  is $\om$-convex, there are such almost complex structures that make $V'$ a pseudoconvex domain, so the discs with boundary inside $V'$ do not escape from $V'$. Therefore, the proof above shows that $c_\lag(V',V')=c_\lag(V')\leqslant e_d(U)$.

\noindent{\bf  (iii)} Let $ i : B(a) \hra N $ be a symplectic embedding such that $ W = i(B(a)) $. Take $ \epsilon > 0 $, and define $ U = i(B(a-\epsilon)) \Subset N $, $ V = i(B(a-\epsilon/3)) \Subset N $, $ U' = h(U) \Subset N' $ and $ V' = h(i(B(a - 2 \epsilon /3))) \Subset N' $. Take any $ r > 0 $. Applying theorem~\ref{thm:rigcodim2} (ii) to $ h^{-1} $, we conclude that $ c_\hz^\circ(U,V) \leqslant c_\hz^\circ(U' \times T_r^*S^1 ,V' \times T_r^* S^1) $. Since $U$, $U'\times T^*_rS^1$ are deformation retracts of $V$, $V'\times T_r^*S^1$ respectively, we have in fact 
$$
c^\circ_\hz(U)\leqslant c^\circ_\hz(U'\times T^*_rS^1).
$$
Now, on one hand, $c^\circ_\hz(U)=a-\eps$, because $U=i(B(a))$. On the other hand, since $U'\Subset h(W)$ and $h(W)$ can be symplectically embedded into $D(A)\times \C^{n-2}$ by assumption, $U'\times T_r^*S^1$ can be symplectically embedded into $D(A)\times D(mA)^{n-2}\times T_r^*S^1$, which embeds in turn into $X:=D(A)\times (S^2_{mA})^{n-2} \times \mathbb{T}^2_r $ (here $ S^2_A $ is the $2$-sphere of area $ A $, $ S^2_{mA} $ is the $2$-sphere of area $ mA $, and $ \mathbb{T}_r^2 $ is the $2$-torus of area $ r $). Finally, $X$ embeds into $Y:=S^2_{3A}\times (S^2_{mA})^{n-2} \times \mathbb{T}^2_r $ as a dispaceable subset whose displacement energy is at most $A$. By \cite{usher}, we therefore get that $c^\circ_\hz(X,Y)\leqslant e_d(X,Y)\leqslant A$. But again, a loop of $X$ which is contractible in $Y$ is obviously contactible in $X$, so $c^\circ_\hz(X,Y)=c^\circ_\hz(X)$. Finally,
$$
a-\eps=c^\circ_\hz(U)\leqslant c^\circ_\hz(U'\times T_r^*S^1)\leqslant c^\circ_\hz(X)=c_\hz^\circ(X,Y)\leqslant A.
$$
 Since this holds for any $ \epsilon > 0 $, we get $ a \leqslant A $.
  \cqfd

\section{$\cc^0$-rigidity of the reduction of a hypersurface}\label{sec:rigred}
\subsection{Reduction of a hypersurface}\label{sec:defred}
We first define the symplectic structure transverse to the characteristic foliation, also called the reduction of the hypersurface. 
It is a rather classical notion, but let us give a precise definition which we will use here. Our definition will be more suitable for the $ \cc^0 $ symplectic geometry, however, as we will see soon, it is equivalent to a classical one.

\begin{defn} \label{D:reduction}

Let $ \Sigma $ be a hypersurface.
\begin{enumerate}

\item
We say that an open topological submanifold$\,$\footnote{Recall that a topological submanifold of a topological manifold $ X $ is a subset $ Y \subset X $, such that there exists a topological manifold $ Z $ and a map $ i : Z \rightarrow X $ which is a homeomorphism onto the image $ i(Z) = Y $.} $ \; U^{2n-2}\subset \Sigma$ is (topologically) transverse to the characteristic foliation of $ \Sigma $ if $ U $ has a \nbd $ V\subset \Sigma$ such that $ U$ intersects exactly once each characteristics of $ V$. 

\item 
Let $ U, U' \subset \Sigma$ be $ (2n-2) $-dimensional topological submanifolds $ U, U' \subset \Sigma$, that are transverse to the characteristic foliation of $ \Sigma $. We say that $ U $ and $ U' $ are equivalent (denoting $ U\sim U'$) if there exists a (continuous) homotopy $ F : W \times [0,1] \rightarrow \Sigma $, $ t \in [0,1] $, of an open topological manifold $ W^{2n-2} $, such that $ F_{|W \times \{ 0 \}} $ is a homeomorphism onto $ U $, $ F_{|W \times \{ 1 \}} $ is a homeomorphism onto $ U' $, and  
such that for each $ x \in W $, the trajectory $ t \mapsto F(x,t) $ goes along a characteristic of $ \Sigma $.

\item
The reduction of a smooth hypersurface $\Sigma$, denoted by $\red(\Sigma)$, is defined as the set of open topological submanifolds $U^{2n-2} \subset \Sigma $ which are transverse to the characteristic foliation of $ \Sigma $, considered modulo the above equivalence relation.

\end{enumerate}
\end{defn}

\medskip

\noindent Now let us address several points:
\begin{itemize}
\item On a topological submanifold $ U \subset \Sigma $ which is transverse to the characteristic foliation, we have a natural structure of a smooth symplectic manifold. Indeed,
let $ V $ be a neighbourhood of $ U $ in $ \Sigma $ such that $ U$ intersects exactly once each characteristics of $ V$, as in definition~\ref{D:reduction}. Then any point $ z \in U $ has a neighbourhood $ U_1 \subset U $ such that $ U_1 $ lies inside a (smooth) flow-box $ \Phi : W_1 \times (0,1) \rightarrow \Sigma $, where $ \im \Phi \subset V $. Then the map $ \phi := \pi \circ \Phi^{-1} : U_1 \rightarrow W_1 $ is injective and hence, by the Invariance of Domain theorem, is a homeomorphism onto the open image $ \phi(U_1) \subset W_1 $ (here $ \pi : W_1 \times (0,1) \rightarrow W_1 $ is the natural projection). The map $ \phi $ induces natural smooth and symplectic structures on $ U_1 $.

\item If two topological submanifolds $ U, U' \subset \Sigma$ are equivalent ($ U\sim U'$), then they are symplectomorphic via a homotopy $ F : W \times [0,1] \rightarrow \Sigma $, as in definition \ref{D:reduction}. Let us describe explicitly the symplectomorphism between $ U $ and $ U' $. By continuity of $ F $ and since $ U $ is topologically transverse to the characteristic foliation of $ \Sigma $, for any point $ z \in W $ and any $ t \in [0,1] $ there exists a neighbourhood $ W_1 \Subset W $ of $ z $ such that the closure of the image $ F(W_1 \times \{ t \}) $ lies inside a (smooth) flow-box $ \Phi : W_2 \times (0,1) \rightarrow \Sigma $, and moreover the map $ \phi_t := \pi \circ \Phi^{-1} \circ F : W_1 \times \{ t \}  \rightarrow W_2 $ is injective (here $ \pi : W_1 \times (0,1) \rightarrow W_1 $ is the natural projection, as before). Then, by the Invariance of Domain theorem, $ \phi_t $ is a homeomorphism onto the open image $ W_3 := \phi_t (W_1) \subset W_2 $. This induces a symplectic structure on $ W_1 \times \{ t \} $. Moreover, since $ W_1 \Subset W $ and $ F(W_1 \times \{ t \}) \Subset \Phi (W_2 \times (0,1)) $, it follows that we also have $ F(W_1 \times \{ t' \}) \Subset \Phi (W_2 \times (0,1)) $ as well and moreover $ \phi_{t'}(z,t') = \phi_t(z,t) $ for every $ z \in W_1 $, whenever $ t' \in [0,1] $ is sufficiently close to $ t $ (here $ \phi_{t'} = \pi \circ \Phi^{-1} \circ F : W_1 \times \{ t' \}  \rightarrow W_2 $). Hence the induced symplectic structures on $ W_1 \times \{ t \} $ and on $ W_1 \times \{ t' \} $ coincide, when $ t' $ is sufficiently close to $ t $.

\item Let $ h: \Sigma \rightarrow \Sigma' $ be a homeomorphism between hypersurfaces, that preserves the characteristic foliation. Then $ h $ defines a natural map $ \hat{h} : \red(\Sigma) \rightarrow \red(\Sigma') $ by $ \hat{h}([U]) := [h(U)] \subset \Sigma' $. Clearly, the definition does not depend on the representative $ U $ of $ [U] $. Recall that by a theorem of Opshtein~\cite{opshtein}, if $ h : M \rightarrow M' $ is a symplectic homeomorphism and if $ h $ maps a smooth hypersurface $ \Sigma $ onto a smooth hypersurface $ \Sigma' $, then the restriction $ h|_\Sigma : \Sigma \rightarrow \Sigma' $ preserves the characteristic foliation, and hence we get a natural induced map $ \hat{h} : \red(\Sigma) \rightarrow \red(\Sigma') $.

\end{itemize}

\noindent The next proposition will be useful in the sequel. It implies, in particular, that any element of the reduction of a hypersurface admits a globally smooth representative. 

\begin{proposition}\label{prop:topred}
Let $U^{2n-2}\subset \Sigma$ be a topological submanifold transverse to the characteristic foliation of $\Sigma$. There exists a (smooth) flow-box 
$\Phi: W \times (-1,1)\to \Sigma$ of the characteristic foliation such that :
\begin{enumerate}
\item $U\subset \im \Phi$,
\item $\Phi^{-1}(U)=\{(z,g(z)) \, | \, z\in W\}$ for some continuous function $g:W\to (-1,1)$.
\end{enumerate}
\end{proposition}

\noindent {\it Proof:}
Let $ V \subset \Sigma$ be a \nbd of $ U$ in $\Sigma$ such that $ U$ intersects exactly once each characteristic of $ V$. After shrinking $ V $, if necessary, we may further assume that $ V $ does not contain closed characteristics. Fix a vector field $\vec R$ tangent to the characteristic distribution in $ V$. Around each $x\in U$, there exists an open neighbourhood $ V_x \subset V $ of $ x $ and adapted coordinates $\Phi_x:V_x\subset  V\to B(\eps_x)_z\times (-\delta_x,\delta_x)_t$ where  $\Phi_x$ is a diffeomorphism such that $\phi_x(0)=x$ and $\nf{\partial \Phi_x}{\partial t}=\vec R\circ \Phi_x$. For each $ x \in U $ and $ y \in V_x $, denote by $ \pi_x(y) $ the $ z $-coordinate of $ \Phi_x(y) $. Then for each $ x \in U $, the map $ \pi_x : U \cap V_x \rightarrow B(\eps_x) $ is continuous and bijective by assumption, hence it is a homeomorphism by the Invariance of Domain theorem. After shrinking each $ V_x $ if necessary (via decreasing $ \eps_x $),  we can assume that
\begin{itemize}
 \item For each $ x \in U $,  and each $y \in U \cap V_x $ we have $\Phi_x(y)=\big(\pi_x(y),g_x(\pi_x(y))\big)$,  where $g_x:B(\eps_x)\to (-\delta_x/2,\delta_x/2)$ is a continuous function. 
 \item For some sequence of points $ x_1,x_2, \ldots \in U $, the union $ \cup_{i=1}^\infty V_{x_i} $ is a locally finite covering of $ U $. 
\end{itemize}

Call $\Phi_i:=\Phi_{x_i}$, $V_i:=V_{x_i}$, $ \eps_i := \eps_{x_i} $, $\delta_i:=\delta_{x_i}$, $ \pi_i :=\pi_{x_i} $, and $ g_i := g_{x_i} $. The vector field $\vec R$ gives the $V_i$ the structure on an affine bundle, and the change of coordinates can be written: \fonction{\Phi_{ij}:=\Phi_j\circ \Phi_i^{-1}}{\Phi_i(V_i\cap V_j)}{\Phi_j(V_i\cap V_j)}{(z,t)}{(\phi_{ij}(z),t+\chi_{ij}(z)),} where $\phi_{ij}$ is a smooth diffeomorphism and $\chi_{ij}$ is smooth. Note that $\chi_{ij}$ can be expressed in terms of the $g_i,g_j$: since for $x\in U$,
$$
\begin{array}{ll}
\Phi_j(x)&=\big(\pi_j(x),g_j(\pi_j(x))\big)\\
 & =\Phi_{ij}\big(\pi_i(x),g_i(\pi_i(x))\big)=(\phi_{ij}(\pi_i(x)),g_i(\pi_i(x))+\chi_{ij}(\pi_i(x))),
\end{array}
$$
we have $\phi_{ij}=\pi_j\circ \pi_i^{-1}$ and $\chi_{ij}(\pi_i(x))=g_j(\pi_j(x))-g_i(\pi_i(x))$. The topological submanifold $ U $ inherits a smooth structure from the family of homeomorphisms $\pi_i:U_i\to B(\eps_i)$, where $ U_i := U \cap V_i $. By $ W $ we denote the corresponding smooth manifold ($ W = U $ as a set).

Choose a smooth partition of unity $ \rho_i : U \rightarrow [0,1] $, $ i = 1,2, \ldots $ on $ W $ with  $\supp \rho_i\subset U_i$. Now, for each $ i $, denote $ \delta_i' := \frac{1}{2} \min_{V_j \cap V_i \neq \emptyset} \delta_j $ and pick smooth functions $ g_i^- ,g_i^+ : B(\eps_i) \rightarrow (-\delta_i, \delta_i) $ such that $$ g_i - \delta_i' < g_i^- < g_i < g_i^+ < g_i+ \delta_i' $$ on $ B(\epsilon_i) $. Define for each $i$ the functions $ G_i^-, G_i^+ : B(\eps_i) \rightarrow (-\delta_i,\delta_i) $ by 
\begin{align*}
G_i^-(z) = \sum_j \rho_j(\pi_i^{-1}(z)) \left(g_j^-(\phi_{ij}(z)) + \chi_{ji}(\phi_{ij}(z))\right) , \\
 G_i^+(z) = \sum_j \rho_j(\pi_i^{-1}(z)) \left(g_j^+(\phi_{ij}(z)) + \chi_{ji}(\phi_{ij}(z))\right) .
 \end{align*}
Note that these sums are in fact finite, and that although $\phi_{ij}(z)$ is not defined when $\pi_i^{-1}(z)\notin U_i\cap U_j$, the sum itself is indeed well-defined because the term $\rho_j(\pi_i^{-1}(z))$ vanishes in this case. This family of functions is equivariant in the sense that $ \Phi_{ik} (z, G_i^-(z)) = (\phi_{ik}(z), G_k^-(\phi_{ik}(z))) $ 
and $ \Phi_{ik} (z, G_i^+(z)) = (\phi_{ik}(z), G_k^+(\phi_{ik}(z))) $ for each $ z \in U_i \cap U_k $ (because  $\chi_{jk}\circ \phi_{ij}=\chi_{ji}\circ\phi_{ij}+\chi_{ik}$). Because of that property, the map \fonction{\Phi}{W\times(-1,1)}{\Sigma}{(z,t)}{ \Phi_i^{-1}(z,((1-t)G_i^-(z) + (1+t)G_i^+(z))/2) \text{ if }(z,t) \in U_i \times (-1,1)} is a well defined smooth embedding. The map $ \Phi $ is the desired flow-box.\cqfd

\subsection{Proof of theorem \ref{thm:nshyp}}

Let $h\in \sympeo(M,M')$ take a smooth hypersurface $\Sigma$ to a smooth hypersurface $\Sigma'$. Let $e\simeq(B(a),\om_\st)\in \red (\Sigma)$, and fix a smooth representative $U\subset \Sigma$ of $e$. Then $h(U)$ is a topological submanifold which is transverse to the characteristic foliation, that represents $e' := \hat{h}(e) \in \red (\Sigma')$. Take $ \epsilon > 0 $ small enough, and denote by $ U_\epsilon \Subset U $ the representative of $ B(a-\epsilon) \subset B(a) $. By proposition \ref{prop:topred}, $ h(U) $ is a continuous section of a smooth flowbox. Then, by lemma \ref{le:straighthomeo} below, there exists a symplectic homeomorphism $f$ of $M'$ which fixes $\Sigma'$, such that $f \circ h(U_\epsilon)$ is a smooth symplectic submanifold of $\Sigma' $, such that $ f_{|h(U_\epsilon)} : h(U_\epsilon) \rightarrow f \circ h(U_\epsilon) $ is a symplectomorphism, and such that $[f\circ h(U_\epsilon)]= [f(U_\epsilon)] \in \red (\Sigma')$. The symplectic homeomorphism $f\circ h$ therefore takes the smooth, codimension $2$, symplectic submanifold $U_\epsilon \simeq (B(a-\epsilon),\om_\st)$ to the smooth symplectic submanifold $f\circ h(U_\epsilon)\subset M'$, which symplectically embeds into $ Z(A) $. Hence, by theorem \ref{thm:nssymp} we have $ A \geqslant a-\epsilon $. Since $ \epsilon > 0 $ is arbitrarily small, we conclude $A\geqslant a$. \cqfd

\begin{remark}
This proof shows that all the $\cc^0$-rigidity statements that hold for codimension $2$ symplectic submanifolds - in particular theorem \ref{thm:rigcodim2} and corollary \ref{cor:rigcodim2} - also hold for the reduction of a hypersurface. 
\end{remark}

\begin{lemma}\label{le:straighthomeo}
Let $\Sigma$ be a hypersurface, let $ U \Subset \tilde{U} $ be open $ (2n-2) $-dimensional topological submanifolds of $ \Sigma $, such that $ \tilde{U} $ (and hence $ U $ as well) is topologically transverse to the characteristic foliation of $ \Sigma $. Then, there exists a symplectic homeomorphism $f$, compactly supported inside an arbitrarily small \nbd $ \cv $ of the closure of $ U $ in $M$, which preserves  $ \Sigma $, such that $ f (U) $ is a smoothly embedded submanifold, and such that the restriction $ f_{|U} : U \rightarrow f(U) $ is a symplectomorphism. 
\end{lemma}
\noindent{\it Proof:} 
Before starting the proof, let us remark, that for the symplectic homeomorphism $ f $ which we will construct, not only $ f $ will preserve $ \Sigma $ and the restriction $ f_{|U} : U \rightarrow f(U) $ will be a symplectomorphism, but also $ f $ will preserve each characteristic of $ \Sigma $. 

By proposition \ref{prop:topred}, after shrinking $ \tilde{U} $ and $ \cv $, we can endow $ \cv $ with symplectic coordinates $(z,x_n,y_n)$, $ z \in \tilde{W} $, $ x_n \in (-\kappa,\kappa) $, $ y_n \in (-\delta,\delta) $, where $ \kappa > 1 $, such that $ \Sigma \cap \cv = \{(z,x_n,0) \; | \; (z,x_n)\in \tilde{W} \times (-\kappa,\kappa) \} $, and such that $ \tilde{U} = \{ (z, \tilde{F}(z), 0) \; | \; z \in \tilde{W} \} $ is the graph of a continuous function $ \tilde{F} : \tilde{W} \rightarrow (-1,1) $. Moreover, for some open $ W \Subset \tilde{W} $ we have $ U = \{ (z, \tilde{F}(z), 0) \; | \; z \in W \} $. Choose $ \epsilon > 0 $ such that $ \tilde{F}(\overline{W}) \subset (-1+ \epsilon, 1 - \epsilon) $. Using convolution and cutoff we can find a sequence of smooth functions $ \tilde{F}_k : \tilde{W} \rightarrow (-1+\epsilon,1-\epsilon) $, $ k =  0, 1,2, \ldots $, such that $ \tilde F_0 \equiv 0 $, such that all $ \tilde{F}_k $ vanish on the complement of a compact subset of $ \tilde{W} $, such that $ | \tilde{F}_{k} - \tilde{F}_{k-1} | < 1/2^k $ on $ \tilde{W} $, and such that $ \tilde{F}_k $ converges to $ \tilde{F} $ on $ W $, when $ k \rightarrow \infty $. Choose a smooth function $ u : \mathbb{R}  \rightarrow \mathbb{R} $ such that $ \supp (u) \subset (-1,1) $ and $ u = 1 $ on $ [-1+\epsilon,1-\epsilon] $. Moreover, choose a smooth function $ v : \mathbb{R} \rightarrow \mathbb{R} $ such that $ \supp (v) \subset (-\delta,\delta) $, $v(0)=0$ and $ v'(0) = 1 $. Define a family of Hamiltonian functions $ H_{kl} : M \rightarrow \mathbb{R} $, where $ k,l = 1,2,\ldots $, such that $ H_{kl} = 0 $ outside $ \cv $, and such that $ H_{kl}(z,x_n,y_n) = (\tilde{F}_k(z) - \tilde{F}_{k-1}(z) ) u(x_n) v(l y_n) / l $. The Hamiltonian vector field generated by $H_{kl}$ is given by 
$$
\begin{array}{ll}
X_{H_{kl}}(z,x_n,y_n)=\ds \big(\tilde F_k(z)-\tilde F_{k-1}(z)\big)u(x_n)v'(ly_n)\frac\partial{\partial {x_n}} & \hspace*{-,2cm}\ds - \big(\tilde F_k(z)-\tilde F_{k-1}(z)\big)u'(x_n)\frac{v(ly_n)}{l}\frac\partial{\partial {y_n}}\\
 &\hspace{-,2cm}\ds +
u(x_n)\frac{v(ly_n)}{l}X_{\tilde F_k-\tilde F_{k-1}},
\end{array}
$$
where $X_{\tilde F_k-\tilde F_{k-1}}$ is the Hamiltonian vector field on $ \tilde{W} $ generated by the function $\tilde F_k-\tilde F_{k-1}$. Note that $\Phi^t_{H_{kl}}$ has compact support in $\tilde W\times (-1,1)\times (-\nf \delta l,\nf \delta l) $, preserves $\Sigma$, and that 
$$
\Phi^t_{H_{kl}}(z,\tilde F_k(z),0)=(z,\tilde F_{k+1}(z),0) \hspace{,5cm} \forall z\in \tilde W. 
$$
Note also that for a fixed $k$, there exists $l_k^0$ such that for $l \geqslant l_k^0$, the $ \cc^0 $ norm of $X_{H_{kl}}$ is bounded from above by $ c / 2^k $, so the $\cc^0$ distance from the time-$ 1 $ map $\Phi^1_{H_{kl}}$ to the identity is bounded by $ C / 2^k $. Choose finally $l_k\geq l_k^0$ inductively, so that $\phi_k:=\Phi^1_{H_{kl_k}}$ has support in $\phi_{k-1}\circ \dots\circ \phi_1(\tilde W\times(-1,1)\times(-\nf 1k,\nf 1k ))\supset \Sigma$, and define $\psi_k:=\phi_k\circ\dots\circ \phi_1$. Then, $(\psi_k)$ $\cc^0$-converges to a map $\psi$ which is the identity outside $\tilde W\times (-1,1)\times(-\eps,\eps)$, fixes $\Sigma$, and verifies $\psi(z,0,0)=(z,\tilde F(z),0)$ on $W$. It remains to show that $\psi$ is injective to get that $\psi$ is a symplectic homeomorphism, and to put $f=\psi^{-1}$. Observe that for two points $p\neq q$ not in $\Sigma$, we have $p,q\notin \{x_n\in (-\nf 1k,\nf 1k)\}$ for $k$ large enough, so $\psi(p)=\psi_k(p)\neq \psi_k(q)=\psi(q)$.  Similarly, if $p\notin \Sigma$, and $q\in \Sigma$, we have $\psi(p)=\psi_k(p)\notin \Sigma$ and $\psi(q)\in \Sigma$, so $\psi(p)\neq \psi(q)$. Also, since $H_{kl}\equiv 0$ on $\Sigma$, the $\psi_k$ preserve each characteristic of $\Sigma$, so this still holds for $\psi$. Therefore, if $p$ and $q$ belong to different characteristics (in coordinates, this means $z(p)\neq z(q)$), we have $\psi(p)\neq \psi(q)$. Finally, we have to prove that when restricted to each characteristics, $\psi$ is injective. Putting $\tau_k(z):=\tilde F_k(z)-\tilde F_{k-1}(z)$, we can write 
$$
X_{H_{kl}}(z,x_n,0)=\tau_k(z)u(x_n)\frac \partial{\partial x_n}. 
$$
Thus fixing $z$, we have 
$$
\begin{array}{ll}
\psi(z,x_n,0)& =\lim \psi_k(z,x_n,0)\\
& \ds=\lim \Phi^1_{\tau_k(z)u(x_n)\frac\partial{\partial {x_n}}}\circ\dots\circ  \Phi^1_{\tau_1(z)u(x_n)\frac\partial{\partial {x_n}}}(z,x_n,0)\\
 &\ds= \lim \Phi^{\tau_k(z)}_{u(x_n)\frac\partial{\partial {x_n}}}\circ\dots\circ  \Phi^{\tau_1(z)}_{u(x_n)\frac\partial{\partial {x_n}}} (z,x_n,0)\\
  & \ds=\lim \Phi^{\tau_k(z)+\dots +\tau_1(z)}_{u(x_n)\frac\partial{\partial {x_n}}} (z,x_n,0)\\
  & \ds= \Phi^{\sum \tau_i(z)}_{u(x_n)\frac\partial{\partial {x_n}}}(z,x_n,0).\\
  \end{array}
$$
The third equality above holds because $z$ is constant along the vector fields $\tau_k(z)u(x_n)\frac\partial{\partial {x_n}}$. Since the series $\sum \tau_k(z) $ converges, we see that the restriction of $\psi$ to each characteristic is the value of the flow of the time-independent vector field $u(x_n)\frac\partial{\partial {x_n}}$ for some finite time.  The restrictions of $\psi$ to the characteristics of $\Sigma$ are therefore injective, which conclude our proof.
\cqfd

This lemma has another noteworthy corollary:
\begin{corollary}\label{cor:symptononsymp}
There exists a symplectic homeomorphism $h$ with support in an arbitrary \nbd of $D:=\D^{n-1}\times \{0\}\subset \C^n$ such that $h(D)$ is a smooth, non-symplectic, submanifold. 
\end{corollary}
\noindent {\it Proof:} Consider a \nbd $\cv$ of $D$ in $\C^n$, and let $\Sigma:=\cv\cap(\D^{n-1}\times \R)\subset \C^n$. Choose a compactly supported continuous function $f:\D^{n-1}\to \R$ with smooth graph, infinite partial derivatives at some points, but $\cc^0$-small enough so that its graph lies in $\Sigma$. By lemma \ref{le:straighthomeo}, there exists a symplectic homeomorphism $h$ with support in $\cv$, such that $f(D)=\Graph(f)\subset \Sigma$. Now at each point where $f$ has infinite derivative, $\Graph(f)$ is tangent to the characteristic direction $\nf \partial{\partial x_n}$, so is not symplectic.\cqfd

\section{$\cc^0$-rigidity of the spectrum of a Lagrangian}\label{sec:rigspec}
This section is devoted to the proof of theorem \ref{thm:rigspec}. We begin with a lemma that turns the problem into a question on the Liouville class of a Lagrangian submanifold in a cotangent. Recall that for a Lagrangian submanifold $L\subset (M,\om)$, we denote by $\ca^L_\om:H_2(M,L)\to \R$ the area homomorphism (see p. \pageref{def:areahom}).
\begin{lemma}\label{le:rigspecotan}
If $h:M\to M'$ is a symplectic homeomorphism with $h(L)=L'$, where $L,L'$ are two smooth Lagrangians, and if $\ca_\om^L\neq h^*\ca_{\om'}^{L'}$, then there is a sequence of Lagrangian embeddings $i_k:L\hra T^* L'$ with the following properties:
\begin{itemize}
\item[\sbull] $i_k\overset{\cc^0}{\lra}\iota$, where $\iota$ is a topological embedding of $L$ whose image is the zero-section in $T^* L'$,
\item[\sbull] There is a class $[\gamma]\in H_1(L)$ such $i_k^*\lambda_\can[\gamma]\nrightarrow 0$.
\end{itemize}
\end{lemma}
\noindent {\it Proof:} By Weinstein's \nbd theorem, there is a symplectomorphism $\Phi$ between a \nbd $\cu$ of $L'$ in $M'$ and a \nbd $\cv$ of the zero section in $T^*L'$ with $\Phi(L')=O_{L'}$ (the zero section). If $h_k:U_k\to M'$ is a sequence of symplectic embeddings which $\cc^0$-converges to $h$, then 
$h_k(L)\subset \cu$ for $k$ large enough, so $i_k:=\Phi\circ {h_k}_{|L}$ is a Lagrangian embedding of $L$ into $T^*L'$ and the sequnce  $(i_k)$ $ \cc^0 $-converges to $\iota:=\Phi\circ h$.

Let $(\Sigma,\partial \Sigma)\to (M,L)$ be a surface with boundary on $L$, and let $ k \in \mathbb{N} $ be a large enough index. The image $h_k(\Sigma)\subset M'$ is a smooth surface with boundary $h_k(\partial \Sigma)\subset h_k(L)$. The image $h_k(\partial \Sigma)$ is $\cc^0$-close to $h(\partial \Sigma)$, which is a finite union of continuous curves in $L'$. We can therefore find a union of thin annuli $A_k\subset \cu$ with smooth boundaries, connecting $h_k(\partial \Sigma)$ and a smooth perturbation of $h(\partial \Sigma)$ in $L'$. Then, $\Sigma'_k:=h_k(\Sigma)\cup A_k$ is a surface in $M'$ with boundary  in $L'$, which represents $[h(\Sigma)]\in H_2(M',L')$. Moreover, 
$$
\ca_{\om'}(A_k)=\ca_{\om'}(\Sigma_k')-\ca_{\om'}(h_k(\Sigma))=\ca_{\om'}(\Sigma'_k)-\ca_\om(\Sigma)=[\om']([h(\Sigma)])-[\om]([\Sigma]).
$$
On the other hand, since $A_k\subset \cu$, $\Phi(A_k)$ provides a finite union of smooth annuli with one side on $\Phi\circ h_k(\partial \Sigma)=i_k(\partial \Sigma)$ and another side on $0_{L'}$. Applying Stokes theorem, we therefore finally get that if $\ca_\om^L\neq h^*\ca_{\om'}^{L'}$, and $[\Sigma]\in H_2(M,L)$ is such that $[\om]\cdot[\Sigma]\neq[\om']\cdot [h(\Sigma)]$ then for large $ k $, $i_k(\partial \Sigma)$ verifies:

$$ \int_{i_k(\partial \Sigma)}\lambda_\can=\ca_{\om'}(A_k)
\nrightarrow 0. $$ Now put $ \gamma := \partial \Sigma $. \cqfd 

Notice that in the previous lemma, the manifolds $L$ and $L'$ are homeomorphic, but not necessarily diffeomorphic.  In a view towards theorem \ref{thm:rigspec}, we focus now on tori, which might admit different smooth structures (see {\it e.g.} \cite{hssh}). 

\begin{proposition}\label{prop:torrig}
Let $L$ be a smooth submanifold homeomorphic to a torus $\T^n$, and  
$i_k: L\to T^*\T^n$ be a sequence of Lagrangian embeddings which $\cc^0$-converge to  
a topological embedding $\iota$ of $L$ whose image is 
the zero-section in $T^* \T^n$. Then, $[i_k^*\lambda_\can]\to 0$.
\end{proposition}
\noindent{\it Proof:} We argue by contradiction. Let 
$a_k:= [i_k^* \lambda_\can] \in H^1(L)$ and assume that $a_k\nrightarrow 0$. After passing to a subsequence if necessary, we can assume that the distance from each $a_k$ to $0\in H^1(\T^n)$ is bounded from below. If $ \pi : T^* \T^n \rightarrow \T^n $ denotes the standard projection onto the base, $\pi\circ\iota:L\to \T^n$ is a homeomorphism, which induces a linear isomorphism $(\pi\circ \iota)^*$ in cohomology. Thus, $ b_k := (\pi \circ \iota)_* a_k \in H^1(L') $ also remains at bounded distance away from $0$. Represent now the class $b_k$ by a constant $1$-form $\vartheta_k$ on $\T^n$. This is a closed form, whose norm $\Vert \vartheta_k(q)\Vert\geqslant \eps_0$ (for any fixed norm on the space of constant $ 1 $-forms). Define also the $1$-form $ \hat{\vartheta}_k := \pi^* \vartheta_k $ on $ T^* \T^n $. We have $ [i_k^* \hat{\vartheta}_k] = [i_k^* \pi^* \vartheta_k] = [(\pi \circ i_k)^* \vartheta_k] =  a_k $ in the cohomology $ H^1(\T^n) $, since the map $ \pi \circ i_k : L \rightarrow \T^n $ is $ \cc^0 $-close,  hence homotopic,  to $\iota$ for large $ k $. The shifted Lagrangian 
$$
L_k:=i_k(L)-\vartheta_k:=\{(q,p-\vartheta_k), \; (q,p)\in i_k(\T^n)\}
$$  
is 
$\cc^0$-close to 
the graph of $\vartheta_k$, hence disjoint from $0_{\T^n}$ for $k$ large enough. In addition, 
 $L_k$ is an exact Lagrangian embedding of $L$ into $T^*\T^n$, because $[(i_k-\vartheta_k)^*\lambda_\can]=[i_k^*(\lambda_\can - \hat{\vartheta}_k)] = a_k - a_k =0$ (the first equality holds by definition of $\lambda_\can$). However, by Gromov's theorem \cite{gromov}, a closed exact Lagrangian submanifold of a cotangent bundle must intersect the zero section.\cqfd

\noindent{\it Proof of theorem \ref{thm:rigspec}:} Let $L\subset M^{2n}$ be a Lagrangian torus (hence diffeomorphic to a standard torus $\T^n$), $h:M\to M'$ be a symplectic homeomorphism such that $L':=h(L)$ is smooth. By \cite{lasi}, $L'$ is a Lagrangian submanifold in $M'$, homeomorphic to $\T^n$. Assume by contradiction that $\ca_\om^L\neq h^*\ca_{\om'}^{L'}$, and consider $h^{-1}$ instead of $h$. By lemma \ref{le:rigspecotan}, we get a sequence of Lagrangian embeddings $i_k:L'\hra T^*L = T^*\T^n$ which $\cc^0$-converge to a topological embedding $\iota:L'\hra T^*\T^n$ with $\iota(L')=O_{\T^n}$, and such that $[i_k^*\lambda_\can]\nrightarrow 0$. But this contradicts proposition \ref{prop:torrig}. \cqfd

We conjecture in fact :
\begin{conj} \label{conj:lagr-spec-strong}
If $ i_k : L \hookrightarrow T^* L' $, $ k = 1,2,\ldots $ is a sequence of smooth Lagrangian embeddings which $ \mathcal{C}^0 $-converges to a topological embedding  $ \iota :L \hookrightarrow T^* L' $ with $\iota(L)=0_{L'}$, then $ [i_k^* \lambda_{\can}] \rightarrow 0 $ in $ H^1(L,\mathbb{R}) $. 

\end{conj}

\section{Relative Eliashberg-Gromov's theorem for coisotropic and symplectic submanifolds}\label{sec:elgr}

 \subsection{Proof of proposition~\ref{prop:eli-grom-symp-codim2}}
 Let $N\subset M$ be a symplectic submanifold of codimension $ 2 $, and $h:M\to M'$ be a symplectic homeomorphism whose restriction to $N$ is a smooth diffeomorphism onto a smooth submanifold $N'=h(N)$. Notice that since $N'$ has codimension $2$, given a point $x\in N'$, $T_xN'$ is either symplectic or coisotropic. Indeed, $(T_xN')^{\perp \om}$ has dimension $2$, and since $2n-2$ is even, $\ker \om'_{|T_xN'}$ is even-dimensional. We therefore conclude that $T_xN'\cap (T_xN')^{\perp \om}$ has dimension $0$ or $2$, which correspond to the symplectic or coisotropic case, respectively. By \cite{hulese}, the set of coisotropic points must have empty interior. Thus $N'$ is a disjoint union of an open dense set $N'_\om$ which is a symplectic submanifold, and a closed set with empty interior $N'_0$. Call $N_\om:=h^{-1}(N_\om')$ the preimage of the nice part of $N'$.

We first look at the restriction of $h$ to the symplectic submanifold $N_\om$. The symplectic homeomorphism $h$ restricts to a diffeomorphism between symplectic codimension $2$ submanifolds. By theorem \ref{thm:nssymp}, if $B\subset N_\om$ is a symplectic ball of size $a$ and $h(B)\subset N_\om'$ can be embedded into $Z(A):=\D(A)\times\C^{n-2}$, then $A\geqslant a$. By the proof given in \cite[Section 2.2, Theorem 3]{hoze}, such a diffeomorphism must satisfy $d_xf^*\om'_{f(x)}=\lambda_x\om_x$ for all $x\in N_\om$. Arguing as in this proof, we consider $F:=f\times \id:M\times \C\to M'\times \C$. This is still a symplectic homeomorphism, whose restriction to $N\times \C$ is a diffeomorphism. Thus on one hand, $d_{(x,z)}F^*(\om'\oplus\om_\st)=\Lambda_{(x,z)}(\om\oplus\om_\st)$ by the previous analysis, and $d_{(x,z)}F^*(\om'\oplus\om_\st)=(\lambda_x\om)\oplus\om_\st$  on the other hand. We conclude therefore that $\lambda_x=1$, so $df^*\om'=\om$ on $N_\om$. 
But since $N_\om$ is dense in $N$ and $f$ is assumed to be smooth, the differential equality $df^*\om'=\om$ extends through $N_0'$, so $f_{|N}$ is symplectic. \cqfd

 \begin{remark} 
 We have worked under the assumption that $f_{|N}$ is a diffeomorphism by convenience. Following the proof in \cite{hoze}, the same conclusion can be obtained under the weaker assumption that $f_{|N}$ is smooth. However, it is not clear whether the smoothness assumption can be replaced by differentiability only.
 \end{remark}

\subsection{Proof of proposition~\ref{prop:eli-grom-coisotrop}}
The proof of the rigidity of coisotropic submanifolds by Humili\`ere-Leclercq-Seyfaddini appeals to the $ \cc^0 $-dynamical properties of coisotropic submanifolds. More concretely, the following uniqueness theorem of Humili\`ere-Leclercq-Seyfaddini  plays an important role in the proof (we have slightly changed the presentation of the statement):

\begin{thm'}[\cite{hulese}] \label{thm:HLS-uniqueness}
Let $ C $ be a connected coisotropic submanifold of $M$, and let $ \phi^t $ be a continuous Hamiltonian flow, generated by a continuous Hamiltonian $ H $. Then the restriction of $ H $ to $ C $ is a function of time if and only if the flow $ \phi^t $ locally preserves $ C $ and locally flows along the leaves of its characteristic foliation. The meaning of the second condition is that for any $ x \in C $ and $ t_0 $ there exists $ \epsilon > 0 $ such that for 
$ t_0 < t < t_0 + \epsilon $ we have $ \phi^t \circ (\phi^{t_0})^{-1} (x) \in C $, and moreover, $ \phi^t \circ (\phi^{t_0})^{-1} (x) $ lies in the same leaf of the characteristic foliation of $ C $, as the point $ x $.
\end{thm'}

As we prove now, this result implies proposition \ref{prop:eli-grom-coisotrop}. Let $h:M\to M'$ be a symplectic homeomorphism that takes some coisotropic submanifold $C$ to a smooth submanifold $C'$, and such that $h_{|C}$ is a smooth diffeomorphism. By \cite{hulese} we know that $C'$ is coisotropic, and we aim at proving that $h_{|C}$ is symplectic, {\it i.e.} $h^*\om'_{|C'}=\om_{|C}$. We argue by contradiction and assume that $h^*\om'_{|C'}(h(x))\neq \om_{|C}(x)$. Under this assumption, we claim the following: 
\begin{claim}\label{cl:eli-grom-coiso}
There exists a smooth function $F$ defined on a \nbd of $x$, and a smooth extension $G$ of $h_*F_{|C'}$ defined on a \nbd of $h(x)$ with the following properties:
\begin{itemize}
\item[(i)] $F_{|C}$ is constant along the characteristic leaves. Hence $\Phi_F^t$ preserves $C$.
\item[(ii)] $\Phi^t_G$ preserves $C'$, and its action on the reduction of $C$ differs from the action of $h_*\Phi^t_F$. 
\end{itemize}
\end{claim}
The path $ \Phi_K^t := (\Phi^t_G)^{-1}\circ h_*\Phi_F^t=(\Phi^t_G)^{-1}\circ \Phi^t_{h_*F}$ is then a continuous Hamiltonian path in the sense of M\"uller and Oh, generated by the function $K=(h_*F-G)(t,\Phi_G^t)$ \cite{muoh}. Since $\Phi^t_G$ preserves $C'$ (locally around $h(x)$) and $G_{|C'}\equiv h_*F_{|C'}$, $K$ vanishes on $C$. On the other hand, $\Phi_K^t $ does not act trivially on the reduction by claim \ref{cl:eli-grom-coiso}, which means that there is $y$ close to $x$ whose trajectory along $\Phi_K^t$ does not remain in a fixed characteristics. But this is a contradiction with theorem' \cite{hulese} cited above. \cqfd

\noindent {\it Proof of claim \ref{cl:eli-grom-coiso}:} Considering the symplectic homeomorphism $h$ in local symplectic charts near $x,h(x)$ adapted to the coisotropic submanifolds $C,C'$, we can assume that $x=0\in \C^n$, $h$ is defined in a \nbd $U$ of $0$, $h(0)=0$, $h(C_0\cap U)\subset C_0$, where 
$$
C_0:=\{(y_{m+1},\dots,y_n)=0\}\subset \C^n
$$
(here $\C^n$ is endowed with coordinates $z_i=x_i+iy_i$).
Since $h$ preserves $C_0$, \cite{hulese} ensures that it also preserves its characteristic foliation, which is tangent to $\ker \om_{|C_0}=\langle \nf{\partial}{\partial x_{m+1}}\dots,\nf{\partial}{\partial x_n}\rangle$. Thus, $h_{|C_0}$ can be written
$$
h_{|C_0}(z_1,\dots,z_m,x_{m+1},\dots,x_n)=(\hat h(z_1,\dots,z_m),\phi(z_1,\dots,z_m,x_{m+1},\dots,x_n)), 
$$
where $\hat h:\C^m\to \C^m$ and $\phi:C_0\to \R^{n-m}$. Since $h$ is a smooth diffeomorphism on $C_0$, $\hat h$ is a smooth diffeomorphism on a \nbd of $0$  in $\C^m$, and the additional assumption that $h^*\om_{|C_0}(0)\neq \om_{|C_0}(0)$   exactly means that $\hat h$ is not symplectic. 

Consider now a smooth function $f:\C^m\to \R$ and define $F:\C^n\to \R$ by $F(z_1,\dots,z_n)=f(z_1,\dots,z_m)$ (thus $F$ automatically verifies (i)). The push-forward $h_*F$, defined on a \nbd of $0$, is smooth on $C_0$ and its restriction to $C_0$ verifies 
$$
h_*F_{|C_0}(z_1,\dots,z_m,x_{m+1},\dots,x_n)=\hat h_*f(z_1,\dots,z_m). 
$$
Finally, define $G$ on a \nbd of $0$ in $\C^n$ by $G(z_1,\dots,z_n)=\hat h_*f(z_1,\dots,z_m)$. Then $G$ is smooth, extends $h_*F_{|C_0}$, and its flow obviously preserves $C_0$. All this construction depends on the choice of the function $f$ only, and it remains to show that for some $f$, the actions of $\Phi_G^t$ and $h^*\Phi_F^t=\Phi_{h_*F}^t$ on the characteristics of $C_0$ differ. Since $G(z_1,\dots,z_m)=\hat h_*f(z_1,\dots,z_m)$, we have 
\begin{equation}\label{eq:eligrom1}
\Phi_G^t(z_1,\dots,z_n)=(\Phi^t_{\hat h_*f}(z_1,\dots,z_m),z_{m+1},\dots,z_n).
\end{equation}
Similarly $\Phi_F^t(z_1,\dots,z_n)=(\Phi^t_f(z_1,\dots,z_m),z_{m+1},\dots,z_n)$, and 
\begin{equation}\label{eq:eligrom2}
\hspace*{-,3cm}h_*\Phi_F^t(z_1,\dots,z_m,x_{m+1},\dots,x_n)=(\hat h_*\Phi^t_f(z_1,\dots,z_m),\phi(\Phi_f^t(z_1,\dots,z_m),x_{m+1},\dots,x_n)).
\end{equation}
Since the characteristics of $C_0$ are parametrized by the first $m$ complex coordinates $(z_1,\dots,z_m)$, we see by $(\ref{eq:eligrom1})$ and $(\ref{eq:eligrom2})$ that the action of $G$ and $h_*F$ on the characteristics differ if and only if 
$\Phi_{\hat h_*f}^t\neq \hat h_*\Phi^t_f$. Since $\hat h$ is not symplectic, the following well-known lemma concludes the proof.\cqfd
\begin{lemma}
A smooth diffeomorphism $\psi:M\to M'$ is symplectic if and only if $\psi_*\Phi_H^t=\Phi^t_{\psi_*H}$ for all autonomous Hamiltonian  $H:M\to \R$. 
\end{lemma}

\appendix
\section{Main lemmata for theorem \ref{T:C0-flexibility}} \label{sec:append-flex}

The appendix is dedicated for the proof of the following three lemmata which were used in the proofs of theorem \ref{T:C0-flexibility}:

\begin{lemma} \label{L:homotopy-to-isotopy}
Let $ \epsilon > 0 $ be a positive real, $ m \geqslant 6 $ be an integer, $ W \subset \mathbb{R}^m $ be an open set, $ u_1, u_2 : \overline{D} \rightarrow W $ be disjoint smoothly embedded discs, and 
assume that there exists a (continuous) homotopy between $ u_1 $ and $ u_2 $ in $ W $, of size less than $ \epsilon $ (i.e. a continuous map $ F : \overline{D} \times [0,1] \rightarrow W $ such that 
$ F(z,0) = u_1(z) $, $ F(z,1) = u_2(z) $, for all $ z \in \overline{D} $, and that $ \size F < \epsilon $). Then there exists a  smooth embedded isotopy $ \widetilde{F} $ between $ u_1 $ and $ u_2 $ in $ W $, of size less than $ 2 \epsilon $ (i.e. a smooth embedding $ \widetilde{F} : \overline{D} \times [0,1] \rightarrow W $, such that $ \widetilde{F}(z,0) = u_1(z) $, $ \widetilde{F}(z,1) = u_2(z) $, for all $ z \in \overline{D} $, and $ \size \widetilde{F} < 2 \epsilon $). Moreover, if $ m > 6 $, then the estimate on the size of the isotopy can be improved to $ \size \widetilde{F} < \epsilon $.
\end{lemma}
\noindent{\it Proof:} First, we can slightly perturb our homotopy $ F $ between $ u_1 $ homotopy and $ u_2 $, to obtain a smooth homotopy between $ u_1 $ and $ u_2 $ in $ W $, of size less than $ \epsilon $. Hence without loss of generality we may assume that the homotopy $ F $ is smooth.

In case of $ m > 6 $, $ F $ can be further perturbed to a smooth embedding, which is an isotopy between $ u_1 $ and $ u_2 $ of size less than $ \epsilon $. This shows the case $ m > 6 $.

Now consider the case $ m = 6 $. Here we can slightly perturb $ F $ to obtain a smooth immersion $ F : \overline{D} \times [0,1] \rightarrow W $, having only a finite number of double-points which appear at transverse self-intersections, of size less than $ \epsilon $, such that the double-points do not lie on the image of the boundary of $ \overline{D} \times [0,1] $. Applying one more $ \mathcal{C}^0 $-small perturbation to $ F $, we may further assume that for any double-point $ x = F(z_1,t_1) = F(z_2,t_2) $ (here $ z_1,z_2 \in D $, $ t_1,t_2 \in (0,1) $), we have $ z_1 \neq z_2 $. Fix such a double point. Then there exist smooth coordinates $ (z,t,y) $ in a neighbourhood $ V $ of $ F(\{ z_1 \} \times [0,1]) $, where $$ z \in D(\delta) = \{ z \in \mathbb{R}^2 \, | \, | z | < \delta \}, \, t \in (-\delta,1+\delta) , \, y \in D(\delta) \times (-\delta,\delta) \subset \mathbb{R}^3 ,$$ and $ \delta < 1 - |z_1|, 1 - |z_2 | $, such that $$ V \cap F(\overline{D} \times [0,1]) = F((D(\delta) + z_1) \times [0,1]) \cup F((D(\delta) + z_2) \times (t_2-\delta,t_2 + \delta)) , $$ and such that in these coordinates we have $ F(z,t) = (z-z_1,t,0,0,0) $ for $ (z,t) \in (D(\delta) + z_1) \times [0,1] $ and $ F(z,t) = (0,0,t_1,z-z_2,t-t_2) $ for $ (z,t) \in (D(\delta) + z_2) \times (t_2-\delta,t_2+\delta) $. Now pick a smooth function $ c : [0, \infty) \rightarrow [0,1] $ such that $ c(t) = 1 $ for small $ t $, and such that $ c(t) = 0 $ for $ t \geqslant 1 $. Then choose a small enough $ \eta > 0 $, and 
modify the immersion $ F $ on $ (D(\delta) + z_2) \times (t_2-\delta,t_2+\delta) $ according to: 
$$ F(z,t) := (0,0,t_1 + (1+\eta - t_1)c(|z-z_2|/\eta)c(|t-t_2|/\eta),z-z_2,t-t_2), $$ for $ (z,t) \in (D(\delta) + z_2) \times (t_2-\delta,t_2+\delta) $. After changing $ F $ as described above, for each double-point of $ F $, we get a smooth embedded isotopy between $ u_1 $ and $ u_2 $, of size less than $ 2 \epsilon $.\cqfd

\begin{lemma} \label{L:isotopy-closed-discs}
Let $ (M,\omega) $ be a connected symplectic manifold, let $ r > 0 $, $ G = D(r) \subset \mathbb{C} $, and let $ v_1,v_2 : \overline{G} \rightarrow M $ be smoothly embedded symplectic discs, $ v_1^* \omega = v_2^* \omega = \omega_{\st} $. Then there exists a compactly supported Hamiltonian isotopy of $ M $, whose time-1 map $ \psi $ satisfies $ \psi \circ v_1 = v_2 $.
\end{lemma}

\begin{lemma} \label{L:isotopy-nbd-edge}
Let $ W \subset \mathbb{C}^n $ be an open subset, diffeomorphic to a ball, endowed with the standard symplectic structure $ \omega_{\st} $.
\begin{description}
\item[(a) The absolute case:] Let $ \gamma : [0,1] \rightarrow \mathbb{C} $ be a smooth embedded curve, let $ z_1 = \gamma(0) $, $ z_2 = \gamma(1) $, and let $ G \supset \gamma([0,1]) $ be an open neighbourhood, $ G \subset \mathbb{C} $. Assume that $ v_1,v_2 : G \rightarrow W $ are smooth symplectic embeddings, $ v_1^* \omega_{\st} = v_2^* \omega_{\st} = \omega_{\st} $, such that $ v_1 $ and $ v_2 $ coincide on a neighbourhood of $ \{ z_1,z_2 \} \subset G $. Assume moreover that for 
a $1$-form $\lambda$ which is a primitive of $\om_\st$, i.e. $ d\lambda = \omega_{\st} $, we have $ \int_{v_1 \circ \gamma} \lambda = \int_{v_2 \circ \gamma} \lambda $. Then there exists a compactly supported Hamiltonian function $ H : W \times [0,1] \rightarrow \mathbb{R} $, such that on a neighbourhood of $ \{ v_1(z_1),v_1(z_2) \} $ we have that $ H(\cdot,t) = 0 $ for every $ t $, and such that for the time-1 map $ \psi $ of the flow of $ H $, we have that $ \psi \circ v_1 = v_2 $ on a neighbourhood of $ \gamma([0,1]) $. 
 \item[(b) The proper case:] Let now $\gamma:[0,1]\to \C$ be a smooth embedded curve, $G$ be an open \nbd of $\gamma((0,1))$, and $v_1,v_2:G\to W$ be two smooth symplectic embeddings which coincide on the intersection of $ G $ with a \nbd of $\{\gamma(0),\gamma(1)\}$ in $\C$, such that moreover the curve $ (0,1) \rightarrow W $, $ s \mapsto v_1 \circ \gamma (s) $ is properly embedded. We also assume that the actions of $v_1\circ \gamma$ and $v_2\circ \gamma$ are equal:
 $$
 \int_\delta^{1-\delta}\lambda(\dot{\overline{v_1\circ \gamma}}(s))ds= \int_\delta^{1-\delta}\lambda(\dot{\overline{v_2\circ \gamma}}(s))ds, 
 $$
 for all $ 0 < \delta\ll 1$. Then there exists a compactly supported Hamiltonian $H:W\times[0,1]\to \R$ whose flow verifies $\psi^1_H\circ v_1=v_2$ on a \nbd of $\gamma((0,1))$ in $\C$. 

\end{description}

\end{lemma}

\begin{remark}\label{rk:balltononball} As will be apparent from the proof of this lemma, the hypothesis that $W$ is diffeomorphic to a ball can be replaced by the assumption that there exists a regular relative homotopy between $u_1\circ \gamma$ and $u_2\circ \gamma$ inside $W$, {\it i.e.} a smooth map $F:(0,1)\times [0,1]\to W$ such that $F(s,0)=u_1\circ \gamma(s)$, $F(s,1)=u_2\circ \gamma(s)$, $F(s,t)=u_1\circ \gamma(s)=u_2\circ\gamma(s)$ for all $t\in[0,1]$ and $s\approx 0,1$, and $s\mapsto F(s,t)$ is embedded for all $t$. In fact, if $n\geqslant 2$ (in dimension at least $4$), this last condition  is not a restriction since it can be achieved by perturbing any relative homotopy between $u_1\circ \gamma$ and $u_2\circ \gamma$. 
\end{remark}

\begin{lemma} \label{L:isotopy-discs-fixed-bdry}
Let $ n \geqslant 3 $, let $ W \subset \mathbb{C}^n $ be an open subset, diffeomorphic to a ball, endowed with the standard symplectic structure $ \omega_{\st} $, let $  r > 0 $ and $ G = D(r) \subset \mathbb{C} $, and let $ v_1,v_2 : \overline{G} \rightarrow W $ be smoothly embedded symplectic discs, $ v_1^* \omega_{\st} = v_2^* \omega_{\st} = \omega_{\st} $, which coincide on a neighbourhood of the boundary $ \partial G \subset \overline{G} $. Then there exists a compactly supported Hamiltonian function $ H : W \times [0,1] \rightarrow \mathbb{R} $, such that on a neighbourhood of $ v_1(\partial G) $ we have $ H(\cdot,t) = 0 $ for every $ t $, and such that for the time-1 map $ \psi $ of the flow of $ H $, we have $ \psi \circ v_1 = v_2 $ on $ \overline{G} $. In other words, $ v_1 $ can be Hamiltonianly isotopped to $ v_2 $ inside $ W $, while being kept fixed near the boundary.   
\end{lemma}

Although the lemmata \ref{L:isotopy-closed-discs}, \ref{L:isotopy-nbd-edge}, \ref{L:isotopy-discs-fixed-bdry} can be found in \cite{gromov2,elmi}, we present in this appendix their proofs, for the sake of completeness. The proofs of these lemmata are based on a number of auxiliary lemmata which we describe in section~\ref{SubS:Auxiliary-lemmata}. Let us remark that the proofs in the appendix are not always fully complete, and some of the technical details are left to the reader.

\subsection{Auxiliary lemmata and their proofs} \label{SubS:Auxiliary-lemmata}

The proofs of lemmata~\ref{L:isotopy-closed-discs},~\ref{L:isotopy-nbd-edge} and~\ref{L:isotopy-discs-fixed-bdry} use the following lemmata:

\begin{lemma} \label{L:isotopy-curves-1}
Let $ d \geqslant 4 $, let $ W \subset \mathbb{R}^d $ be an open subset, diffeomorphic to a ball, and let $ \gamma_0, \gamma_1 : [0,1] \rightarrow W $ be smooth embedded curves which coincide on a neighbourhood of the endpoints of $ [0,1] $. Then there exists a smooth homotopy $ F : [0,1] \times [0,1] \rightarrow W $ such that $ F(0,t) = \gamma_0(t) $, $ F(1,t) = \gamma_1(t) $ for $ t \in [0,1] $, such that for some $ 0 < \epsilon < 1/2 $ we have $ F(s,t) = \gamma_0(t) = \gamma_1(t) $ for every $ s \in [0,1] $ and $ t \in [0,\epsilon] \cup [1-\epsilon,1] $, and such that for every $ s \in [0,1] $, the curve $ [0,1] \rightarrow W $, $ t \mapsto F(s,t) $ is smoothly embedded.
\end{lemma} 
\begin{proof}
Without loss of generality we may assume that $ W $ is an open ball in $ \mathbb{R}^d $. The map $ F' : [0,1] \times [0,1] \rightarrow W $ defined by $ F'(s,t) = (1-s) \gamma_0(t) + s \gamma_1(t) $, $ (s,t) \in [0,1] \times [0,1] $, fulfils all the needed requirements, except may be for the requirement that for every $ s \in [0,1] $, the curve $ [0,1] \rightarrow W $, $ t \mapsto F'(s,t) $ is smoothly embedded. However, using standard arguments of general position, one can slightly perturb $ F' $ and obtain the needed smooth homotopy $ F $.
\end{proof}

\begin{lemma} \label{L:isotopy-curves-2} 
Let $ d \geqslant 5 $, let $ W \subset \mathbb{R}^d $ be an open subset, diffeomorphic to a ball, let $ \gamma : S^1 \rightarrow W $ be a smooth closed embedded curve, and let $ F : [0,1] \times S^1 \rightarrow W $ be a smooth map such that for every $ u \in [0,1] $, the curve $ S^1 \rightarrow W $, $ t \mapsto F(u,t) $ is smoothly embedded, and such that $ F(0,t) = F(1,t) = \gamma(t) $. Then there exists a smooth map $ \hat{F} : [0,1] \times [0,1] \times S^1 \rightarrow W $, such that we have $ \hat{F}(u,0,t) = \hat{F}(0,s,t) = \hat{F}(1,s,t) = \gamma(t) $, $ \hat{F}(u,1,t) = F(u,t) $ for every $ u,s \in [0,1] $ and $ t \in S^1 $, and such that for any $ u, s \in [0,1] $, the curve $ S^1 \rightarrow W $, $ t \mapsto F(u,s,t) $ is smoothly embedded.    
\end{lemma}
\begin{proof}
Without loss of generality we may assume that $ W $ is an open ball in $ \mathbb{R}^d $. The map $ \hat{F}' : [0,1] \times [0,1] {\times S^1}\rightarrow W $ defined by $ \hat{F}'(u,s,t) = s F(u,t) + (1-s) \gamma(t) $, $ (s,t) \in [0,1] \times [0,1] $, fulfils all the needed requirements, except may be for the requirement that for every $ u,s \in [0,1] $, the curve $ [0,1] \rightarrow W $, $ t \mapsto \hat{F}'(u,s,t) $ is smoothly embedded. However, using standard arguments of general position, one can slightly perturb $ \hat{F}' $ and obtain the needed smooth map $ \hat{F} $.
\end{proof}

\begin{lemma} \label{L:ham-isotopy-curves-fixed-end}
Let $ (M,\omega) $ be a symplectic manifold, and let $ F : [0,1] \times [0,1] \rightarrow M $ be a smooth map, such that for every $ s  \in [0,1] $, the curve $ [0,1] \rightarrow M $, $ t \mapsto F(s,t) $ is smoothly embedded, such that for some $ 0 < \epsilon < 1/2 $, we have $ F(s,t) = F(0,t) $ for $ (s,t) \in [0,1] \times ( [0,\epsilon] \cup [1-\epsilon,1]) $, and moreover such that for every $ s_0 \in [0,1] $ the symplectic area of the rectangle $ [0,s_0] \times [0,1] \rightarrow M $, $ (s,t) \mapsto F(s,t) $ is zero. Then there exists a Hamiltonian function $ H : M \times [0,1] \rightarrow \mathbb{R} $, $ H = H(x,s) $, with a Hamiltonian flow $ \psi_{H}^s $, $ s \in [0,1] $, such that we have $ \psi_{H}^s (F(0,t)) = F(s,t) $, and moreover on some neighbourhood of $ \{ F(0,0), F(0,1) \} $ we have $ H(\cdot,s) = 0 $ for every $ s \in [0,1] $.  
\end{lemma}
\begin{proof}
For every $ s \in [0,1] $, look at the vector field $ X_s(F(s,t)) = \frac{\partial}{\partial s} F(s,t) $, $ t \in [0,1] $, along the curve $ [0,1] \rightarrow M $, $ t \mapsto F(s,t) $. Denote by $ \alpha_s(F(s,t)) \in T^*_{F(s,t)} M $ by $ \alpha_s|_{F(s,t)}(\cdot) = \omega(X_s(F(s,t)), \cdot) $. Then $ \alpha_s $ is a section of $ T^*M $ above the curve $  [0,1] \rightarrow W $, $ t \mapsto F(s,t) $, we have that $ \alpha_s $ vanishes near the endpoints of the curve, and that $ \int_{0}^1 \alpha_s(\frac{\partial}{\partial t}F(s,t)) \, dt = 0 $. Hence it is easy to see that at least locally, near any given point $ s_0 \in [0,1] $, for some neighbourhood $ s_0 \in I \subset [0,1] $ of $ s_0 $ in $ [0,1] $, we can find a compactly supported  function $ H_I : W \times \overline{I} \rightarrow \mathbb{R} $, $ H_I = H_I(x,s) $, such that $ d_x H_I (F(s,t),s) = \alpha_s |_{F(s,t)} $ for every $ (s,t) \in \overline{I} \times [0,1] $, and such that $ H_I = 0 $ on a neighbourhood of $ \{ F(0,0), F(0,1) \} $. But then we can produce a ``global" Hamiltonian function $ H : M \times [0,1] \rightarrow \mathbb{R} $ via a partition of unity. That is, choose a covering of $ [0,1] $ by sufficiently short intervals $ \{ I_j \}_{j=1,\ldots,m} $ which are open in $ [0,1] $, and compactly supported  functions $ H_{I_j} : M \times \overline{I} \rightarrow \mathbb{R} $, $ H_{I_j} = H_{I_j}(x,s) $, such that $ d_{x}H_{I_j} (F(s,t),s) = \alpha_s |_{F(s,t)} $ for every $ (s,t) \in \overline{I_j} \times [0,1] $, and such that $ H_{I_j} = 0 $ on a neighbourhood of $ \{ F(0,0), F(0,1) \} $. Then choose a smooth partition of unity $ \kappa_j : [0,1] \rightarrow \mathbb{R} $, $ j = 1,\ldots,m $, such that $ \kappa_j $ is compactly supported in $ I_j $ and such that $ \sum_j \kappa_j = 1 $ on $ [0,1] $, and define $ H : M \times [0,1] \rightarrow \mathbb{R} $ by $ H(x,s) = \sum_j \kappa_j(s) H_{I_j} (x,s) $. Then $ d_x H (F(s,t),s) = \alpha_s |_{F(s,t)} $ for every $ (s,t) \in [0,1] \times [0,1] $, and $ H = 0 $ on a neighbourhood of $ F([0,1] \times \{ 0,1 \} ) $. This means that the Hamiltonian vector field of $ H $ at time $ s $, restricted to the curve $ [0,1] \rightarrow W $, $ t \mapsto F(s,t) $, coincides with $ X_s $, and therefore the Hamiltonian flow of $ H $ takes the curve $ [0,1] \rightarrow W $, $ t \mapsto F(0,t) $ exactly through the homotopy $ F $. 
\end{proof}

\begin{lemma} \label{L:ham-isotopy-circles}
Let $ (M,\omega) $ be a symplectic manifold, and let $ \hat{F} : [0,1] \times [0,1] \times S^1 \rightarrow M $ be a smooth map, such that for every $ u_0,s_0 \in [0,1] $, the curve $ S^1 \rightarrow M $, $ t \mapsto \hat{F}(u_0,s_0,t) $ is smoothly embedded, and moreover the symplectic area of the cylinder $ [0,s_0] \times S^1 \rightarrow M $, $ (s,t) \mapsto \hat{F}(u_0,s,t) $ is zero. Then there exists a smooth family of compactly supported Hamiltonian functions $ H_u : M \times [0,1] \rightarrow \mathbb{R} $, $ H_u = H_u(x,s) $, $ u \in [0,1] $ with Hamiltonian flows $ \psi_{H_u}^s $, $ s \in [0,1] $, such that we have $ \psi_{H_u}^s ( \hat{F}(u,0,t)) = \hat{F}(u,s,t) $, and moreover the following holds: if for some $ u_0,s_0 \in [0,1] $ we have $ \frac{\partial}{\partial s} \hat{F}(u_0,s_0,t) = 0 $ for any $ t \in S^1 $, then $ H_{u_0}(x,s_0) = 0 $ for every $ x \in M $.  
\end{lemma}

\begin{lemma} \label{L:ham-isotopy-curves}
Let $ (M,\omega) $ be a symplectic manifold, and let $ \hat{F} : [0,1] \times [0,1] \times [0,1] \rightarrow M $ be a smooth map, such that for every $ u,s \in [0,1] $, the curve $ [0,1] \rightarrow M $, $ t \mapsto \hat{F}(u,s,t) $ is smoothly embedded. Then there exists a smooth family $ \psi_{u,s} \in \Ham(M,\omega) $, $ u,s \in [0,1] $ of compactly supported Hamiltonian diffeomorphisms of $ M $, such that we have $ \psi_{u,s} (\hat{F}(u,0,t)) = \hat{F}(u,s,t) $.
\end{lemma}

\begin{lemma} \label{L:from-isotopy-to-flow-discs}
Let $ (M,\omega) $ be a symplectic manifold, let $ r > 0 $ and $ G = D(r) \subset \mathbb{C} $, and let $ v : [0,1] \times \overline{G} \rightarrow M $, $ v(t,z) = v_t(z) $ be a smooth isotopy of symplectic discs, i.e. $ v_t^* \omega = \omega_{\st} $ on $ G $ for every $ t \in [0,1] $. Then there exists a compactly supported smooth Hamiltonian flow $ \psi^t $, $ t \in [0,1] $ on $ M $ such that $ v_t = \psi^t \circ v_0 $ for every $ t \in [0,1] $.
\end{lemma}

\begin{lemma} \label{L:from-isotopy-to-flow-annuli}
Let $ (M,\omega) $ be a symplectic manifold, let $ a > 0 $ and let $ \Sigma = S^1 \times [0,a] $ be an annulus with the standard symplectic form $ \omega_{\st} $. Let $ \hat{v} : [0,1] \times [0,1] \times \Sigma \rightarrow M $, $ \hat{v}(u,s,z) = \hat{v}_{u,s}(z) $ be a smooth map such that for each $ u,s \in [0,1] $, the map $ \Sigma \rightarrow M $, $ z \mapsto \hat{v}_{u,s}(z) $ is a smooth symplectic embedding, $ \hat{v}_{u,s}^* \omega = \omega_{\st} $. Then there exists a smooth family of compactly supported Hamiltonian functions $ H_u : M \times [0,1] \rightarrow \mathbb{R} $, $ H_u = H_u(x,s) $, $ u \in [0,1] $ with Hamiltonian flows $ \psi_{H_u}^s $, $ s \in [0,1] $, such that we have $ \psi_{H_u}^s ( \hat{v}(u,0,z)) = \hat{v}(u,s,z) $, and moreover the following holds: if for some $ u_0,s_0 \in [0,1] $ we have $ \frac{\partial}{\partial s} \hat{v}(u_0,s_0,z) = 0 $ for any $ z \in \Sigma $, then $ H_{u_0}(x,s_0) = 0 $ for every $ x \in M $.  
\end{lemma}

The proofs of lemmata~\ref{L:ham-isotopy-circles},~\ref{L:ham-isotopy-curves},~\ref{L:from-isotopy-to-flow-discs} and~\ref{L:from-isotopy-to-flow-annuli} are quite similar to the proof of lemma~\ref{L:ham-isotopy-curves-fixed-end}: first one can find the corresponding Hamiltonian function locally and then use partition of unity to unify these ``local" Hamiltonian functions. Therefore we omit the proofs of lemmata~\ref{L:ham-isotopy-circles},~\ref{L:ham-isotopy-curves},~\ref{L:from-isotopy-to-flow-discs} and~\ref{L:from-isotopy-to-flow-annuli}. 

\subsection{Proofs of main lemmata}

Let us turn to the proofs of lemmata~\ref{L:isotopy-closed-discs},~\ref{L:isotopy-nbd-edge} and~\ref{L:isotopy-discs-fixed-bdry}.

\noindent{\it Proof of lemma~\ref{L:isotopy-closed-discs}:}
By the symplectic neighbourhood theorem, on a neighbourhood $ U_1 $ of $ v_1(\overline{G}) $ there exist local symplectic coordinates $ (x_1,y_1,\ldots,x_n,y_n) = (z_1,\ldots,z_n) $ such that in these coordinates we have $$ U_1 = \{ (x_1,y_1,\ldots,x_n,y_n) \; | \; x_1^2 + y_1^2 < (r + \epsilon)^2, x_j^2 + y_j^2 < \epsilon^2 \,\, \text{for} \,\, 2 \leqslant j \leqslant n \} ,$$ $$ v_1(x_1,y_1) = (x_1,y_1,0,0,\ldots,0,0) ,$$ and on a neighbourhood $ U_2 $ of $ v_2(\overline{G}) $ there exist local symplectic coordinates \\ $ (x_1,y_1,\ldots,x_n,y_n) $ (we loosely use the same notation for these coordinates as well) such that in these coordinates we have $$ U_2 = \{ (x_1,y_1,\ldots,x_n,y_n) \; | \; x_1^2 + y_1^2 < (r + \epsilon)^2, x_j^2 + y_j^2 < \epsilon^2 \,\, \text{for} \,\, 2 \leqslant j \leqslant n \} ,$$ $$ v_2(x_1,y_1) = (x_1,y_1,0,0,\ldots,0,0) ,$$ where $ \epsilon > 0 $ is small. In the sequel by $ p = (x_1,y_1,\ldots,x_n,y_n) \in U_1 $ (respectively, by $ p = (x_1,y_1,\ldots,x_n,y_n) \in U_2 $) we will always mean that $ p $ is written in the local coordinates of $ U_1 $ (respectively, that $ p $ is written in the local coordinates of $ U_2 $). It easily follows from the Moser's argument, that for some small $ 0 < \delta < \epsilon $, there exists a Hamiltonian diffeomorphism $ \psi' $ of $ M $, such that $ \psi'(v_1(0,0)) = v_2(0,0) $, and such that for any $ p = (x_1,y_1,\ldots,x_n,y_n) \in U_1 $ with $ x_j^2+y_j^2 < \delta^2 $, $ j=1,\ldots,n $, we have $ \psi'(p) = (x_1,y_1,\ldots,x_n,y_n) \in U_2 $. Now choose a smooth immersion $ f : \overline{G} \rightarrow D(\delta) $ with $ f^*\omega = \omega_{\st} $, and define two smooth families of discs $ v_{1,t}, v_{2,t} : \overline{G} \rightarrow M $, $ t \in [0,\pi / 2) $, by $ v_{1,t}(z) = (\cos(t) z, \sin(t) f(z),0,\ldots,0) \in U_1 $ and $ v_{2,t}(z) = (\cos(t) z, \sin(t) f(z),0,\ldots,0) \in U_2 $. We have $ v_{1,t}^* \omega = v_{2,t}^* \omega = \omega_{\st} $, $ t \in [0,\pi/2) $. Moreover we have $ v_{1,\pi/2 - \delta /r}(\overline{G}) \subset  \{ (x_1,y_1,\ldots,x_n,y_n) \in U_1 \; | \; x_j^2+y_j^2 < \delta^2 \,\, \text{for} \,\, 1 \leqslant j \leqslant n \} $, $ v_{2,\pi/2 - \delta /r}(\overline{G}) \subset \{ (x_1,y_1,\ldots,x_n,y_n) \in U_1 \; | \; x_j^2+y_j^2 < \delta^2 \,\, \text{for} \,\, 1 \leqslant j \leqslant n \} $, and hence $ \psi' \circ v_{1,\pi/2 - \delta /r} = v_{2,\pi/2 - \delta /r} $, so there exists a compactly supported Hamiltonian isotopy of $ M $ which takes $ v_{1,\pi/2-\delta /r} $ to $ v_{2,\pi/2-\delta /r} $. By lemma~\ref{L:from-isotopy-to-flow-discs}, there exists a compactly supported Hamiltonian isotopy of $ M $ which takes $ v_{1,0} = v_1 $ to $ v_{1,\pi/2 - \delta /r} $, and there exists another compactly supported Hamiltonian isotopy of $ M $ which takes $ v_{2,0} = v_2 $ to $ v_{2,\pi/2 - \delta /r} $. So, finally, the discs $ v_1 $ and $ v_2 $ are isotopic via a compactly supported Hamiltonian isotopy of $ M $.\cqfd

\noindent {\it Proof of lemma~\ref{L:isotopy-nbd-edge}:}
First of all, without loss of generality we may assume that $ \gamma(t) = (t,0) $ for $ t \in [0,1] $. We first take care of  the absolute situation (a).
We divide the construction of the Hamiltonian $ H $ into two steps. At the first step we find a Hamiltonian isotopy 
that brings the map $v_1$ to a map $v_1'$ which coincides with $v_2$ on $\gamma([0,1])$. The second step provides a Hamiltonian isotopy between $v_1$ and a map $v_1''$ which coincides with $v_2$ on a \nbd of $\gamma([0,1])$. We then explain how the proof of (a) readily implies the proper case (b). \\ \\
{\bf (a) Step I} \\
By lemma~\ref{L:isotopy-curves-1}, we can find a smooth homotopy $ F : [0,1] \times [0,1] \rightarrow W $ such that $ F(0,t) = v_1 \circ \gamma(t) $, $ F(1,t) = v_2 \circ \gamma(t) $ for $ t \in [0,1] $, such that for some $ 0 < \epsilon < 1/2 $ we have $ F(s,t) = v_1 \circ \gamma(t) = v_2 \circ \gamma(t) $ for every $ s \in [0,1] $ and $ t \in [0,\epsilon] \cup [1-\epsilon,1] $, and such that for every $ s \in [0,1] $, the curve $ [0,1] \rightarrow W $, $ t \mapsto F(s,t) $ is smoothly embedded. After making a $ C^0 $ small perturbation of $ F $ on $ [0,1]\times(\epsilon/3,\epsilon/2)$, we may assume that for all $ s \in [0,1] $, the $ \lambda $-actions of the curves $ [0,1] \rightarrow W $, $ t \mapsto F(s,t) $ are equal, and moreover that we still have that $ F(0,t) = v_1 \circ \gamma(t) $, $ F(1,t) = v_2 \circ \gamma(t) $ for $ t \in [0,1] $, and that the curve $ [0,1] \rightarrow W $, $ t \mapsto F(s,t) $ is smoothly embedded.

\begin{remark} \label{R:perturbation}
One possible way to do this is as follows. We have $ F(s,t) = v_1 \circ \gamma(t) = v_2 \circ \gamma(t) $ for $ t \in [0,\epsilon] $. Since $ v_1 \circ \gamma $ is smoothly embedded, one can find local symplectic coordinates $ (x_1,y_1, \ldots, x_n,y_n) $ on a neighbourhood $ U $ of $ v_1\circ \gamma ((0,\epsilon)) $, such that in these coordinates $ U = (0,\epsilon) \times (-\delta , \delta) \times D(\delta)^{\times n-1} $, and $ v_1 \circ \gamma (t) = (t,0,0,\ldots,0) $ for $ t \in (0,\epsilon) $, where $ \delta > 0 $ is a small positive number. Now fix a non-negative smooth function $ \kappa : (0,\epsilon) \rightarrow \mathbb{R} $, not identically $ 0 $, such that $ \supp (\kappa) \subset (\epsilon /3, \epsilon /2) $, and such that $ \| \kappa \|_{\infty} $ is small enough. Let $ \nu : [0,1] \rightarrow \mathbb{R} $ be a smooth function, which will be specified later. Now we define $ F' : [0,1] \times [0,1] \rightarrow M $ by $ F'(s,t) = F(s,t) $ for $ (s,t) \in [0,1] \times (\{0\} \cup [\epsilon,1]) $, and by $ F'(s,t) = (t,0,\kappa(t) \cos(\nu(s) t), \kappa(t) \sin(\nu(s) t), 0,0, \ldots ,0,0) $, for $ (s,t) \in [0,1] \times (0,\epsilon) $, where the latter equality is written in the chosen local coordinates on $ U $. We have $ \omega(F'|_{[0,s] \times [0,1]}) = \omega(F'|_{[0,s] \times [0,\epsilon]}) + \omega(F'|_{[0,s] \times [\epsilon,1]}) = \omega(F'|_{[0,s] \times [0,\epsilon]}) + \omega(F|_{[0,s] \times [0,1]}) = \frac{1}{2} (\nu(0) - \nu(s)) \int_{0}^{\epsilon} \kappa^2(t) dt +  \omega(F|_{[0,s] \times [0,1]}) $, where the latter equality follows from a simple computation. Hence if we define the function $ \nu $ to be $ \nu(s) = 2 \omega(F|_{[0,s] \times [0,1]})/ (\int_0^{\epsilon} \kappa^2(t) dt ) $, for $ s \in [0,1] $, then we get $ \omega( F'|_{[0,s] \times [0,1]}) = 0 $ for every $ s \in [0,1] $. Now replace $ F $ by $ F' $.
\end{remark}

Now, by lemma~\ref{L:ham-isotopy-curves-fixed-end} one can find a compactly supported Hamiltonian $ H' : W \times [0,1] \rightarrow \mathbb{R} $ such that on some neighbourhood of $ \{ F(0,0), F(0,1) \} $ we have $ H'(\cdot,s) = 0 $ for every $ s \in [0,1] $, and such that the time-1 map $ \psi' $ of $ H' $ satisfies $ \psi' \circ F(0,t) = F(1,t) $, for every $ t \in [0,1] $. Hence denoting $ v_1' := \psi' \circ v_1 $, we get that $ v_1' = v_2 $ on $ \gamma([0,1]) $ and moreover $ v_1' = v_2 $ on a neighbourhood of the endpoints $ \{ z_1,z_2 \} $. \\ \\
{\bf (a) Step II} \\
Recall that $ \gamma(t) = (t,0) $ for $ t \in [0,1] $, and denote by $ \hat{\gamma} $ the curve $ \hat{\gamma} : [0,1] \rightarrow W $, $ \hat{\gamma}(t) = v_1' \circ
\gamma (t) = v_2 \circ \gamma (t) $. Define $ X_0'(\hat{\gamma}(t)) = \frac{\partial}{\partial y} v_1'(t,0) $, $ X_1'(\hat{\gamma}(t)) = \frac{\partial}{\partial y} v_2(t,0) $ and $ Y(\hat{\gamma}(t)) = \frac{\partial}{\partial x} v_1'(t,0) = \frac{\partial}{\partial x} v_2(t,0) $ (here $ v_1' = v_1'(x,y) $, $ v_2 = v_2(x,y) $, where $ (x,y) \in \overline{V} \subset \mathbb{R}^2 $), so that $ X_0',X_1',Y $ are vector fields along the curve $ \hat{\gamma} $. Since $ v_1',v_2 $ are symplectic, we get that $ \omega_{\st}(Y,X_0') = \omega_{\st}(Y,X_1') = 1 $ at each point $ \hat{\gamma}(t) $, for $ t \in [0,1] $. Hence if we define $ X_s' = (1-s) X_0' + s X_1' $ for $ s \in [0,1] $, then $ \omega_{\st}(Y,X_s') = 1 $ at each point $ \hat{\gamma}(t) $, for $ t \in [0,1] $. Hence it is not hard to see that for some neighbourhood $ G' \supset \gamma([0,1]) $, $ G' \subset G $, such that $ G' $ is a topological disc bounded by a closed smooth simple curve, there exists a smooth family of smooth symplectic maps $ w_s : \overline{G'} \rightarrow W $, $ w_s^*\omega_{\st} = \omega_{\st} $, $ s \in [0,1] $, such that on the union of $ \gamma([0,1]) $ with a neighbourhood of $ \{ z_1,z_2 \} $ we have $ w_s = v_1' = v_2 $ for every $ s \in [0,1] $, such that $ w_0 = v_1' $ and $ w_1 = v_2 $ on a neighbourhood of $ \gamma([0,1]) $, and such that $ \frac{\partial}{\partial y} w_s(t,0) = X_s'(\hat{\gamma}(t)) $, $  \frac{\partial}{\partial x} w_s(t,0) = Y(\hat{\gamma}(t)) $, for each $ s,t \in [0,1] $. This family $w_s$ can be realized as follows: write $v_1'(t,y)=\hat{\gamma}(t)+yX_0'(\hat{\gamma}(t))+R'(t,y)$ and $v_2(t,y)=\hat{\gamma}(t)+yX_1'(\hat{\gamma}(t))+R_2(t,y)$, where $R'(t,y),R_2(t,y)\in O(|y|^2)$. Then $\tilde w_s(t,y):=\hat{\gamma}(t)+yX_s'(\hat{\gamma}(t))+sR_2(y,t)+(1-s)R'(y,t)$ interpolates between $v_1'$ and $v_2$. Moreover, the restriction of $w_s$ to a sufficiently small \nbd $G'$ of $\gamma$ is symplectic in the sense that $w_s^*\om_\st$ is an area form, but it does not verify yet that $w_s^*\om_\st=\om_\st$. This last point is achieved by a source reparametrization, and by changing $ G' $ to a smaller \nbd if necessary (applying a parametric version of the Moser's trick).

Now we apply lemma~\ref{L:from-isotopy-to-flow-discs} to conclude that there exists a compactly supported Hamiltonian function $ H'' : W \times [0,1] \rightarrow \mathbb{R} $ whose Hamiltonian flow isotopes $ w_0 $ to $ w_1 $ through the family $ w_{s} $, $ s \in [0,1] $. The points on the image by $ v_1' $ of the union of $ \gamma([0,1]) $ with a neighbourhood of $ \{ z_1,z_2 \} $, stay fixed under the flow of $ H'' $. This implies that $ H''(\cdot,s) $ is constant and its differential is zero on the image by $ v_1' $ of the union of $ \gamma([0,1]) $ with a neighbourhood of $ \{ z_1,z_2 \} $, at each time $ s \in [0,1] $. Without loss of generality we may assume that $ H''(\cdot,s) $ and its differential are zero on the image by $ v_1' $ of the union of $ \gamma([0,1]) $ with a neighbourhood of $ \{ z_1,z_2 \} $, at each time $ s \in [0,1] $ (if not, then we can achieve this by adding to $ H'' $ a function depending solely on time and then making a cutoff outside the union of the images of $ w_s $, $ s \in [0,1] $). But then if we multiply $ H'' $ by a function which equals $ 1 $ on the complement of a small neighbourhood of $ \{ v_1'(z_1),v_1'(z_2) \} $, and equals $ 0 $ on a smaller neighbourhood of $ \{ v_1'(z_1),v_1'(z_2) \} $, the time-$s$ map of the flow of the resulting Hamiltonian function $ \tilde{H}'' $ still maps $ w_0 $ to $ w_s $, but we moreover have that $ \tilde{H}'' = 0 $  on a neighbourhood of $ \{ v_1'(z_1),v_1'(z_2) \} = \{ v_1(z_1),v_1(z_2) \} = \{ v_2(z_1),v_2(z_2) \} $. Finally, the concatenation of the flows of $ H' $ and $ \tilde{H}'' $ give us the desired Hamiltonian flow.\\ \\

\noindent {\bf (b)} First, choose $ 0 < \epsilon < 1/2 $ such that $ v_1 = v_2 $ on a \nbd of $ \gamma((0,\epsilon] \cup [1-\epsilon,1)) $. Then, similarly as in the proof of the absolute case (a), by applying lemma~\ref{L:isotopy-curves-1} we can find a smooth homotopy $ F : [0,1] \times [\epsilon,1-\epsilon] \rightarrow W $, such that $ F(0,t) = v_1 \circ \gamma(t) $, $ F(1,t) = v_2 \circ \gamma(t) $ for $ t \in [\epsilon,1-\epsilon] $, such that for some $ 0 < \delta < 1/2 - \epsilon $ we have $ F(s,t) = v_1 \circ \gamma(t) = v_2 \circ \gamma (t) $ for $ s \in [0,1] $ and $ t \in [\epsilon,\epsilon+\delta] \cup [1-\epsilon-\delta,1-\epsilon] $, and such that for each $ s \in [0,1] $ the curve $ [\epsilon,1-\epsilon] \rightarrow W $, $ t \mapsto F(s,t) $ is smoothly embedded. Then, since we are in dimension $ 2n \geqslant 4 $, after slightly perturbing $ F $ we may assume that $ \im F \cap v_1 \circ \gamma ((0,\epsilon) \cup (1-\epsilon,1)) = \emptyset $. Moreover, similarly as in the proof of the absolute case (a), after performing one more $ \cc^0 $-small perturbation of $ F $ near $ t = \epsilon, 1-\epsilon $, we may in addition assume that the actions of curves $ [\epsilon,1-\epsilon] \rightarrow W $, $ t \mapsto F(s,t) $ are all equal, when $ s \in [0,1] $. Now, by the proof of the absolute case (a), we can find a compactly supported Hamiltonian function $ \tilde{H} : W \times [0,1] \rightarrow \mathbb{R} $, such that on a \nbd of $ \{ v_1 \circ \gamma(\epsilon), v_1 \circ \gamma(1-\epsilon) \} $ we have $ \tilde{H}(\cdot,s) = 0 $ for all $ s \in [0,1] $, and such that $ \phi_{\tilde{H}}^1 \circ v_1 = v_2 $ on a \nbd of $ \gamma([\epsilon,1-\epsilon]) $, where $ \phi_{\tilde{H}}^1 $ is the time-$1$ map of the Hamiltonian flow generated by $ \tilde{H} $. Choose sufficiently small neighbourhoods $ W' \Subset W'' \Subset W $ of $ \im F $, such that $ \tilde{H}(\cdot,s) = 0 $ on a \nbd of $ v_1 \circ \gamma ((0,\epsilon] \cup [1-\epsilon,1)) \cap W'' $, for all $ s \in [0,1] $. Then any $ (W',W'') $ cut-off of $ \tilde{H} $ gives a desired Hamiltonian function. \cqfd

\noindent{\it Proof of lemma~\ref{L:isotopy-discs-fixed-bdry}:} 
Choose a $1$-form $\lambda$ which is a primitive of $\om$ on $ W $, $ d\lambda = \omega $. We divide the proof into two steps. In the first step we find a Hamiltonian flow that takes the disc $ v_1 $ to the disc $ v_2 $ while the boundary of the disc is being kept fixed. In the second step we correct the Hamiltonian flow from the first step so that now a neighbourhood of the boundary is being kept fixed during the isotopy. \\ \\
{\bf Step I} \\
By lemma~\ref{L:isotopy-closed-discs}, there exists a compactly supported Hamiltonian isotopy $ \phi^u $, $ u \in [0,1] $ of $ W $, such that $ \phi^1 \circ v_1 = v_2 $. Define a map $ F : [0,1] \times S^1 \rightarrow W $ by $ F(u,t) = \phi^u(v_1(t)) $, while here we identify $ S^1 \cong \partial G $. Denote $ \gamma(t) :=  F(0,t) = F(1,t) $ for every $ t \in S^1 $. Our aim is to find a smooth path $ \Phi^{u} $, $ u \in [0,1] $ of compactly supported Hamiltonian diffeomorphisms of $ W $ such that $ \Phi^u ( \gamma(t)) = F(u,t) $ for every $ u \in [0,1] $, $ t \in S^1 $, such that we moreover have $ \Phi^0 = \Phi^1 = id $. Provided that we have such $ \Phi^u $, $ u \in [0,1] $, we can define the Hamiltonian flow $ \psi_1^u = (\Phi^u)^{-1} \circ \phi^u $, and then we have $ \psi_1^1 \circ v_1 = v_2 $ and $ \psi_1^u (v_1(z)) = v_1(z) = v_2(z) $ for every $ u \in [0,1] $ and $ z \in \partial G $. 

First, use lemma~\ref{L:isotopy-curves-2} to find a smooth map $ \hat{F} : [0,1] \times [0,1] \times S^1 \rightarrow W $ such that we have $ \hat{F}(u,0,t) = \hat{F}(0,s,t) = \hat{F}(1,s,t) = \gamma(t) $, $ \hat{F}(u,1,t) = F(u,t) $ for every $ u,s \in [0,1] $ and $ t \in S^1 $, and such that for any $ u, s \in [0,1] $, the curve $ S^1 \rightarrow W $, $ t \mapsto F(u,s,t) $ is smoothly embedded. After slightly perturbing $ \hat{F} $, we may assume that we moreover have that for every $ u,s \in [0,1] $, the $ \lambda $-action of the curve $ S^1 \rightarrow W $, $ t \mapsto \hat{F}(u,s,t) $, equals  the $ \lambda $-action of $ \gamma $, or in other words that for every $ u_0,s_0 \in [0,1] $, the symplectic area of the cylinder $ [0,s_0] \times S^1 \rightarrow W $, $ (s,t) \mapsto \hat{F}(u_0,s,t) $ equals zero.  

\begin{remark}
One possible way to obtain such a perturbation is as follows. Choose a closed arc $ I \subset S^1 $, and apply lemma~\ref{L:ham-isotopy-curves} to the restriction of $ \hat{F} |_{ [0,1] \times [0,1] \times I } $. We obtain a smooth family $ \psi_{u,s} \in \Ham(W,\omega) $, $ u,s \in [0,1] $ of compactly supported Hamiltonian diffeomorphisms of $ W $, such that we have $ \psi_{u,s}(\gamma(t)) = \psi_{u,s}(\hat{F}(u,0,t)) = \hat{F}(u,s,t) $ for every $ (u,s,t) \in [0,1] \times [0,1] \times I $. If we define $ \hat{F}' : [0,1] \times [0,1] \times I \rightarrow W $ by $ \hat{F}' (u,s,t) = \psi_{u,s}^{-1} \circ \hat{F}(u,s,t) $, then we have $ \hat{F}'(u,s,t) = \gamma(t) $ for every $ (u,s,t) \in [0,1] \times [0,1] \times I $. But now, using the fact that all the curves $ S^1 \rightarrow W $, $ t \mapsto \hat{F}'(u,s,t) $ coincide with the curve $ \gamma $ on $ I $, we can perturb $ \hat{F}' $ similarly to a perturbation scheme which was done in Remark~\ref{R:perturbation}, and obtain a new map $ \hat{F}'' : [0,1] \times [0,1] \times S^1 \rightarrow W $ such that for every $ u,s \in [0,1] $, the curve $ S^1 \rightarrow W $, $ t \mapsto \hat{F}''(u,s,t) $ is smoothly embedded and has the same $ \lambda $-action as $ \gamma $, and moreover such that for any $ u,s \in [0,1] $ for which the $ \lambda $-action of the curve $ S^1 \rightarrow W $, $ t \mapsto \hat{F}'(u,s,t) $ already equals  the $ \lambda $-action of $ \gamma $, we have $ \hat{F}''(u,s,t) = \hat{F}'(u,s,t) $ for all $ t \in [0,1] $. Then the map $ [0,1] \times [0,1] \times S^1 \rightarrow W $, $ (u,s,t) \mapsto \psi_{u,s} \circ \hat{F}''(u,s,t) $ is a desired perturbation of $ \hat{F} $.
\end{remark}

Now we apply lemma~\ref{L:ham-isotopy-circles} to obtain a smooth family of Hamiltonian functions $ H_u : W \times [0,1] \rightarrow \mathbb{R} $, $ H_u = H_u(x,s) $, $ u \in [0,1] $ with Hamiltonian flows $ \psi_{H_u}^s $, $ s \in [0,1] $, such that we have $ \psi_{H_u}^s ( \hat{F}(u,0,t)) = \hat{F}(u,s,t) $, and moreover the following holds: if for some $ u_0,s_0 \in [0,1] $ we have $ \frac{\partial}{\partial s} \hat{F}(u_0,s_0,t) = 0 $ for any $ t \in S^1 $, then $ H_{u_0}(x,s_0) = 0 $ for every $ x \in M $. But this implies that $ \psi_{H_0}^1 = \psi_{H_1}^1 = id $, and if we set $ \Phi^u := \psi_{H_u}^1 $, and define the Hamiltonian flow $ \psi_1^u = (\Phi^u)^{-1} \circ \phi^u $, then we get $ \psi_1^1 \circ v_1 = v_2 $ on $ \overline{G} $, and $ \psi_1^u (v_1(z)) = v_1(z) = v_2(z) $ for every $ u \in [0,1] $ and $ z \in \partial G $. \\ \\
{\bf Step II} \\
This step is analogical to the step II in the proof of lemma~\ref{L:isotopy-nbd-edge}. Consider the coordinates $ \rho \in [0,r) $, $ t \in S^1 \cong \mathbb{R} / 2\pi \mathbb{Z} $ on $ \overline{G} \setminus \{ 0 \} $ such that for $ z \in \overline{G} \setminus \{ 0 \} $ we have $ z = (r-\rho)t $, and look at $ v_1 = v_1(\rho,t) $, $ v_2 = v_2(\rho,t) $. Define vector fields $ X_{u,0}(\gamma(t)) = \frac{\partial}{\partial \rho} \psi_1^u \circ v_1(0,t) $, $ X_{u,1}(\gamma(t)) = \frac{\partial}{\partial \rho} v_1(0,t) = \frac{\partial}{\partial \rho} v_2(0,t) $ and $ Y(\gamma(t)) = \frac{\partial}{\partial t} v_1(0,t) = \frac{\partial}{\partial t} v_2(0,t) $, so that $ X_{u,0},X_{u,1},Y $ are vector fields along the curve $ \gamma $. Since $ \psi_1^u \circ v_1,v_2 $ are symplectic, we get that $ \omega_{\st}(Y,X_{u,0}) = \omega_{\st}(Y,X_{u,1}) = 1 $ at each point $ \gamma(t) $ for $ t \in S^1 $. If we define $ X_{u,s} = (1-s) X_{u,0} + s X_{u,1} $ for $ s \in [0,1] $, then $ \omega_{\st}(Y,X_{u,s}) = 1 $ at each point $ \gamma(t) $ for $ t \in S^1 $, and moreover $ X_{0,s} = X_{1,s} = X_{0,0} $ for every $ s \in [0,1] $. Hence it is not hard to see that for some neighbourhood $ G' $ of $ \partial G $, $ G' \subset \overline{G} $, such that $ G' $ is a topological annulus bounded by a smooth simple curve in $ G $ and $ \partial G $, there exists a smooth family of smooth symplectic embeddings $ \hat{v}_{u,s} : \overline{G'} \rightarrow W $, $ \hat{v}_{u,s}(z) = \hat{v}(u,s,z) $, $ \hat{v}_{u,s}^*\omega_{\st} = \omega_{\st} $, $ u, s \in [0,1] $, such that $ \hat{v}_{u,s}(0,t) = \gamma(t) $ for all $ t \in S^1 $, $ u,s \in [0,1] $, such that $ \hat{v}_{0,s} = \hat{v}_{1,s} = v_1 = v_2 $ on $ \overline{G'} $ for each $ s \in [0,1] $, and such that $ \frac{\partial}{\partial \rho} \hat{v}_{u,s}(0,t) = X_{u,s}(\gamma(t)) $ for each $ t \in S^1 $. Now applying lemma~\ref{L:from-isotopy-to-flow-annuli}, we find a smooth family of compactly supported Hamiltonian functions $ H_u : W \times [0,1] \rightarrow \mathbb{R} $, $ H_u = H_u(x,s) $, $ u \in [0,1] $ with Hamiltonian flows $ \psi_{H_u}^s $, $ s \in [0,1] $, such that we have $ \psi_{H_u}^s ( \hat{v}(u,0,z)) = \hat{v}(u,s,z) $, and moreover the following holds: if for some $ u_0,s_0 \in [0,1] $ we have $ \frac{\partial}{\partial s} \hat{v}(u_0,s_0,z) = 0 $ for any $ z \in \overline{V'} $, then $ H_{u_0}(x,s_0) = 0 $ for every $ x \in W $. If we now define the Hamiltonian flow $ \psi_2^u : W \rightarrow W $ by $ \psi_2^u = \psi_{H_u}^1 \circ \psi_1^u $, then we get that $ \psi_2^1 \circ v_1 = v_2 $ on $ \overline{G} $, and $ \psi_2^u (v_1(z)) = v_1(z) = v_2(z) $ for every $ u \in [0,1] $ and $ z \in \overline{G'} $. Let $ H_2 : W \times [0,1] \rightarrow \mathbb{R} $, $ H_2 = H_2(x,u) $ be the compactly supported Hamiltonian function of the flow $ \psi_2^u $. Since the points of $ v_1(\overline{G'}) $ stay fixed under the flow $ \psi_2^u $, the differential of $ H_2(\cdot,u) $ vanishes at each point of $ v_1(\overline{G'}) $, in particular $ H_2(\cdot,u) $ is constant on $ v_1(\overline{G'}) $, for each $ u \in [0,1] $. After adding to $ H_2 $ a function depending solely on time, and making a cutoff outside $ \cup_{t \in [0,1]} \psi_2^t(v_1(\overline{G})) $, if necessary, we may without loss of generality moreover assume that $ H_2(\cdot,u) = 0 $ on $ v_1(\overline{G'}) $, for each $ u \in [0,1] $. Finally, if we make a cutoff of the Hamiltonian function $ H_2 $, multiplying it by a function on $ W $ which equals $ 1 $ outside a small neighbourhood of $ u_1(\partial G) $ and equals $ 0 $ on a smaller neighbourhood of $ u_1(\partial G) $, then we obtain a Hamiltonian function $ H : W \times [0,1] \rightarrow \mathbb{R} $, $ H = H(x,u) $ with a Hamiltonian flow $ \psi^u $, $ u \in [0,1] $, such that on a neighbourhood of $ v_1(\partial G) $ we have $ H(\cdot,u) = 0 $ for every $ u \in [0,1] $, and such that for the time-1 map $ \psi^1 $ of the flow of $ H $, we have $ \psi^1 \circ v_1 = v_2 $ on $ \overline{G} $. \cqfd

\section{Relations between questions \ldots}\label{sec:relquest}

\begin{lemma} \label{L:sympl-homogen-generic}
Every smooth submanifold of a symplectic manifold has a dense relatively open subset which is a union of symplectically homogeneous relatively open subsets. 
\end{lemma}
\noindent{\it Proof:}
Let $ X \subset M $ be a smooth submanifold of a symplectic manifold $ (M,\omega) $. It is enough to show that for every relatively open $ U \subset X $ there exists relatively open symplectically homogeneous subset $ V \subset U $. To show the latter, choose a point $ x_0 \in U $ such that the dimension of the kernel of $ \omega $ in $ T_{x_0} X $ is minimal. Since for $ x \in U $, the dimension of the kernel of $ \omega $ in $ T_x X $ is an upper semi-continuous function of $ x $, it is enough to take $ V $ to be a small relatively open neighbourhood of $ x_0 $ in $ U $. 
\cqfd

\begin{lemma} \label{L:quesA-implies-quesB}
We have an affirmative answer to question \ref{q:Eli-Grom-submanifolds} for given $ X \subset M $ provided that for each relatively open
$ Y \subset X \times S^2 $ we have a positive answer to question \ref{q:c0rigid'}  for $ Y $ inside $ M \times S^2 $. In particular, the same is true for question \ref{q:c0rigid} instead of question \ref{q:c0rigid'}.
\end{lemma}
\noindent{\it Proof:}
First of all, it is enough to show that for given $ X \subset M $, if for every relatively open $ Y \subset X $ we have a positive answer to question \ref{q:c0rigid'} for $ Y $ inside $ M $, then for any symplectic homeomorphism $ h : M \rightarrow M' $ such that $ X' = h(X) \subset M' $ is smooth submanifold and such that the restriction $ h_{|X} : X \rightarrow X' $ is a diffeomorphism, the restriction is either symplectic or antisymplectic. Indeed, assume that we have shown this, and let $ X \subset M $ satisfy the assumptions of the lemma. Let $ h : M \rightarrow M' $ be a symplectic homeomorphism, such that $ h(X) = X' \subset M' $ is a smooth submanifold and such that $ h_{|X} : X \rightarrow X' $ is a diffeomorphism. Define the symplectic homeomorphism \fonction{\hat{h}}{ M \times S^2 }{M' \times S^2}{(x,z)}{(h(x),z).} Since the assumptions of the lemma are satisfied, $ \hat{h}(X \times S^2) = X' \times S^2 \subset M' \times S^2 $ is a smooth submanifold, and the restriction $ \hat{h}_{|X \times S^2} : X \times S^2 \rightarrow X' \times S^2 $ is a diffeomorphism, we conclude that the restriction is either symplectic or antisymplectic. However, since $ \hat{h}_{|X \times S^2} $ acts as the identity on the $ S^2 $-factor, the only possibility is that $ \hat{h}_{|X \times S^2} $ is symplectic. Therefore $ h_{|X} : X \rightarrow X' $ is symplectic.

By lemma~\ref{L:sympl-homogen-generic}, $ X $ has a dense open subset which is a union of symplectically homogeneous relatively open subsets. Since it is enough to show that $ h $ is either symplectic or antisymplectic on this dense open subset, we may assume without loss of generality that $ X $ is itself a symplectically homogeneous submanifold. By our assumptions, $ X $ and $ X' $ are symplectomorphic, hence $ X' $ is also symplectically homogeneous, having the same dimension and symplectic co-rank as $ X $. Now, for given $ x_0 \in X $ and for $ x_0' = h(x_0) \in X' $ we need to show that the differential of $ h_{|X} : X \rightarrow X' $ at $ x_0 $ is either symplectic or antisymplectic. There is a symplectic embedding $ i : B \rightarrow M $ of a small ball $ B \subset \mathbb{R}^{2n} $ centred at the origin, and a symplectic embedding $ i' : B' \rightarrow M $ of a small ball $ B' \subset \mathbb{R}^{2n} $ centred at the origin, such that $ i(0) = x_0 $, $ i'(0) = x_0' $, and such that $ i^{-1}(X) = B \cap L $ and $ i'^{-1}(X') = B' \cap L $, where $ L =  \{ (x_1,y_1,\ldots,x_n,y_n) \; | \; x_{m+1} = \ldots = x_{n} = y_{m+r+1} = \ldots = y_n = 0 \} \subset \mathbb{R}^{2n} $. Let us identify $ B $ with $ i(B) $ via $ i $ and identify $ B' $ and $ i'(B') $ via $ i' $. Denote the differential of $ h : X \rightarrow X' $ at $ x_0 $ by $ A $, and we may think of $ A $ as a linear map $ A : L \rightarrow L $. Now fix some smooth strictly convex bounded open set $ K \subset L $. Take $ \epsilon > 0 $ small enough, and consider $ Y = \epsilon K \subset X $. By our assumption, $ Y $ is symplectomorphic to $ Y' = h(Y) \subset X' $, i.e. there exists a symplectic diffeomorphism $ f : Y \rightarrow Y' $. Moreover, $ Y $ is strictly convex, and since $ \epsilon $ is small and $ h_{|X} : X \rightarrow X' $ is a diffeomorphsim, it follows that the intersection of each of the leaves of the characteristic foliation of $ X $ with $ Y $ is connected, and that the intersection of each of the leaves of the characteristic foliation of $ X' $ with $ Y' $ is connected. Hence it follows that $ f $ on $ Y $ has the form $ f(z,t) = (\phi(z),\psi(z,t)) $, where $ z = (x_1,y_1,\ldots,x_m,y_m) $, $ t = (y_{m+1},\ldots,y_{m+r}) $, and $ \phi $ is a symplectomorphism from $ \pi_z (Y) $ onto $ \pi_z(Y') $, where $ \pi_z : L \rightarrow \mathbb{R}^{2m} \times \{ 0_{r} \} $ is the orthogonal projection. In particular, $ \pi_z(Y) $ and $ \pi_z(Y') $ are symplectomorphic. If we choose any symplectic capacity $ c $, it follows that 
$ c(\pi_z(Y)) = c(\pi_z(Y')) $. But we have $ \lim_{\epsilon \rightarrow 0} \frac{c(\pi_z(Y))}{\epsilon^2} = c(\pi_z(K)) $ and $ \lim_{\epsilon \rightarrow 0} \frac{c(\pi_z(Y'))}{\epsilon^2} = c(\pi_z(A(K))) $. Hence it follows that $ c(\pi_z(K)) = c(\pi_z(A(K))) $ for any strictly convex bounded open body with smooth boundary $ K \subset L $. By continuity, we conclude that $ c(\pi_z(K)) = c(\pi_z(A(K))) $ for any convex bounded open body $ K \subset L $. From here it is easy to show that $ A : L \rightarrow L $ is either a symplectic or an antisymplectic linear map, and we leave this to the reader.\cqfd

\begin{lemma} \label{L:quesD'-implies-quesB}
For given $ X \subset M $, a positive answer to question \ref{q:subsympeo'} implies positive answer to question \ref{q:Eli-Grom-submanifolds}. In particular, for given $ X \subset M $, a positive answer to question \ref{q:subsympeo} implies positive answer to question \ref{q:Eli-Grom-submanifolds}.
\end{lemma}
\noindent{\it Proof:}
It is clearly enough to show the following generalisation of the Eliashberg-Gromov theorem: if $ X \subset M $, $ X' \subset M' $ are smooth submanifolds, and $ f_k : X \rightarrow X' $, $ k = 1,2, \ldots $ is a sequence of symplectic diffeomorphisms, that $ \mathcal{C}^0 $-converge to a diffeomorphism $ f : X \rightarrow X' $, then $ f $ is symplectic. Let us show this. By lemma~\ref{L:sympl-homogen-generic}, $ X $ and $ X' $ have dense subsets, each of which is a union of symplectically homogeneous relatively open subsets. Hence it is enough to show that if $ U \subset X $, $ U'  \subset X' $ are relatively open connected symplectically homogeneous subsets, such that $ f(U) \Subset U' $, then the restriction $ f_{|U} : U \rightarrow U' $ is a symplectic embedding. We have $ f_k(U) \Subset U' $ for large $ k $, hence $ U $ and $ U' $ has the same symplectic co-rank. It is enough to show that for given $ x_0 \in U $, $ f $ is symplectic near $ x_0 $. Denote $ x_0'= f(x_0) $. There is a symplectic embedding $ i : B \rightarrow M $ of a small ball $ B \subset \mathbb{R}^{2n} $ centred at the origin, and a symplectic embedding $ i' : B' \rightarrow M $ of a small ball $ B' \subset \mathbb{R}^{2n} $ centred at the origin, such that $ i(0) = x_0 $, $ i'(0) = x_0' $, and such that $ i^{-1}(U) = B \cap L $ and $ i'^{-1}(U') = B' \cap L $, where $ L =  \{ (x_1,y_1,\ldots,x_n,y_n) \; | \; x_{m+1} = \ldots = x_{n} = y_{m+r+1} = \ldots = y_n = 0 \} \subset \mathbb{R}^{2n} $. Let us identify $ B $ with $ i(B) $ via $ i $ and identify $ B' $ and $ i'(B') $ via $ i' $. Since $ f_k $ are symplectic, on some small neighbourhood $ 0 \in V \subset L $ they have the form $ f_k(z,t) = (\phi_k(z),\psi_k(z,t)) $, where $ z = (x_1,y_1,\ldots,x_m,y_m) $, $ t = (y_{m+1},\ldots,y_{m+r}) $, and $ \phi_k $ are symplectic embeddings from $ \pi_z (V) $ into $ \mathbb{R}^{2m} $, where $ \pi_z : L \rightarrow \mathbb{R}^{2m} \times \{ 0_{r} \} $ is the orthogonal projection. Since the sequence $ f_k $ $ \mathcal{C}^0 $-converges to $ f $, it follows that $ f $ on $ V $ also has the form $ f(z,t) = (\phi(z),\psi(z,t)) $, where the sequence of embeddings $ \phi_k $ $ \mathcal{C}^0 $ converges to $ \phi $. Now the statement follows from the proof of the Eliashberg-Gromov theorem.\cqfd

{\footnotesize
\bibliographystyle{alpha}
\bibliography{biblio.bib}
}


\bigskip
\noindent Lev Buhovski\\
School of Mathematical Sciences, Tel Aviv University \\
{\it e-mail}: levbuh@post.tau.ac.il
\bigskip

\bigskip
\noindent Emmanuel Opshtein\\
Institut de Recherche Math\'{e}matique Avanc\'{e}e \\
UMR  7501, Universit\'{e} de Strasbourg et CNRS \\
{\it e-mail}: opshtein@math.unistra.fr
\bigskip

\end{document}

%% file: bazar.pstex_t
\begin{picture}(0,0)%
\epsfig{file=bazar.pstex}%
\end{picture}%
\setlength{\unitlength}{1243sp}%
\begingroup\makeatletter\ifx\SetFigFont\undefined%
\gdef\SetFigFont#1#2#3#4#5{%
  \reset@font\fontsize{#1}{#2pt}%
  \fontfamily{#3}\fontseries{#4}\fontshape{#5}%
  \selectfont}%
\fi\endgroup%
\begin{picture}(19912,5739)(-2271,-11746)
\put(11573,-8130){\makebox(0,0)[lb]{\smash{{\SetFigFont{5}{6.0}{\rmdefault}{\mddefault}{\updefault}{\color[rgb]{0,0,0}$\psi_1'\big(W_{k_1}(1 / 2^{l_1})\big)$}%
}}}}
\put(10430,-8947){\makebox(0,0)[lb]{\smash{{\SetFigFont{5}{6.0}{\rmdefault}{\mddefault}{\updefault}{\color[rgb]{0,0,0}$\psi_1''$}%
}}}}
\put(16431,-7733){\makebox(0,0)[lb]{\smash{{\SetFigFont{5}{6.0}{\rmdefault}{\mddefault}{\updefault}{\color[rgb]{0,0,0}$\psi_1'(i_{k_1}(D))$}%
}}}}
\put(6526,-7216){\makebox(0,0)[lb]{\smash{{\SetFigFont{5}{6.0}{\rmdefault}{\mddefault}{\updefault}{\color[rgb]{0,0,0}$\psi_1'$}%
}}}}
\put(991,-6586){\makebox(0,0)[lb]{\smash{{\SetFigFont{5}{6.0}{\rmdefault}{\mddefault}{\updefault}{\color[rgb]{0,0,0} }%
}}}}
\put(5037,-10162){\makebox(0,0)[lb]{\smash{{\SetFigFont{5}{6.0}{\rmdefault}{\mddefault}{\updefault}{\color[rgb]{0,0,0}$i_{k_1,l_1}^a(D)$}%
}}}}
\put(5037,-10733){\makebox(0,0)[lb]{\smash{{\SetFigFont{5}{6.0}{\rmdefault}{\mddefault}{\updefault}{\color[rgb]{0,0,0}$i_{k_1}(D)$}%
}}}}
\put(-535,-6661){\makebox(0,0)[lb]{\smash{{\SetFigFont{5}{6.0}{\rmdefault}{\mddefault}{\updefault}{\color[rgb]{0,0,0}$W^a\big(\nf 1{2^{k_1}}\big)$}%
}}}}
\put(-249,-11091){\makebox(0,0)[lb]{\smash{{\SetFigFont{5}{6.0}{\rmdefault}{\mddefault}{\updefault}{\color[rgb]{0,0,0}$W_{k_1}\big(1 / 2^{l_1}\big)$}%
}}}}
\put(-2249,-8376){\makebox(0,0)[lb]{\smash{{\SetFigFont{5}{6.0}{\rmdefault}{\mddefault}{\updefault}{\color[rgb]{0,0,0}$\psi_1'$}%
}}}}
\put(16426,-6856){\makebox(0,0)[lb]{\smash{{\SetFigFont{5}{6.0}{\rmdefault}{\mddefault}{\updefault}{\color[rgb]{0,0,0}$i_{k_2}(D)$}%
}}}}
\put(16426,-10276){\makebox(0,0)[lb]{\smash{{\SetFigFont{5}{6.0}{\rmdefault}{\mddefault}{\updefault}{\color[rgb]{0,0,0}$i_{k_2}(D)=\psi_1''(\psi_1'(i_{k_1}(D)))$}%
}}}}
\put(16426,-10771){\makebox(0,0)[lb]{\smash{{\SetFigFont{5}{6.0}{\rmdefault}{\mddefault}{\updefault}{\color[rgb]{0,0,0}$i^a(D)$}%
}}}}
\put(16426,-7351){\makebox(0,0)[lb]{\smash{{\SetFigFont{5}{6.0}{\rmdefault}{\mddefault}{\updefault}{\color[rgb]{0,0,0}$i^a(D)=\psi'_1(i^a_{k_1,l_1}(D))$}%
}}}}
\put(11161,-6451){\makebox(0,0)[lb]{\smash{{\SetFigFont{5}{6.0}{\rmdefault}{\mddefault}{\updefault}{\color[rgb]{0,0,0}$\psi_1''$}%
}}}}
\put(5037,-8805){\makebox(0,0)[lb]{\smash{{\SetFigFont{5}{6.0}{\rmdefault}{\mddefault}{\updefault}{\color[rgb]{0,0,0}$i^a(D)$}%
}}}}
\end{picture}%